\documentclass{article}

\usepackage[
  backend=biber,
  style=numeric,
  sorting=none,
  giveninits=true,
  maxbibnames=99,
  doi=false,
  isbn=false
]{biblatex}
\renewbibmacro{in:}{}
\DeclareNameAlias{default}{family-given}
\DeclareFieldFormat[article,inbook,incollection,inproceedings,patent,thesis,unpublished]{citetitle}{#1}
\DeclareFieldFormat[article,inbook,incollection,inproceedings,patent,thesis,unpublished]{title}{#1} 
\DeclareFieldFormat{url}{\url{#1}}

\usepackage{arxiv}

\usepackage[utf8]{inputenc} 
\usepackage[T1]{fontenc}    
\usepackage[hidelinks]{hyperref}       
\usepackage{url}            
\usepackage{xfrac}       

\usepackage{graphicx}%
\usepackage{multirow}%
\usepackage{amsmath,amssymb,amsfonts}%
\usepackage{amsthm}%
\usepackage{mathrsfs}%
\usepackage[title]{appendix}%
\usepackage{xcolor}%
\usepackage{textcomp}%
\usepackage{manyfoot}%
\usepackage{booktabs}%
\usepackage{algorithm}%
\usepackage{algorithmicx}%
\usepackage{algpseudocode}%
\usepackage{listings}%
\usepackage{enumitem}

\usepackage{subcaption}
\usepackage{tikz}
\usepackage{cleveref}
	\crefname{equation}{}{}  
	\Crefname{equation}{}{}
\usepackage{multicol}
\usepackage{caption}
\usepackage{subcaption}
\usepackage{fontawesome5}
\usepackage{mathtools}
\usepackage{bm}
\usepackage{xspace}
\usepackage[dvipsnames]{xcolor}
\usepackage{enumitem}
\usepackage{fdsymbol}

\allowdisplaybreaks

\usepackage{comment}
\usepackage{rotating}
\usepackage{thm-restate}
\usepackage{type1cm}
\usepackage{anyfontsize}

\newcounter{panel}[figure]

\usepackage{caption}
\usepackage{etoolbox}
\newcommand{\tablecontentfont}{\fontsize{9.4}{10.7}\selectfont}
\AtBeginEnvironment{tabular}{\tablecontentfont}
\AtBeginEnvironment{tabular*}{\tablecontentfont}

\usetikzlibrary{trees}
\usetikzlibrary{patterns.meta}
\usetikzlibrary{arrows.meta}
\usetikzlibrary{shapes.geometric}
\usetikzlibrary{positioning,calc}

\colorlet{valuecolor}{Gray}
\definecolor{mycolor1}{RGB}{33,113,181} 
\definecolor{mycolor2}{RGB}{250,150,100}  
\definecolor{revision}{RGB}{228,26,28}  
\colorlet{subtree1}{mycolor1}
\colorlet{subtree2}{mycolor2}
\tikzset{
    general node/.style={circle, draw, minimum size=0.75cm, scale=0.9, align=center},
    small node/.style={circle, draw, inner sep=1.5pt},
    invisible node/.style={circle, inner sep=1.5pt},
    value-1 node/.style={fill=valuecolor},
    value-half node/.style={pattern={Lines[angle=45, distance=1.6mm, line width=0.6mm]}, pattern color=valuecolor},
    value-23rd node/.style={pattern={Lines[angle=45, distance=1mm, line width=0.6mm]}, pattern color=valuecolor},
    revised node/.style={label={[yshift=-0.8em, xshift=-1.0em] \textcolor{revision}{\scalebox{0.75}{\faStar}}}},
    small revised node/.style={label={[yshift=-0.5em, xshift=-0.5em]\textcolor{revision}{\scalebox{0.5}{\faStar}}}},
    half-revised node/.style={label={[yshift=-0.8em, xshift=-1.0em] \textcolor{revision}{\scalebox{0.75}{\faStarHalf*}}}},
    level-1 node/.style={draw=mycolor1, line width=1.5pt},
    level-2 node/.style={draw=mycolor2, line width=1.5pt},
    level 1/.style={level distance=30mm, sibling distance=30mm},
    level 2/.style={level distance=30mm, sibling distance=15mm},
    level 3/.style={level distance=30mm, sibling distance=13mm},
    subtree1 node/.style={rectangle, minimum size=1.5mm, fill=subtree1},
    subtree2 node/.style={diamond, fill=subtree2},
    spade node/.style={rectangle, draw, fill=subtree1, inner sep=3pt}, 
    club node/.style={rectangle, draw, inner sep=3pt}, 
    heart node/.style={diamond, draw, fill=subtree1, inner sep=2pt}, 
    diamond node/.style={diamond, draw, inner sep=2pt},
    sunny node/.style={circle, draw, fill=subtree2, inner sep=2.3pt},
    snowy node/.style={circle, draw, inner sep=2.3pt}
}

\theoremstyle{plain}%
\newtheorem{theorem}{Theorem}[section]
\newtheorem{proposition}[theorem]{Proposition}%
\newtheorem{corollary}[theorem]{Corollary}%
\newtheorem{lemma}[theorem]{Lemma}%

\theoremstyle{definition}%
\newtheorem{example}{Example}%
\newtheorem{claim}{Claim}
\newtheorem{assumption}{Assumption}
\newtheorem{conjecture}{Conjecture}

\newtheorem{definition}{Definition}%

\definecolor{dandelion}{RGB}{255, 215,   0}
\definecolor{brown}{RGB}{139,  69,  19}
\definecolor{slate}{RGB}{47,  79 , 79}
\definecolor{pink}{RGB}{255,  20, 147}
\definecolor{violet}{RGB}{138,  43, 226}
\definecolor{olive}{RGB}{128, 128,   0}
\definecolor{aqua}{RGB}{127, 255, 212}
\definecolor{navy}{RGB}{0, 0, 128}
\definecolor{indigo}{RGB}{75,   0, 130}
\definecolor{gold}{RGB}{165, 124, 0}
\definecolor{silver}{RGB}{131, 137, 0}
\definecolor{apple}{RGB}{34, 139, 34}

\def\red#1{{\color{red} #1}}

\def\oldemptyset\emptyset
\def\emptyset\varnothing
\def\red#1{{\color{red}#1}}
\def\root{\rho}

\def\tr{^{\top}}

\newcommand{\tree}{\mathcal{T}}
\newcommand{\nstages}{T}
\newcommand{\nrevisions}{K}
\newcommand\nodes{\mathcal{N}}

\newcommand\scenario{\omega}
\newcommand\scenarioset{\Omega}
\newcommand\height{\nstages}
\newcommand{\parent}[1]{\operatorname{pa}(#1)}
\newcommand{\descendant}[1]{\operatorname{des}(#1)}

\newcommand{\stage}[1]{\tau(#1)}
\newcommand{\dimstrategic}[1]{d^1_{#1}}
\newcommand{\dimoperationalint}[1]{d^2_{#1}}
\newcommand{\dimoperationalcon}[1]{d^3_{#1}}

\newcommand{\set}[1]{\{#1\}}
\newcommand{\order}[1]{\calO(#1)}
\newcommand{\revisableset}[1]{\mathcal{X}_{#1}}

\newcommand{\children}[1]{\operatorname{ch}(#1)}

\newcommand{\indicator}[1]{\mathbf{1}_{\{#1\}}}

\newcommand{\setdiff}{\backslash}
\newcommand{\join}{\vee}
\newcommand{\norm}[1]{\left\lVert#1\right\rVert}
\newcommand{\sib}[1]{\operatorname{sib}(#1)}
\newcommand{\proj}[2]{\operatorname{proj}_{#1}\left(#2\right)}
\newcommand\Real{\mathbb{R}}
\newcommand\Bin{\{0,1\}}
\newcommand\Int{\mathbb{Z}}

\newcommand\NPhard{\mathcal{NP}\text{-hard}}
\newcommand\NPcomplete{\mathcal{NP}\text{-complete}}

\newcommand{\subheight}[1]{\norm{#1}}
\newcommand{\subtreefamily}[2]{\scrS_{#1}(#2)}
\newcommand{\hypercube}[2]{\operatorname{HC}_{#1}(#2)}
\newcommand{\subroot}[1]{\root(#1)}
\newcommand{\leftnodes}[1]{\mathcal{L}(#1)}

\newcommand{\conv}[1]{\operatorname{conv}(#1)}
\newcommand{\MSP}{MSP}
\newcommand{\DP}{DP}

\def\Btree{{$B$-tree}}
\def\Stree{{$S$-tree}}

\DeclareMathOperator*{\st}{\text{s.t.}}
\DeclareMathOperator*{\LP}{\operatorname{LP}}
\DeclareMathOperator*{\Binomial}{\operatorname{Binomial}}

\makeatletter
\newcommand{\x}{%
  \@ifnextchar\bgroup{\x@one}{\ensuremath{x}}%
}
\newcommand{\x@one}[1]{%
  \@ifnextchar\bgroup
    {\x@two{#1}}               
    {\ensuremath{x(#1)}}       
}
\newcommand{\x@two}[2]{%
  \ensuremath{x(#1)_{#2}}      
}

\newcommand{\y}{%
  \@ifnextchar\bgroup{\y@one}{\ensuremath{y}}%
}
\newcommand{\y@one}[1]{%
  \@ifnextchar\bgroup
    {\y@two{#1}}
    {\ensuremath{y(#1)}}
}
\newcommand{\y@two}[2]{%
  \ensuremath{y(#1)_{#2}}

}

\newcommand{\cc}{%
  \@ifnextchar\bgroup{\cc@one}{\ensuremath{c}}%
}
\newcommand{\cc@one}[1]{%
  \@ifnextchar\bgroup
    {\cc@two{#1}}               
    {\ensuremath{c(#1)}}       
}
\newcommand{\cc@two}[2]{%
  \ensuremath{c(#1)_{#2}}      
}

\newcommand{\plan}{%
  \@ifnextchar\bgroup{\plan@one}{\ensuremath{\pi}}%
}
\newcommand{\plan@one}[1]{%
  \@ifnextchar\bgroup
    {\plan@two{#1}}               
    {\ensuremath{\pi(#1)}}       
}
\newcommand{\plan@two}[2]{%
  \ensuremath{\pi(#1,#2)}      
}

\newcommand{\rv}{%
  \@ifnextchar\bgroup{\rv@one}{\ensuremath{r}}%
}
\newcommand{\rv@one}[1]{%
  \@ifnextchar\bgroup
    {\rv@two{#1}}               
    {\ensuremath{r(#1)}}       
}
\newcommand{\rv@two}[2]{%
  \ensuremath{r(#1)_{#2}}      
}
\makeatother

\newcommand{\calE}{\mathcal{E}}

\newcommand\calL{\mathcal{L}}

\newcommand\calS{\mathcal{S}}

\newcommand{\calO}{\mathcal{O}}

\newcommand{\calH}{\mathcal{H}}

\newcommand\scrS{\mathscr{S}}

\newcommand{\overbar}[1]{\mkern 1.5mu\overline{\mkern-1.5mu#1\mkern-1.5mu}\mkern 1.5mu}

\usepackage{fancyhdr}
\title{Balancing adaptability and predictability: $\nrevisions$-revision multistage stochastic programming}

\author{
 Chengwenjian Wang \\
  Department of Industrial and Systems Engineering\\
  University of Minnesota, Twin Cities\\
  Minneapolis, MN, 55455 \\
  \texttt{wcwj@umn.edu} \\
   \And
 Alexander S. Estes \\
  Robert H. Smith School of Business and Institute of Systems Research\\
  University of Maryland\\
  College Park, MD, 20742 \\
  \texttt{aestes@umd.edu} \\
  \And
 Jean-Philippe P. Richard \\
  Department of Industrial and Systems Engineering\\
  University of Minnesota, Twin Cities\\
  Minneapolis, MN, 55455 \\
  \texttt{jrichar@umn.edu} \\
}

\begin{document}

\maketitle

\begin{abstract}
A standard assumption in multistage stochastic programming is that decisions are made after observing the uncertainty from the prior stage. 
The resulting solutions can be difficult to implement in practice, as they leave practitioners ill-prepared for future stages. 
To provide better foresight, we 
introduce the $\nrevisions$-revision approach.
This new framework requires plans to be specified in advance. 
To maintain flexibility, we allow plans to be revised a maximum of $K$ times as new information becomes available. 
We analyze the complexity of $\nrevisions$-revision problems, showing $\mathcal{NP}$-hardness even in a simple setting.
We examine, both theoretically and computationally, the impact of the $\nrevisions$-revision approach on the objective compared with classical multistage stochastic programming models and the partially adaptive approach introduced in \cite{basciftci2024,kayacik2025}.
We develop two MIP formulations, one directly from our definition and the other based on a combinatorial characterization.
We analyze the tightness of these formulations and propose several methods to strengthen them.
Computational experiments on synthetic problems and practical applications demonstrate that the $\nrevisions$-revision approach is both computationally tractable and effective in reaching near-optimal performance while increasing the predictability of the solutions produced.

\keywords{Multistage Stochastic Programming \and $\nrevisions$-revision \and Integer Programming \and Polyhedral Theory}
\end{abstract}

\newpage
\section{Introduction} \label{sec: introduction}

Optimizing decisions over time in the presence of uncertainty is an important and well-studied problem. 
Multistage stochastic programming (\MSP) is a fundamental paradigm to tackle this problem, allowing decisions to be made sequentially across multiple stages.
This paradigm produces fully adaptive policies, where actions made in one stage depend on the information observed and actions made in all previous stages.  

Despite its flexibility, classical \MSP{} has notable limitations. 
In particular, fully adaptive policies generated by these models can vary significantly across scenarios. 
Such variability may cause operational challenges, including difficulties in communication, coordination, and execution. 
Moreover, practical decision-making processes often require a certain degree of predictability to facilitate planning and allocate resources. 
Decision makers typically prefer an initial strategic plan that is robust enough to handle uncertainty without frequent changes. 
Thus, there is a need for modeling approaches that take advantage of the adaptivity afforded by \MSP{}  but also incorporate considerations of predictability. 

In this paper, we propose a novel approach that balances these conflicting considerations.
In our approach, we provide a \emph{plan} at the start of the planning horizon. 
The plan specifies the strategic decisions in each future stage and must be followed unless a \emph{revision} is applied.
When a revision is applied, a new plan for the remaining future decisions is created, and this plan will guide all future decisions until another revision is applied.
To promote predictability of the solution, we require the number of revisions applied in each realized scenario to be bounded above by some non-negative integer $\nrevisions$.

The $\nrevisions$-revision approach we propose has several appealing features. 
It provides decision makers with explicit control over the flexibility of the implemented policy through the choice of parameter $\nrevisions$. 
By evaluating the gain achieved when increasing $\nrevisions$ from smaller to larger values, decision makers can identify the trade-off that best fits their application between expected objective value and predictability in the solution. 
Furthermore, this approach aligns with human decision-making processes, where an initial plan is usually established first and revised as uncertainty resolves. 

Many problems may benefit from being formulated and solved with the $\nrevisions$-revision approach. 
One example is the \emph{single airport ground-holding program} (SAGHP), which is concerned with the management of a single airport and the flights destined for that airport \cite{richetta1993, richetta1994, ball2003}.
Under bad weather or other factors, the capacity of an airport can decrease. 
To prevent inbound aircraft from exceeding the capacity of the airport, the Federal Aviation Administration (FAA) can hold flights destined to that airport on the ground,  transferring delays from air to ground.
\MSP{} is a popular solution approach for this problem, since it is desirable for decisions to be made dynamically in response to weather and other changes \cite{mukherjee2007, estes2020}.
However, classical \MSP{} models, being highly adaptive, lack commitment guarantees, making them challenging to operationalize. 
In practice, the FAA would prefer to revise the assigned delays only sparingly.
Our $\nrevisions$-revision approach is useful in this situation, as it allows the administration some dynamism and flexibility in its traffic management initiatives, while avoiding an excessive number of revisions, which pose logistical challenges for the airlines involved. 

Another example is the dynamic \emph{lot-sizing} problem, a classical production planning problem \cite{belvaux2001, loparic2001}. 
Production scheduling decisions often require commitment in advance due to staffing or equipment setup considerations. 
However, applying the same schedule to each scenario is economically inefficient when demand is uncertain \cite{hu2018}. 
Applying the $\nrevisions$-revision approach would allow manufacturers to initially commit to certain production decisions and then strategically revise their schedules a limited number of times in response to realized demand and updated demand forecasts. 

\subsection{Related work}
\label{sec: related work}

\MSP{} is a popular modeling framework for various applications, including production scheduling, capacity planning, transportation, and energy system management \cite{mukherjee2007, estes2020, bhattacharya2016, ding2017, housh2013, hu2018, huang2005, huang2009, nagar2008, pereira1991, xu2015}.
Recently, several studies have proposed variants of \MSP{} models aimed at addressing practical issues related to the lack of predictability and consistency of decisions across scenarios.
Among those approaches, the works of \cite{basciftci2024, kayacik2025} are particularly relevant to our study.
The authors of \cite{basciftci2024} propose a partially adaptive two-stage stochastic programming method.
A \emph{revision point} is selected for each decision at the start of the program, and the decision can only be revised at this point of time.
Subsequently, \cite{kayacik2025} generalizes this idea to a partially adaptive multistage framework, allowing multiple predetermined revision points for each decision.

Our $\nrevisions$-revision approach differs notably from the aforementioned partial adaptivity approach. 
In the $\nrevisions$-revision approach, the timing of revision decisions can vary scenario-by-scenario, whereas under partial adaptivity, revision points are predetermined and uniform across scenarios.
The flexibility of our approach can be highly beneficial because practical decision revisions often correspond not to a specific predetermined stage, but rather to key scenario-dependent events.
For example, consider SAGHP, where scenarios represent airport capacity fluctuations driven by weather conditions \cite{mukherjee2007, estes2020}.
Under a partially adaptive approach, one must select a fixed stage to adjust the plan.
In these situations, however, it is more practical to revise decisions precisely when critical scenario-dependent events, such as significant weather shifts, occur, rather than restricting revisions to stages fixed across all scenarios.

Another approach for balancing optimality and predictability is the $K$-adaptability approach, in which $K$ candidate solutions are computed upfront and the best one is selected after observing the realized uncertainty. 
This approach was introduced in \cite{bertsimas2010} with complexity results, exact and approximate solution methods for two-stage $K$-adaptable stochastic optimization later provided in \cite{buchheim2019, malaguti2022} .
Similarly, \cite{hanasusanto2015, subramanyam2020} apply the $K$-adaptability approach to two-stage robust optimization.

In contrast to our proposed $\nrevisions$-revision approach, which applies to multistage problems, the $K$-adaptability approach is designed for two-stage problems.
Further, the $\nrevisions$-revision approach focuses on limiting the number of plan changes along each scenario, whereas the $K$-adaptability approach limits the number of plans being used across the scenarios. 
Hence, the  $\nrevisions$-revision approach has the additional dimension of timing when an original plan must be substituted for another along each scenario. 

We work primarily with formulations in an extensive form deterministic equivalent, in which the scenario tree is fully expanded and decision variables are associated with each node.
For some problems, the extensive form can be solved directly \cite[e.g.,][]{richetta1994,mukherjee2007,estes2020,korpeoglu2011,yin2021,ding2017,hu2018}. 
In general, \MSP{} may involve large scenario trees for which direct solution of the extensive form deterministic equivalent is not tractable. 
There is a wide variety of methods designed to improve the tractability of \MSP{}. 
For example, there are various methods designed to generate concise scenario trees or to reduce the size of existing scenario trees \cite[e.g.,][]{mimouni2026scenariotree,kirui2025tree,henrion2022,horejsova2020evaluation,pflug2016empirical,pflug2015,xu2015,heitsch2009a,heitsch2009,casey2005scenario,heitsch2003scenario,dupacova2003}. 
Such methods could be used directly in combination with the extensive formulations.
In addition, there are many decomposition methods that have been developed for \MSP{}; see \cite{romeijnders2026survey} for a recent survey.
Designing decomposition methods for $\nrevisions$-revision \MSP{} is beyond the scope of this paper.
Instead, we discuss the potential applicability of existing decomposition methods to our formulations in \Cref{subsec: decomposition methods}.

\subsection{Contributions}
In this paper, we lay the theoretical and algorithmic foundation for the solution of 
\MSP{} models with $\nrevisions$-revision constraints.
We make the following contributions:
\begin{enumerate}
\item 
We introduce and rigorously define $\nrevisions$-revision \MSP. 
\item We analyze the computational complexity of $\nrevisions$-revision \MSP. 
Specifically, we establish that even simple problems with $\nrevisions$-revision constraints are $\NPhard$.
\item 
We introduce a combinatorial description of the $\nrevisions$-revision constraint that helps in developing formulations. 
\item 
We propose two mixed integer programming (MIP) formulations and establish tightness results for both. 
We further describe the type of scenario tree structures for which they are most appropriate to use. 
\item 
We develop techniques to reduce the sizes of both formulations and to strengthen the relaxation bounds they provide.
Specifically, for the first formulation, we propose a variable-fixing procedure and a cut generation approach that uses facet-defining inequalities. 
For the other, we propose a constraint generation algorithm and an extended formulation.
\item 
We demonstrate the effectiveness of our results on synthetic and practical problems, including $\nrevisions$-revision variants of 0-1 stochastic unconstrained optimization, uncapacitated lot-sizing, capacity planning, and SAGHP problems. 
For practical SAGHP instances, we show that $\nrevisions$-revision models can be solved to optimality without significant deterioration from the fully adaptive optimal objective, and with clear improvement compared to the optimal objectives from partially adaptive models.

\end{enumerate}

\subsection{Organization}
The rest of this paper is organized as follows. 
We first introduce background concepts and notation in \Cref{sec:preliminary}.
In Section~\ref{sec: description}, we formally define the $\nrevisions$-revision problem, investigate its complexity, analyze the trade-off between optimal objective value and decision consistency provided by $\nrevisions$-revision constraints,
and present an alternate combinatorial characterization of the $\nrevisions$-revision requirement.
In \Cref{sec: formulations}, we introduce two formulations of the $\nrevisions$-revision problem. 
In \Cref{sec: methods}, we describe methods to strengthen them and to reduce their sizes.
We investigate variants of these two formulations computationally across a range of problems in \Cref{sec: experiments}.
We give concluding remarks in \Cref{sec: conclusion}.

\section{Preliminaries}\label{sec:preliminary}

\paragraph{Notation.}\label{subsec:notation}
For positive integers $a$ and $b$, we use $[a]$ to denote the set $\{1, 2, \ldots, a\}$ and $[a:b]$ to denote the set $\{a, a+1, \ldots,b\}$.
Given a set $S \subseteq [n]$, we use $\mathbf{1}_S$ to represent the vector $x$ in $\Real^n$ for which $x_i=1$ when $i\in S$ and $x_i=0$, otherwise. 
We use 
$\proj{x}{Q}$ for the projection of a set $Q$ onto the space spanned by vector $x$.
Finally, for a polytope $P \subseteq \Real^n$, we define $P^I = \conv{P \cap \Int^n}$, where $\conv{S}$ is used to represent the convex hull of set $S$.

\paragraph{Strength of MIP formulations.}\label{subsec:mip_formulation}
Let $\mathcal{X} \subseteq \Bin^n$ be a set of feasible binary solutions of interest. 
A system of linear inequalities with integrality requirements
\begin{equation} \label{eq:mip_formulation}
    A_x x + A_\lambda \lambda + A_\nu \nu  \le b, \quad
    x \in \Bin^n, \quad
    \lambda  \in \Real^s, \quad
    \nu  \in \Int^t, 
\end{equation}
is an \emph{MIP formulation} for $\mathcal{X}$ if the projection of \Cref{eq:mip_formulation} onto the $x$-variables is precisely $\mathcal{X}$. 
This MIP formulation is \emph{sharp} if and only if the projection onto the $x$-variables of its LP relaxation is exactly $\conv{\mathcal{X}}$.
It is called \emph{locally ideal} if and only if all basic feasible solutions of its LP relaxation, of which there is at least one, have integral $x$ and $\nu$ values.
In particular, if a formulation involves only the variables $x$ (\emph{i.e.}, $s=t=0$) and is locally ideal, we simply call it \emph{ideal}.
A locally ideal formulation is sharp, but a sharp formulation is not necessarily locally ideal.
We refer to \cite{vielma2015} for a discussion.
We define the \emph{size} of an MIP formulation to be the total number of its variables and constraints.

\paragraph{Scenario tree.}\label{subsec:scenario tree}
The uncertainty present in an \MSP{}  model is often represented using a \emph{scenario tree} \cite{heitsch2009,heitsch2009a,pflug2015}. 
Each \emph{node} of the tree represents a specific \emph{state}, \emph{i.e.}, a particular realization of all uncertain parameters up to its corresponding \emph{stage}.
In this paper, we use $\tree$ to denote a scenario tree and $\nstages$ to denote the number of its stages.
We let $\root$ be the root of the tree.
We use $\nodes$ to denote the set of all nodes of $\tree$.
We denote the stage of node $v \in \nodes$ by $\stage{v}$.
As it is common practice, we assume without loss of generality that all leaves of $\tree$ are at the same final stage, \emph{i.e.}, each root-to-leaf path has length $\nstages$. 
We denote the set of all nodes at stage $t$ by $\nodes_t$.
A \emph{scenario} $\scenario$ is a complete path from the root node $\root$ to a leaf node of the tree. 
Each scenario $\scenario$ represents one possible complete unfolding of all uncertainties over the entire planning horizon.
We denote the set of all such scenarios by $\scenarioset$ and the probability of scenario $\scenario$ by $p_{\scenario}$. 

We also make use of common graph theory terminology \cite{BondyMurty2008}.
We refer to the set of children of a non-leaf node $v$ as $\children{v}$. 
The unique \emph{parent} of a non-root node $v$ is denoted by $\parent{v}$.
We use $\descendant{v}$ to denote the set of \emph{descendants} of a node $v$.
There is a natural partial order on the nodes of a tree, where $u < v$ (\emph{resp.} $u > v$) if and only if $u$ is an ancestor (\emph{resp.} descendant) of $v$.
Naturally, $u \leq v$ (\emph{resp.} $u \geq v$) denotes that $u$ is an ancestor (\emph{resp.} descendant) of $v$ or is $v$ itself.
The \emph{join} of two nodes $u$ and $v$,  $u \join v$, is 
the greatest lower bound of $u$ and $v$ in this partial order, \emph{i.e.},
their common ancestor at the maximum stage.
We define the \emph{subtree of $\tree$ rooted at $v$}, $\tree(v)$, to be the subgraph of $\tree$ induced by $v$ and all its descendants.
An \emph{embedded subtree} of a tree is obtained by deleting or contracting some edges from the original tree \cite{lozano2004}.
We denote the set of nodes of an embedded subtree $\calS$ by $\nodes(\calS)$.

\section{Problem description}\label{sec: description}

\subsection{Defining \texorpdfstring{$\nrevisions$}{K}-revision constraints} \label{subsec: definition}

Our $\nrevisions$-revision approach for a $\nstages$-stage decision process can be described in the following dynamic way.
The process begins by constructing an initial plan, which prescribes the decisions for all stages.
Then, at the beginning of each subsequent stage, after the current state of the system is revealed, a decision is made to either adhere to the existing plan or to revise it.
If the plan is not revised, the decisions for the current stage are implemented as specified by the current plan. 
If a revision occurs, a new plan is created for all remaining stages. 
The decisions for the current stage are then implemented according to this new plan.  
The $\nrevisions$-revision approach implements this procedure with the constraint that, for any scenario, the plan is revised no more than $\nrevisions$ times by the end of the decision horizon.

We distinguish between two types of decisions: \emph{strategic} and \emph{operational}. 
Strategic decisions represent higher-level decisions that are critical to know in advance for planning purposes, whereas operational decisions represent lower-level and less critical decisions.
For instance, in a lot-sizing problem, whether to produce at a given time is a strategic decision, whereas how much to produce is an operational one.
We model situations where the decision maker wishes to commit to a plan for strategic decisions and is reluctant to deviate from it, but is flexible with changes to operational decisions.

To describe a generic formulation for the $\nrevisions$-revision approach, we model a sequential decision-making problem with $\nstages$ periods as an \MSP{} model on a scenario tree $\tree$.

We assume that each stage $t\in [\nstages]$ has $\dimstrategic{t}$ strategic binary decisions.
Since binary encoding can represent integer numbers, this setting easily generalizes to integer decision variables, whereas continuous strategic decisions fall outside the scope of our analysis.
Furthermore, we assume there are $\dimoperationalint{v}$ integer and $\dimoperationalcon{v}$ continuous operational decisions at node $v$.
Notably, we assume that the dimensions of the strategic decisions are consistent across all nodes in the same stage but we allow the dimensions of the operational variables to vary from node to node.
This assumption is essential because a common dimension across all nodes in the same stage makes it meaningful to compare strategic decisions across nodes and determine whether a plan revision occurs, a notion we formalize in the following paragraphs.
We let $\x{v} \in \Bin^{\dimstrategic{\stage{v}}}$  and $\y{v} \in \Int^{\dimoperationalint{v}} \times \Real^{\dimoperationalcon{v}}$ be the strategic and operational decision variables at node $v$, respectively.
We refer to the collection $\{(\x{v}, \y{v})\}_{v \in \nodes}$ as a \emph{policy}, which maps each node to a specific decision. 
Similarly, we refer to $\{\x{v}\}_{v \in \nodes}$ as a \emph{strategic policy}.

We define a \emph{plan} $\pi(v,\cdot)$ at a node $v$ in stage $\tau(v)$ to be a collection $\{\pi(v,t)\}_{t\in[\tau(v):\nstages]}$ specifying the committed strategic decisions for stages $\tau(v)$ to $\nstages$ based on the information realized at state $v$. 
For an \MSP{} problem, we define a \emph{plan adjustment policy} $\{\plan{v,\cdot}\}_{v\in\nodes}$ to be a collection of plans, with one plan for each node.

\begin{example}
Consider a 3-stage problem where $d^1_t=1$ for $t \in [3]$ on the scenario tree depicted in the left panel of \Cref{fig: 1-revision illustration}. 
A plan adjustment policy specifies a plan at each node, such as $\pi(\root,\cdot)=[1,0,1]$, $\pi(v_1,\cdot)=[0,1]$, and $\pi(v_2,\cdot)=[1,0]$.
\end{example}

Any component of $\plan{v,\cdot}$ following the first one specifies a strategic decision to be executed in future stages that we reserve the right to revise if the need arises. 
For the current stage $\stage{v}$, however, the strategic decision has to coincide with $\plan{v,\stage{v}}$.
Hence, we say that the plan adjustment policy $\plan$ is \emph{compatible} with strategic policy $\x$ if 
\begin{eqnarray*}
\plan{v}{\stage{v}} = \x{v}, \quad \forall v \in \nodes.
\end{eqnarray*}
We say that the plan adjustment policy $\plan$ requires a \emph{revision} at node $v$ if the active plan at node $v$ is not consistent with the plan active at the parent of $v$.
Specifically, $\plan{v,\cdot}$ requires a revision at $v$ if there exists some $t \in [\stage{v}:\nstages]$ such that $\plan{v}{t} \neq \plan{\parent{v}}{t}$.
We use $r_{\plan}$ to denote the \emph{revision policy} associated with $\plan$, in which $r_{\plan}(v)$ describes whether a revision is required at node $v$, so that
\begin{equation*}
    r_{\plan}(v) = 
    \begin{cases}
        1, \quad \exists \, t\in [\stage{v}:\nstages]\,\, \textrm{s.t. } \plan{v,t} \neq \plan{\parent{v},t},   \\
        0, \quad \text{otherwise}.
    \end{cases}
\end{equation*}
It is worth mentioning that even though a plan adjustment policy corresponds to a unique strategic policy and a unique revision policy, a strategic policy may have multiple compatible plan adjustment policies with different revision policies.

\addtocounter{example}{-1}
\begin{example}[continued]
Consider the scenario tree in the left panel of \Cref{fig: 1-revision illustration} and the
strategic policy $\x$ given by $\x{\root}=\x{v_2}=\x{v_4}=1$ and $\x{v_1}=\x{v_3}=\x{v_5}=\x{v_6}=0$.
The plan adjustment policy $\plan$ shown in the figure is compatible with this strategic policy $\x$.
For instance, the first entry of $\plan{v_1}$ is equal to $\x{v_1}=0$, and the first entry of $\plan{v_2}$ is equal to $\x{v_2}$.
Further, $r_{\plan}(v_2)=1$ because the last two components of $\plan{\root}$ form a vector that is different from $\plan{v_2}$.  
\end{example}

As the decision maker values both optimality and consistency, we seek to limit the number of revisions under each scenario.
Formally, a plan adjustment policy $\plan$ is \emph{$\nrevisions$-revisable} if $\sum_{v \in \scenario}r_{\plan}(v) \leqslant \nrevisions$ for all $\scenario \in \scenarioset$.
This leads to the following definition.

\begin{definition}[$\nrevisions$-revision constraint/$\nrevisions$-revisable policy] \label{def: revision constraint}
Given a scenario tree $\tree$ and a non-negative integer $\nrevisions$,
a strategic policy $\{\x{v}\}_{v \in \nodes}$ is said to satisfy the \emph{$\nrevisions$-revision constraint} (or equivalently to be \emph{$\nrevisions$-revisable}) if there exists a $\nrevisions$-revisable plan adjustment policy $\{\plan{v}\}_{v\in \nodes}$ that is compatible with $\x$.
We denote the set of all $\nrevisions$-revisable strategic policies for $\tree$ by $\revisableset{\nrevisions}(\tree)$ or simply by $\revisableset{\nrevisions}$ when $\tree$ is clear from the context.
\end{definition}

We refer to the minimum integer $\nrevisions$ such that a given strategic policy $x$ is $\nrevisions$-revisable as the \emph{minimum revisability number} of $x$.
There is a polynomial algorithm, which we describe in \Cref{subsec: subtree enhancement}, to find the minimum revisability number of a strategic policy; see \Cref{cor: minimum revisability number}.

\addtocounter{example}{-1}
\begin{example}[continued] 
In the left panel of \Cref{fig: 1-revision illustration}, the plan adjustment policy $\plan$ is such that only two nodes need a revision, and for every scenario, there is either $1$ or $0$ revision in total.
Hence, the strategic policy shown in this panel is $1$-revisable.
For the plan adjustment policy given in the right panel, the scenario $\root \xrightarrow{} v_2 \xrightarrow{} v_5$ contains two nodes requiring a revision.
We will argue in Section~\ref{subsec: combinatorial characterization} 
that the minimum revisability number of the strategic policy in this panel is $2$.
\end{example}

\begin{figure}[htbp]
\begin{minipage}[c]{0.48\textwidth}
\centering
\scalebox{1}{
\begin{tikzpicture}[
    scale=0.8,
    grow=right,	
    every node/.style={general node}
    ]
    
    \node [value-1 node, label={[yshift=-5.1em]$[1, \color{mycolor1}0 \color{black}, \color{mycolor2} 1 \color{black} ]$}]{$\root$}
        child {node [revised node, value-1 node, level-1 node, label={[yshift=-4.7em] $[\color{mycolor1}1 \color{black}, \color{mycolor2} 0 \color{black} ]$}] {$v_2$}
            child {node [level-2 node, label={[yshift=-4.2em] $[ \color{mycolor2} 0 \color{black} ]$}] {$v_6$}}
            child {node [level-2 node, label={[yshift=-4.2em] $[ \color{mycolor2} 0 \color{black} ]$}] {$v_5$}}
        }
        child {node [level-1 node, label={[yshift=-4.7em] $ [\color{mycolor1}0 \color{black}, \color{mycolor2} 1 \color{black} ]$}] {$v_1$}
            child {node [value-1 node, level-2 node, label={[yshift=-4.2em] $[ \color{mycolor2} 1 \color{black} ]$}] {$v_4$}}
            child {node [revised node, level-2 node, label={[yshift=-4.2em] $[ \color{mycolor2} 0 \color{black} ]$}] {$v_3$} }
        };

\end{tikzpicture}
}
\end{minipage}
\begin{minipage}{0.48\textwidth}

\begin{tikzpicture}[
    scale=0.8,
    grow=right,	
    every node/.style={general node}    
    ]
    
    \node [value-1 node, label={[yshift=-5.1em]$[1, \color{mycolor1}0 \color{black}, \color{mycolor2} 1 \color{black} ]$}]{$\root$}
        child {node [revised node, value-1 node, level-1 node, label={[yshift=-4.7em] $[\color{mycolor1}1 \color{black}, \color{mycolor2} 1 \color{black} ]$}] {$v_2$}
            child {node [value-1 node, level-2 node, label={[yshift=-4.2em] $[ \color{mycolor2} 1 \color{black} ]$}] {$v_6$}}
            child {node [revised node, level-2 node, label={[yshift=-4.2em] $[ \color{mycolor2} 0 \color{black} ]$}] {$v_5$}}
        }
        child {node [level-1 node, label={[yshift=-4.7em] $ [\color{mycolor1}0 \color{black}, \color{mycolor2} 1 \color{black} ]$}] {$v_1$}
            child {node [value-1 node, level-2 node, label={[yshift=-4.2em] $[ \color{mycolor2} 1 \color{black} ]$}] {$v_4$}}
            child {node [revised node, level-2 node, label={[yshift=-4.2em] $[ \color{mycolor2} 0 \color{black} ]$}] {$v_3$}}
        };
\end{tikzpicture}
\end{minipage}
\caption{Illustration of plan adjustment and revision policies.}
\footnotesize Note. A node $v$ is shaded gray if $\x{v}=1$, whereas a node is blank if $\x{v}=0$. The vectors beneath the nodes specify a plan adjustment policy. The last entry of any vector corresponds to the strategic decision being made/planned for stage $3$, the first entry for any second-stage node and the second entry for the root node relate to stage $2$, and the first entry for the root node relates to stage $1$. A node is marked with a star if the plan adjustment policy produces a revision at that node.
\label{fig: 1-revision illustration}
\end{figure}

We now define a $\nrevisions$-revision \MSP{}  problem as follows:
\begin{samepage}
\begin{subequations} \label{eq: K-revision MSP}
\begin{align}
	\max \quad  & \textstyle \sum_{v \in \nodes} \big(\cc{v}\tr \x{v} + f(v)\tr \y{v}\big) \label{eq: obj} \\
	\st  \quad & \textstyle  
        \mathrlap{\sum_{u \leq v} \big( D(u,v) \x{u} + H(u,v) \y{u} \big) \leqslant b(v) ,} 
        \qquad &&  \forall v \in \nodes, \label{eq: base constr}   \\
			   & \x{v} \in \Bin^{\dimstrategic{\stage{v}}}, \qquad && \forall v \in \nodes, \label{eq: IVRx} \\
			   & \y{v} \in \Int^{\dimoperationalint{v}} \times \Real^{\dimoperationalcon{v}},  \qquad && \forall v \in \nodes,  \label{eq: IVRy} \\
			   & \x \in \revisableset{\nrevisions}. \label{eq: K-revision}
\end{align}	
\end{subequations}
\end{samepage}
In this formulation, objective \eqref{eq: obj} results from simplifying its natural expression
$$\textstyle \sum_{\scenario \in \scenarioset} p_{\scenario} \bigg(\sum_{v \in \scenario} \left(\cc{v}\tr \x{v} + f(v)\tr \y{v}\right)\bigg)$$
by incorporating probabilities inside of objective coefficients. 
Constraints \Cref{eq: base constr} specify the formulation of the base problem, which is an \MSP{}  problem.
For example, it can be the stochastic lot-sizing problem or the stochastic dynamic SAGHP problem.
There, $D(u,v)$ and $H(u,v)$ are matrices that capture how decisions taken at node $u$ propagate and affect the feasibility of decisions at a downstream node $v$.
Inequalities \Cref{eq: IVRx} and \Cref{eq: IVRy} specify the types of the variables.
Constraint \Cref{eq: K-revision} imposes the $\nrevisions$-revision requirement, as described in \Cref{def: revision constraint}.

Instead of imposing a hard revision budget, we can also penalize revisions in the objective.
We discuss such revision penalty models and how they relate to \Cref{eq: K-revision} in \Cref{sec: penalty models}.

There are also other ways to impose a $\nrevisions$-revision constraint on the strategic decision variables when $d^1_t > 1$ for some $t$.
In some cases, we can consider that a revision is triggered whenever any of the variables $\plan{v,t}_i$ at node $v$ deviates from the plan of its parent.
Hence, the revision budget can be seen to be shared across all the strategic decision variables.
If one strategic decision in a plan is changed, this would take one revision from our global revision budget. 
This naturally fits applications such as SAGHP, where the FAA does not adjust delays for flights on an individual basis, but instead issues a single revision that could adjust the delay assigned to any affected flight.
However, in other contexts, it may be beneficial to handle revisions for different groups of strategic decisions separately. 
Our formulation \Cref{eq: K-revision MSP} could be easily extended to accommodate this situation by introducing multiple constraints of the form \Cref{eq: K-revision}, each corresponding to a distinct group of strategic decisions.
We focus on the first setting in this paper.

\begin{figure}[tbp]
  \centering
  \begin{minipage}{0.38\textwidth}
    \centering
    \scalebox{1}{
    \begin{tikzpicture}[
        scale=0.8,
        grow=right,
        every node/.style={general node}
      ]
      \node {$\root$}
        child {node [label={[yshift=-3.4em] $\x{v_2}=[\x{v_2}{i}]_{i\in [3]}$}] {$v_2$}
          child {node  {$v_6$}}
          child {node  {$v_5$}}
        }
        child {node [label={[yshift=-3.4em] $\x{v_1}=[\x{v_1}{i}]_{i\in[3]}$}] {$v_1$}
          child {node  {$v_4$}}
          child {node  {$v_3$}}
        };
    \end{tikzpicture}}
  \end{minipage}\hspace{1cm}%
  \raisebox{-0.5cm}{
  \begin{minipage}{0.52\textwidth}
    \centering
    \scalebox{1}{
    \begin{tikzpicture}[
        scale=0.8,
        grow=right,
        every node/.style={general node}
      ]
      \tikzstyle{level 1}=[level distance=20mm, sibling distance=30mm]
      \tikzstyle{level 2}=[level distance=14mm, sibling distance=15mm]
      \tikzstyle{level 3}=[level distance=14mm]
      \tikzstyle{level 4}=[level distance=14mm]
      \tikzstyle{level 5}=[level distance=18mm, sibling distance=13mm]

      \node {$\root$}
        child {node [inner sep=0.2em, label={[yshift=-0.25em] $\x{v_2^1}$}] {$v_2^1$}
          child {node [inner sep=0.2em,label={[yshift=-0.25em] $\x{v_2^2}$}] {$v_2^2$}
            child {node [inner sep=0.2em,label={[yshift=-0.25em] $\x{v_2^3}$}] {$v_2^3$}
              child {node  {$v_6$}}
              child {node  {$v_5$}}
            }
          }
        }
        child {node [inner sep=0.2em,label={[yshift=-0.25em] $\x{v_1^1}$}] {$v_1^1$}
          child {node [inner sep=0.2em,label={[yshift=-0.25em] $\x{v_1^2}$}] {$v_1^2$}
            child {node [inner sep=0.2em,label={[yshift=-0.25em] $\x{v_1^3}$}] {$v_1^3$}
              child {node  {$v_4$}}
              child {node  {$v_3$}}
            }
          }
        };
    \end{tikzpicture}}
  \end{minipage} }
    \vspace{+0.1cm}
  \caption{Transforming multi-dimensional to one-dimensional strategic decisions.}
  \label{fig: transform multi-d to 1-d}
\end{figure}

In this setting, we can transform a problem with $\dimstrategic{t}>1$ at stage $t$ to a problem where $\dimstrategic{t}=1$ for all stages $t$. 
This can be done simply by splitting a stage $t$ with $\dimstrategic{t}>1$ into $\dimstrategic{t}$ artificial stages with a 1-dimensional decision for each; 
see \Cref{fig: transform multi-d to 1-d}. 
The transformation also preserves the number of revisions along each scenario. 
In this figure, a change of plan at $v_1$ in the original tree can be fully implemented at $v_1^1$ of the expanded tree, still incurring a single revision. 
Conversely, no uncertainty is revealed between the artificial nodes $v_1^1$, $v_1^2$, and $v_1^3$.
Hence, any revision at $v_1^2$ and $v_1^3$ can be moved to $v_1^1$.
Because of the transformation, we make the following simplifying assumption without loss of generality. 
\begin{assumption} \label{assumption: 1-dimension}
For each $t\in[\nstages]$, $\dimstrategic{t}=1$, \emph{i.e.}, there is exactly one strategic decision variable at each node in the scenario tree.
\end{assumption}
Throughout this paper, we only relax this assumption in some specific contexts, all stated explicitly: the analysis of complexity in \Cref{subsec: complexity}, certain practical applications in \Cref{sec: experiments}, and the multi-dimensional extension of formulations in \Cref{sec: multi-d formulations}.

We now discuss basic properties of the problem. 
Let $z_\nrevisions$ denote the optimal objective value of \Cref{eq: K-revision MSP} where we make the dependence on the revision limit $\nrevisions$ explicit.
Increasing $\nrevisions$ allows greater flexibility in adapting strategic decisions, which results in improved objective values, as the model becomes less restrictive.  
The following monotonicity property naturally follows. 
\begin{proposition}
It holds that 
$z_0 \leqslant z_1 \leqslant \ldots \leqslant z_{\nstages-2} \leqslant z_{\nstages-1}=z_\infty$.
\hfill $\qed$
\end{proposition}

Two extreme cases correspond to known models.  
When $\nrevisions=0$, strategic decisions for all stages must be fully specified at the root node.
In problems where operational variables $y$ only affect their corresponding stages and serve purely as recourse decisions, this is equivalent to a two-stage stochastic program with a first-stage decision over $x$ and where recourse decisions $y$ are made in the second stage based on scenario outcomes.
When $\nrevisions \geqslant \nstages-1$, 
the $\nrevisions$-revision constraint becomes non-binding because revisions can be made at all nodes except the root, where a revision is never needed.
In this case, the model permits full adaptivity of the strategic decisions $x$, and the problem reduces to the classical \MSP{} model formulated over the original scenario tree.
Therefore, we develop complexity and polyhedral results for \Cref{eq: K-revision MSP} in the intermediate regimes where $1 \leqslant \nrevisions \leqslant \nstages-2$.

\subsection{Complexity analysis}  \label{subsec: complexity}

When the base \MSP{} model \Cref{eq: obj}-\Cref{eq: IVRy} is difficult to solve, we would expect the problem to remain difficult when adding a $\nrevisions$-revision constraint in most typical cases.
Therefore, we consider next the simplest possible $\nrevisions$-revision \MSP{} model, where there are no operational variables and the only constraint imposed on the strategic variables, besides variable bounds, is the $\nrevisions$-revision constraint.
We argue that, even in this simple setting, the problem is $\NPhard$.
We also describe that the difficulty inherently comes when many decisions are made at each stage or when the number of stages is large. 
In contrast, we show that there is a polynomial algorithm when these characteristics are fixed.
We do not impose \Cref{assumption: 1-dimension} in this section as the dimensionality of the decision variables is a non-negligible factor for computational complexity. 

Consider the problem
\begin{equation} \label{eq: hypercube}
\begin{aligned}
	\max \quad  & \textstyle \sum_{v \in \nodes} \cc{v}\tr \x{v}  \\
	\st  \quad & \x{v} \in \Bin^{\dimstrategic{\stage{v}}}, \quad \forall v \in \nodes,  \\
			   & \x \in \revisableset{\nrevisions},
\end{aligned}
\tag{$\text{HC}_\nrevisions$}
\end{equation}
which we refer to as the \emph{$\nrevisions$-revision hypercube problem}.
In a slight overload of notation, we also use $\hypercube{\nrevisions}{\tree, \cc}$ to denote the optimal value of the $\nrevisions$-revision hypercube problem on scenario tree $\tree$ with objective vector $\cc = \{\cc{v}\}_{v \in \nodes}$, where $\cc{v}$ has dimension $\dimstrategic{\stage{v}}$. 
In particular, the number of decisions in each stage $\set{\dimstrategic{t}}_{t \in [\nstages]}$ is encoded in $\cc$ and $\tree$.

In addition to being the simplest problem for which a $\nrevisions$-revision constraint can be imposed, the $\nrevisions$-revision hypercube problem is also the most fundamental because it appears in the formulation of all other $\nrevisions$-revision problems with binary variables. 
Hence, understanding its properties, from both the formulation and complexity points of view yields key insights into the difficulty of $\nrevisions$-revision problems and into how the strength of a formulation depends on both the strength of its base problem formulation and that of the formulation used to model the $\nrevisions$-revision requirement.

We examine the complexity of the $\nrevisions$-revision hypercube problem in the non-trivial cases where $1 \leqslant \nrevisions \leqslant \nstages - 2$.
Three parameters affect the problem scale: the number of stages $\nstages$, the number $|\nodes_t|$ of nodes per stage $t$, and the dimensions $\set{\dimstrategic{t}}_{t \in [\nstages]}$ of decisions at each stage.
We show that the problem is strongly $\NPhard$ even when $\nstages$ is fixed.
\begin{restatable}{theorem}{ThmNPhardGeneral}
\label{thm: nphard general}
    The $\nrevisions$-revision hypercube problem is strongly $\NPhard$ for any fixed integers $\nrevisions$ and $\nstages$ such that $1 \leqslant \nrevisions \leqslant \nstages-2$ when $\set{\dimstrategic{t}}_{t \in [\nstages]}$ is part of the input.
\end{restatable}

The proof of \Cref{thm: nphard general} directly follows from the special case where $\nrevisions = 1$ and $\nstages = 3$ as we argue in \Cref{proof: nphard general}.
We thus focus on this special case next. 

\begin{theorem} \label{thm: nphard K=1 T=3}
The $\nrevisions$-revision hypercube problem is strongly $\NPhard$ when $\nrevisions=1$ and $\nstages=3$ and when both $\set{|\nodes_t|}_{t \in [\nstages]}$ and $\set{\dimstrategic{t}}_{t \in [\nstages]}$ are part of the input.
\end{theorem}
\begin{proof}
Given a tree $\tree$ with $\nstages=3$, a cost vector $c$, and a profit threshold $\theta$, we consider the decision version of problem $\hypercube{1}{\tree, c}$, which consists of determining whether $\hypercube{1}{\tree, c} \geqslant \theta$.

We use a reduction from the decision version of the maximum directed cut (MAX-DICUT) problem, where we consider a directed graph $G = (V, A)$ and answer whether there is a subset of vertices $C \subset V$ such that the number of arcs in $A(C, \bar{C}):=\{a=(i,j), i\in C, j\in \bar{C}\}$ is no less than some number $\eta$.
MAX-DICUT is known to be $\NPcomplete$ \cite{papadimitriou1988}.

Consider an instance of MAX-DICUT and its associated directed graph $G=(V,A)$. 
We construct a scenario tree $\tree$ with a root at the first stage and an objective vector $\cc$ as follows.  
On the second stage of $\tree$, we define $|A|$ nodes, \textit{i.e.}, 
for each $a = (i, j) \in A$, we define a node named $v_a$  and connect it to the root $\root$ with $\cc{\root}=0$.
We let the decisions on each second-stage node be $|V|$-dimensional (\textit{i.e.}, we let $\dimstrategic{2}=|V|$) and we assign the objective vector $\cc{v_a} = \indicator{i} - \indicator{j}$ to node $v_a$. 
Here, $\indicator{i}\in \{0, 1\}^{|V|}$ is a binary vector where the entry corresponding to vertex $i$ is equal to $1$ and all others are equal to $0$.
Lastly, for each second-stage node $v_a$, we define two children $w_{a}^+$ and $w_a^{-}$, each with a 1-dimensional decision (\emph{i.e.}, $\dimstrategic{3}=1$) and objective coefficients $\cc{w_{a}^+}=+1$ and $\cc{w_a^{-}}=-1$; 
see \Cref{fig: proof of np-complete} for an example. 
This reduction is polynomial in the size of the input instance $G$ of MAX-DICUT.

Next we show that, for $\tree$ and $\cc$ constructed as above, the following two statements are equivalent:
\begin{enumerate}[label=(\Roman*)]
\item $\operatorname{MAX-DICUT}(G) \geqslant \eta$; \label{dicutcase1}
\item $\hypercube{1}{\tree, c} \geqslant |A| + \eta$.
\label{dicutcase2}
\end{enumerate}

First, we show that 
\ref{dicutcase1} implies \ref{dicutcase2}.
Take a cut $C \subseteq V$ such that $|A(C, \overbar{C})| \geqslant \eta$. 
We define the plan adjustment policy at the root node to be $\plan{\root}{1} = 0$, $\plan{\root}{2} = \mathbf{1}_C$, and $\plan{\root}{3} = 0$. 
Then, consider a second-stage node $v_a$ where:
 \begin{enumerate}[label=(\roman*)]
    \item $a = (i,j) \in A(C, \overbar{C})$. We let $\plan{v_a}{2} = \plan{\root}{2}$ and $\plan{v_a}{3}=\plan{\root, 3}$, and thus $r_{\plan}(v_a)=0$.
    Hence, the objective at $v_a$ is $\cc{v_a}\tr \x{v_a} = (\indicator{i}-\indicator{j})\tr \mathbf{1}_C = 1 - 0 = 1$.
    For the children of $v_a$, $w_{a}^+$ and $w_a^-$, we can set $\plan{w_{a}^+}{3}=1$, $\plan{w_{a}^-}{3}=0$, and thus $r_{\plan}(w_{a}^+)=1$.
    The corresponding objective at these two children is $\cc{w_{a}^+}\tr \x{w_{a}^+} + \cc{w_a^-}\tr \x{w_a^-} = 1 - 0 = 1$.
    \item $a=(i,j) \in A \backslash A(C, \overbar{C})$. 
    We take $\plan{v_{a}}{2} = \indicator{i} \neq \mathbf{1}_C$ and $\plan{v_{a}}{3} = 0$, which implies that $r_{\plan}(v_{a})=1$.
    Hence, we cannot revise at the children of $v_{a}$, where we have $\x{w_{a}^+} = \plan{w_{a}^+}{3} = \plan{v_a}{2} = \plan{w_{a}^-}{3} = \x{w_{a}^-}$ and $\cc{w_{a}^+}\tr \x{w_{a}^+} + \cc{w_{a}^-}\tr \x{w_{a}^-} = \x{w_{a}^+} - \x{w_{a}^-} = 0$.
\end{enumerate}
The above construction gives a $1$-revisable plan adjustment policy $\plan$ and a strategic policy $\x$ that is compatible with $\plan$.
The objective value of any subtree $\tree(v_a)$ is $2$ when $a \in A(C, \overbar{C})$  and $1$ when $a \notin A(C,\overbar{C})$.
Therefore, we find a solution within the revision budget with objective $2|A(C,\overbar{C})| + |A \setdiff A(C,\overbar{C})| = |A| + |A(C,\overbar{C})| \geqslant |A| + \eta$.

Second, we show that \ref{dicutcase2} implies \ref{dicutcase1}. 
Let $\x = \set{\x{v}}_{v \in \nodes}$ be an optimal strategic policy for $\hypercube{1}{\tree,c}$ with objective at least $|A| + \eta$, and $\plan = \{\plan{v, \cdot}\}_{v\in \nodes}$ be a plan adjustment policy that is compatible with $\x$ and is $1$-revisable. 
Consider the plan at the root for the second stage $\plan{\root}{2} \in \Bin^{|V|}$.
Construct the cut $C:=\{i \in V \mid \plan{\root}{2}_i = 1\}$.
We argue that $|A(C, \overbar{C})| \geqslant \eta$.
To see this, consider each $v_a$ such that $a=(i,j) \in A \backslash A(C, \overbar{C})$.
If the policy $\pi$ does not have a revision at $v_a$, the objective at $v_a$ is $\cc{v_a}\tr \x{v_a} = (\indicator{i}-\indicator{j})\tr \plan{\root}{2} \leqslant 0$, because $\indicator{i}\tr \plan{\root}{2}=1$ and $\indicator{j}\tr \plan{\root}{2}=0$ cannot hold simultaneously.
If the policy $\pi$ has a revision at $v_a$, then the objective at $v_a$ may be 1, but the policy $\pi$ could not have a revision at either of the children, and hence the total objective at the children is $0$.
In both cases, the objective obtained at $\tree(v_a)$ is at most $1$.
It is easy to see that the objective obtained by the policy $\pi$ at any subtree $\tree(v_a)$ for any $a \in A$ is at most $2$.
Thus, the total objective achieved by the policy $\pi$ is at most $2|A(C, \overbar{C})| + |A\setdiff A(C, \overbar{C})| = |A| + |A(C, \overbar{C})|$. 
By definition, the policy $\pi$ achieves an objective of at least $|A| + \eta$, so it must be true that $\eta \leqslant|A(C, \overbar{C})|$.
 
The above-described reduction thus shows that $\hypercube{1}{\tree, c}$ is $\NPhard$.
Moreover, it is strongly $\NPhard$ because the numerical parameters $\cc$ in the construction satisfy $\cc(v)_i \in \{\pm 1, 0\}$ for all $v \in \nodes$ and $i \in [d_{\stage{v}}^1]$.
\end{proof}

\begin{figure}[tbp]
    \centering
    \begin{minipage}{0.4 \textwidth}
    \raisebox{0\height}{
        \begin{tikzpicture}
        \tikzset{
            vertex/.style={circle,draw,minimum size=1.5em},
            edge/.style={->,> = latex'}
        }
        
            \node[vertex] (1) at (0,0) {1};
            \node[vertex] (2) at (2,0) {2};
            \node[vertex] (3) at (1,2) {3};
            \path [->] (1) edge node [below]{$a_1$} (2);
            \path [->, bend left=10] (2) edge node [pos=0.7, below, xshift=-2mm] {$a_2$} (3);
            \path [->, bend left=10] (3) edge node [pos=0.7, above, xshift=2mm] {$a_3$} (2);
        \end{tikzpicture}
       	}
    \end{minipage}
    \hspace{-1cm}
    \begin{minipage}{0.5 \textwidth}
        \centering
        \begin{tikzpicture}[
            scale=0.8,
            grow=right,	
            every node/.style={general node, minimum size=0.8, scale=0.95}
            ]
    		\tikzstyle{level 1}=[level distance=30mm, sibling distance=24mm]
		    \tikzstyle{level 2}=[level distance=30mm, sibling distance=12mm]
        \node{$\root$}
            child {node [inner sep=2pt, label={[yshift=-3.8em] $\cc{v_{a_3}}=[0,-1,+1]$}] {$v_{a_3}$}
                child {node[inner sep=1.0pt, label={[yshift=-6.6em] $\cc{w_{a_3}^+}=+1$}] {$w_{a_3}^+$}}
                child {node[inner sep=1.0pt] {$w_{a_3}^-$}}
            }
            child {node [inner sep=2pt, label={[yshift=-3.8em] $\cc{v_{a_2}}=[0,+1,-1]$}] {$v_{a_2}$}
                child {node[inner sep=1.0pt] {$w_{a_2}^+$}}
                child {node[inner sep=1.0pt] {$w_{a_2}^-$}}
            }
            child {node [inner sep=2pt, label={[yshift=-3.8em] $\cc{v_{a_1}}=[+1,-1,0]$}] {$v_{a_1}$}
                child {node[inner sep=1.0pt] {$w_{a_1}^+$}}
                child {node[inner sep=1.0pt, label={[yshift=-2.2em] $\cc{w_{a_1}^-}=-1$}] {$w_{a_1}^-$}}
            };
        \end{tikzpicture}
    \end{minipage}
    \vspace{-0.35cm}
    \caption{An example of the reduction from MAX-DICUT to the 1-revision hypercube problem.} 
    \label{fig: proof of np-complete}
\end{figure}

In the above results, we consider situations where $\set{\dimstrategic{t}}_{t \in [\nstages]}$ is part of the input.
Applying the transformation shown in \Cref{fig: transform multi-d to 1-d} to ensure that \Cref{assumption: 1-dimension} holds, we obtain the following corollary. 
\begin{corollary} \label[corollary]{cor: NP-hard d=1}
The $\nrevisions$-revision hypercube problem where $\dimstrategic{t} = 1$ for all $t \in [\nstages]$ is $\NPhard$ for any fixed integer $\nrevisions$ such that $1 \leqslant \nrevisions \leqslant \nstages-2$ when both $\nstages$ and $\set{|\nodes_t|}_{t \in [\nstages]}$ are part of the input.
\end{corollary}

In contrast to \Cref{cor: NP-hard d=1}, we show next that, when $\nstages$ is fixed, there is a polynomial algorithm for the problem even if the number of nodes in the tree grows.
\begin{proposition}
There is a polynomial algorithm for the $\nrevisions$-revision hypercube problem where $\dimstrategic{t} = 1$ for all $t \in [\nstages]$, $\nrevisions$ and $\nstages$ are fixed such that $1 \leqslant \nrevisions \leqslant \nstages-2$, and $\set{|\nodes_t|}_{t\in[\nstages]}$ is part of the input. 
\end{proposition}
\begin{proof}
We consider the subproblem associated with $\hypercube{\nrevisions}{\tree, \cc}$ that is a $\nrevisions$-revision hypercube problem on the subtree $\tree(v)$ with objective $\cc$ restricted to the nodes of $\tree(v)$, \emph{i.e.}, $\hypercube{k}{\tree(v), \cc{\descendant{v} \cup \{v\}}}$. 
We denote the value of this optimization problem given a plan $\plan{v, \cdot}$ active at $v$  as $G(v, k, \plan{v,\cdot})$.
Then, there is a direct dynamic programming (\DP)  recursive formula:
\begin{equation} \label{eq: dp-HC}
\begin{aligned}
        & G\big(v,k,\plan{v,\cdot}\big) = 
         \cc{v}\cdot \plan{v}{\stage{v}} \\
        + &\sum_{ u \in \children{v} } \max \left\{ G\big(u,k,\plan{v}{[\stage{v}+1:\nstages]}\big),  \max_{\plan{u,\cdot} \in \Bin^{\nstages-\stage{v}}} G\big(u,k-1,\plan{u, \cdot}\big)  \right\},
\end{aligned}
\end{equation}
where we consider the subproblem rooted at every $u \in \children{v}$. 
The first term in the $\max$ represents the case where the plan is not revised at $u$ and the second term represents the case where it is revised.
We also need the initial condition $G\big(v, 0, \plan{v,\cdot}\big)=\sum_{u \in \tree(v)} c(u) \cdot \pi(v,\stage{u})$ for each node $v \in \nodes$.
Further, the optimal objective is $\hypercube{\nrevisions}{\tree, c}= \max_{\plan{\root,\cdot}\in\Bin^\nstages} \{G(\root, \nrevisions, \plan{\root,\cdot})\}$.
Since there are $\order{|\nodes|\cdot\nrevisions\cdot 2^\nstages}$ states in the \DP{} algorithm and each state requires $\order{|\nodes|\cdot 2^\nstages}$ operations to compute, the algorithm's running time is bounded by $\order{|\nodes|^2 \cdot\nrevisions \cdot 2^{2\nstages}}$, which is polynomial in $|\nodes|$ when $\nstages$ is fixed.
\end{proof}
Although the worst-case running time of the above algorithm depends on the revision budget $\nrevisions$, we do not classify the algorithm as pseudo-polynomial since $\nrevisions$ is bounded by $\nstages$, which is clearly bounded by the input size of the problem.

\subsection{Effect of the \texorpdfstring{$\nrevisions$}{K}-revision constraint}
\label{sec: effect of revisions}

In this section, we compare our $\nrevisions$-revision approach with fully adaptive and with partially adaptive \MSP. 
In particular, we determine how different the optimal objective values of these problems can be in the worst case.
The answer to this question depends strongly on the presence of constraints in the base problem. 
These theoretical results will serve as a basis and will be in stark contrast with the results we obtain on practical problems in \Cref{sec: experiments}.

Consider an \MSP{} problem and denote its optimal value by $z_{\text{MS}}$. 
As introduced earlier, for $\nrevisions \in \Int_+$, we use $z_{\nrevisions}$ to denote the optimal value of the problem when a $\nrevisions$-revision constraint is added. 
As $z_{\text{MS}} \geqslant z_\nrevisions$, we define $\frac{|z_{\text{MS}}-z_\nrevisions|}{|z_\nrevisions|}$ as the \emph{relative loss} of $\nrevisions$-revision when $z_\nrevisions \neq 0$. 
We next show that the relative loss of $\nrevisions$-revision is bounded above for the hypercube problem.
\begin{proposition} \label[proposition]{prop: effect of revision} ~\
For the $\nrevisions$-revision hypercube problem where $1 \leqslant \nrevisions \leqslant \nstages-2$, we have $z_{\nrevisions} \geqslant \frac{\nrevisions}{\nstages-1} z_{\textnormal{MS}}$, which implies that $\frac{z_{\textnormal{MS}} - z_{\nrevisions}}{\vert z_\nrevisions \vert} \leqslant \frac{\nstages-1}{\nrevisions}-1$.
\end{proposition}
\begin{proof}
Let $\bar{x}$ be an optimal solution of the problem with no revision limit, \emph{i.e.}, $z_{\text{MS}} = \cc\tr \bar{\x}$.
Also let $\bar{\x}_L$ denote $\bar{\x}$ with all components corresponding to stages not in $L = \set{t_1, \ldots, t_{\nrevisions}}$ set to zero, \emph{i.e.}, $\bar{\x}_L(v) = \bar{\x}(v)$ for all $v \in \nodes$ such that $\stage{v} \in L$ and $\bar{x}_L(u)=0$ for all $u \in \nodes$ such that $\stage{u} \notin L$.
Then, $\bar{\x}_L$ is feasible to the $\nrevisions$-revision constraint since $|L|=\nrevisions$.
Hence, $z_{\nrevisions} \geqslant \max_{L = \set{t_1, \ldots, t_{\nrevisions}}}\{\cc\tr \bar{\x}_L \}$.
We multiply both sides by $\binom{\nstages-1}{\nrevisions}$ and write
\begin{equation*}
\begin{aligned}
      z_{\nrevisions} \cdot \binom{\nstages-1}{\nrevisions} \geqslant & \max_{L \subseteq [T]: |L|=K}\{\cc\tr \bar{\x}_L \} \cdot \binom{\nstages-1}{\nrevisions}
     \geqslant   \sum_{L \subseteq [T]: |L|=K} \cc\tr (\bar{x}|_L)  \\
     = &   \sum_{L \subseteq [T]: |L|=K} \sum_{t \in L} \cc\tr \bar{x}|_{\set{t}} 
     =   \binom{\nstages-2}{\nrevisions-1} \sum_{t \in [\nstages]} \cc\tr \bar{x}|_{\set{t}} 
     =  \binom{\nstages-2}{\nrevisions-1} \cc\tr \bar{x},
\end{aligned}
\end{equation*}
which implies $z_\nrevisions  \geqslant \frac{\nrevisions}{\nstages-1}z_{\text{MS}}$.
\end{proof}

Next, we introduce notation to describe  partially adaptive \MSP{} models. 
Consider a set of stages $L = \{t_1, \ldots, t_{\nrevisions}\} \subseteq [\nstages]$ where $1 \leqslant t_1 < \ldots < t_{\nrevisions} \leqslant \nstages$.
Let $z^{\text{PA}}_\nrevisions(L)$ denote the optimal value of the model in which revisions can only be made at stages belonging to $L$. 
We refer to $L$ as the \emph{adaptive stages} of the \emph{partially adaptive} MSP model.
We let $z_{\nrevisions}^{\text{PA}}:= \max_{L \subseteq [\nstages] : |L|=\nrevisions} z^{\text{PA}}_\nrevisions(L)$ be the objective achieved by the partially adaptive \MSP{} model with $K$ adaptive stages chosen optimally. We denote this model by $\nrevisions$-PA.
A detailed discussion of a similar version of this adaptive \MSP{}  model is conducted in \cite{basciftci2024} and \cite{kayacik2025}.
\Cref{fig: ms_ts_1revision} gives an example of the difference among the three models when $\nrevisions=1$.
It follows from the definition of $z^{\text{PA}}_\nrevisions$ that $z^{\text{PA}}_\nrevisions \leqslant z_\nrevisions$.
When $z^{\text{PA}}_{\nrevisions} \neq 0$, we define $\frac{z_{\nrevisions}-z^{\text{PA}}_\nrevisions}{|z^{\text{PA}}_{\nrevisions}|}$ as the \emph{relative value} of $\nrevisions$-revision.

\begin{figure}[tbp]
    \centering
    \begin{minipage}[t]{0.45\textwidth}
    \centering
    \begin{tikzpicture}
        \node[rounded corners=3pt, inner sep=5pt] {
        \begin{tikzpicture}[
    	scale=0.8,
            grow=right,
            every node/.style={small node}
        ]
        \tikzstyle{level 1}=[level distance=12mm, sibling distance=26mm]
        \tikzstyle{level 2}=[level distance=12mm, sibling distance=12mm]
        \tikzstyle{level 3}=[level distance=12mm, sibling distance=7mm]
        \node{}
            child {node {}
                child {node {}
                    child {node {}
                    }
                    child {node[small revised node] {}
                    }
                }
                child {node[small revised node] {}
                child {node {}
                    }
                    child {node {}
                    }
                }
            }
            child {node[small revised node] {}
                child {node {}
                child {node {}
                    }
                    child {node {}
                    }
                }
                child {node {}
                child {node {}
                    }
                    child {node {}
                    }
                }
            };
        \draw[dashed] (24mm,-25mm) -- (24mm,25mm);
        \end{tikzpicture}
        };
        \node[anchor=south west, inner sep=2pt] at (-1.7,-1.8) {$z_{\text{MS}}$};
    \end{tikzpicture}
    \end{minipage}
    \hspace*{-3em}
    \vspace*{+1em}
    \begin{minipage}[b]{0.45\textwidth}
    \centering
    \begin{minipage}[t]{0.45\textwidth}
        \centering
        \begin{tikzpicture}
            \node[rounded corners=3pt, inner sep=5pt] {
            \begin{tikzpicture}[
                scale=0.8,
                grow=right,
                every node/.style={small node}
            ]
            \tikzstyle{level 1}=[level distance=12mm, sibling distance=6mm]
            \tikzstyle{level 2}=[level distance=12mm, sibling distance=6mm]
            \tikzstyle{level 3}=[level distance=12mm, sibling distance=6mm]
    
            \node{}
                child {node {}
                    child {node {} 
                        child {node {}}
                    }
                    child {node {} 
                        child {node {}}
                    }
                    child {node {} 
                        child {node {}}
                    }
                    child {node {} 
                        child {node {}}
                    }
                };
            \end{tikzpicture}
            };
            \node[anchor=south west, inner sep=2pt] at (-1.7,-1) {$z^{\text{PA}}_1$};
        \end{tikzpicture}
    \end{minipage}

    \vspace{2em}

    \begin{minipage}[t]{0.45\textwidth}
        \centering
        \begin{tikzpicture}
        \node[rounded corners=3pt, inner sep=5pt] {
            \begin{tikzpicture}[
                scale=0.8,
                grow=right,
                every node/.style={small node}
            ]
            \tikzstyle{level 1}=[level distance=12mm, sibling distance=10mm]
            \tikzstyle{level 2}=[level distance=12mm, sibling distance=8mm]
            \tikzstyle{level 3}=[level distance=12mm, sibling distance=6mm]
    
            \node{}
                child {node {}
                    child {node {}
                        child {node {}}
                        child {node {}}
                    }
                    child {node {}
                        child {node {}}
                    }
                }
                child {node {}
                    child {node {}
                        child {node {}}
                    }
                };
            \end{tikzpicture}
        };
        \node[anchor=south west, inner sep=2pt] at (-1.5,-0.8) {$z_1$};
        \end{tikzpicture}
    \end{minipage}
\end{minipage}
\caption{Illustration of the distinct decisions allowed under partially adaptive and $\nrevisions$-revision models.}
\label{fig: ms_ts_1revision}
\begin{minipage}{\textwidth}
\footnotesize
Note. In the left panel, we describe the scenario tree of an \MSP{} in which decisions may differ at every node. To form a $1$-PA model with $t_1^*=3$ (marked by a dashed line), we only allow decisions to vary among nodes at that stage, leading to the top tree of the right panel, where the number of possible distinct decisions is reduced. To create a $1$-revision \MSP{} with revisions on the star nodes in the left panel, we restrict plans to be changed only at these nodes, leading to the bottom tree of the right panel.
\end{minipage}
\end{figure}

The proof of \Cref{prop: effect of revision} constructs feasible solutions to the partially adaptive model and uses the fact that any binary solution $\x\in\Bin^{\nodes}$ satisfying the revision constraint is a feasible solution to the $\nrevisions$-revision hypercube problem.
However, in the presence of a more complex base problem, this property may no longer hold, leading to different relationships between $z_{\nrevisions}$ and $z_{\text{MS}}$. 
To illustrate the difference, we consider an \MSP{} model for the uncapacitated lot-sizing problem, see \Cref{subsec: lot-sizing formulation}, and claim that both the relative loss and the relative value of the $1$-revision can be large.
A proof is given in \Cref{proof: large loss and value}.

\begin{restatable}{proposition}{PropLargeLossValue}
\label[proposition]{prop: large loss and value}
For stochastic uncapacitated lot-sizing problems, as demands and costs grow:
    \begin{enumerate}[label=(\Roman*)]
        \item the relative loss of $\nrevisions$-revision can be arbitrarily large;
        \item the relative value of $\nrevisions$-revision can be arbitrarily large.
    \end{enumerate}
\end{restatable}

Finally, although the $\nrevisions$-revision constraint may lead to considerable losses in objective value for some artificially constructed examples, 
its performance on practical instances is often very strong, only incurring minimal losses compared to fully adaptive models and showing substantial gains over partially adaptive ones, as we will demonstrate numerically in Section~\ref{sec: experiments}.

\subsection{Combinatorial characterization}
\label{subsec: combinatorial characterization}
In \Cref{subsec: definition}, we define the $\nrevisions$-revision constraint, or the set $\revisableset{\nrevisions}$, by introducing the notions of plans and revisions. 
Here, we present an equivalent combinatorial characterization of the $\nrevisions$-revision constraint, which requires only the strategic policy.
This characterization will help in deriving formulations for the $\nrevisions$-revision constraint that do not require the introduction of plan variables.
To obtain the desired characterization, we first introduce a special family of subtrees of the scenario tree $\tree$, which we refer to as \emph{equi-level binary embedded subtrees} (ELBE subtrees).

\begin{definition}[ELBE subtree]
Given a scenario tree $\tree$, 
let $\calS$ be a perfect binary tree with nodes $\nodes(\calS)$ and root $\subroot{\calS}$.
We say that $\calS$ is an \emph{ELBE subtree} of $\tree$ if 
\begin{enumerate}[label=(\arabic*)]
    \item
    $\nodes(\calS) \subseteq \nodes(\tree)$,  $\subroot{\calS}$ is the join of $\nodes(\calS) \setdiff \{\subroot{\calS}\}$ in $\tree$;
    \item 
    For $v, w \in \nodes(\calS)$, $v$ is an ancestor of $w$ in $\calS$ if and only if $v$ is an ancestor of $w$ in $\tree$;
    \item 
    For each sibling pair $\{p, q\}$, nodes $p$ and $q$ belong to the same stage when viewed as nodes in $\tree$.
\end{enumerate}
\end{definition}

\begin{figure}[tbp]
    \centering
    \begin{minipage}[t]{0.45\textwidth}
    \centering
    \begin{tikzpicture}[
	scale=0.8,
        grow=right,
        every node/.style={small node}
    ]
    \tikzstyle{level 1}=[level distance=12mm, sibling distance=32mm]
    \tikzstyle{level 2}=[level distance=12mm, sibling distance=16mm]
    \tikzstyle{level 3}=[level distance=12mm, sibling distance=8mm]
    \tikzstyle{level 4}=[level distance=12mm, sibling distance=3.5mm]
    \node[subtree1 node]{}
        child {node {}
            child {node[subtree1 node] {}
            child {node[subtree1 node] {}
                    child {node {}
                    }
                    child {node {}
                    }
                }
                child {node[subtree1 node] {}
                    child {node {}
                    }
                    child {node {}
                    }
                }
            }
            child {node {}
            child {node {}
                    child {node {}
                    }
                    child {node {}
                    }
                }
                child {node {}
                    child {node {}
                    }
                    child {node {}
                    }
                }
            }
        }
        child {node[subtree2 node] {}
            child {node {}
            child {node[subtree2 node] {}
                    child {node[subtree2 node] {}
                    }
                    child {node[subtree2 node] {}
                    }
                }
                child {node {}
                    child {node {}
                    }
                    child {node {}
                    }
                }
            }
            child {node[subtree1 node] {}
            child {node[subtree2 node] {}
                    child {node[subtree2 node] {}
                    }
                    child {node[subtree2 node] {}
                    }
                }
                child {node {}
                    child {node[subtree1 node] {}
                    }
                    child {node[subtree1 node] {}
                    }
                }
            }
        };
    \node[draw=none] [left=5pt] {$\tree$};
    \end{tikzpicture}
    \end{minipage}
    \begin{minipage}[t]{0.05\textwidth}
    \centering
    \vspace*{-6.5em}
    \begin{tikzpicture}
        \draw[->, thick] (0,0) -- (0.7,0);
    \end{tikzpicture}
    \end{minipage}
    \hspace*{-3em}
    \vspace*{+1em}
    \begin{minipage}[b]{0.45\textwidth}
    \centering
    \begin{minipage}[t]{0.45\textwidth}
        \centering
        \begin{tikzpicture}[
            scale=0.8,
            grow=right,
            every node/.style={small node}
        ]
        \tikzstyle{level 1}=[level distance=12mm, sibling distance=12mm]
        \tikzstyle{level 2}=[level distance=12mm, sibling distance=6mm]

        \node[subtree1 node]{}
            child {node[subtree1 node] {}
                child {node[subtree1 node] {}}
                child {node[subtree1 node] {}}
            }
            child {node[subtree1 node] {}
                child {node[subtree1 node] {}}
                child {node[subtree1 node] {}}
            };
            \node[draw=none] [left=5pt] {$\calS_1$};
        \end{tikzpicture}
    \end{minipage}
    
    \vspace{2em}
    \begin{minipage}[t]{0.45\textwidth}
        \centering
        \begin{tikzpicture}[
            scale=0.8,
            grow=right,
            every node/.style={small node}
        ]
        \tikzstyle{level 1}=[level distance=12mm, sibling distance=12mm]
        \tikzstyle{level 2}=[level distance=12mm, sibling distance=6mm]

        \node[subtree2 node]{}
            child {node[subtree2 node] {}
                child {node[subtree2 node] {}}
                child {node[subtree2 node] {}}
            }
            child {node[subtree2 node] {}
                child {node[subtree2 node] {}}
                child {node[subtree2 node] {}}
            };
            \node[draw=none] [left=5pt] {$\calS_2$};
        \end{tikzpicture}
    \end{minipage}
\end{minipage}

\caption{Illustration of ELBE subtrees.}
\label{fig: subtree example}
\end{figure}

We use $\sib{\calS}$ to denote the set of all the sibling pairs in $\calS$.
We use $\subheight{\calS}$ to denote the height of $\calS$ and let $\subtreefamily{h}{\tree}$ be the set of all ELBE subtrees of $\tree$ with height $h$.
\Cref{fig: subtree example} shows examples of ELBE subtrees of $\tree$.
Although nodes in a sibling pair belong to the same stage, nodes at the same level of an ELBE subtree do not all need to come from the same stage in the original tree; e.g., square nodes in Figure~\ref{fig: subtree example}.

\begin{definition}[$x$-inconsistent ELBE subtree]
Given a binary strategic policy $x = \{\x{v}\}_{v \in \nodes}$,  
we say that an ELBE subtree $\calS$ is \emph{$x$-inconsistent} if for any sibling pair $\{p,q\} \in \sib{\calS}$, $x(p) \neq x(q)$. 
\end{definition}

The next theorem establishes that $\nrevisions$-revisable policies $x$ do not exhibit $x$-inconsistent ELBE subtrees of large height and vice-versa.
\begin{theorem} \label{thm: subtree char}
A binary strategic policy $x$ is not $\nrevisions$-revisable if and only if there exists an $x$-inconsistent ELBE subtree with height $\nrevisions+1$.
\end{theorem}

\addtocounter{example}{-1}
\begin{example}[continued]
    The strategic policy $x$ shown in the right panel of \Cref{fig: 1-revision illustration} is not $1$-revisable as the tree itself yields an $x$-inconsistent ELBE subtree with height $2$.
\end{example}

In the proof of \Cref{thm: subtree char}, we make use of the following ancillary lemma, whose proof is given in \Cref{proof: inconsistent pair}. 
In this lemma, we define the \emph{sub-policy} of a strategic policy $\x$ to subtree $\tree(v)$ as $\x|_{\descendant{v}\cup \{v\}}$.

\begin{restatable}{lemma}{LmInconsistentPair}
\label{lm: inconsistent pair}
Let $\x$ be a binary strategic policy that is not $\nrevisions$-revisable on scenario tree $\tree$. 
Then, there must exist a pair of same-stage nodes $p, q \in \nodes$ with $\x{p} \neq \x{q}$, such that the sub-policies for both $\tree(p)$ and $\tree(q)$ are not $(\nrevisions-1)$-revisable.
\end{restatable}

\begin{proof}[Proof of \Cref{thm: subtree char}]
We first prove the reverse implication. 
We use induction on $\nrevisions$. 
First, assume that $\nrevisions=0$, and there is an $\x$-inconsistent ELBE subtree with height~$1$.
Then, there is a pair of nodes $p$, $q$ on the same level in $\tree$ such that $\x{p} \neq \x{q}$. 
Then, for any plan given at the root $\root$, there has to be one revision on either the path from $\root$ to $p$ or on the path from $\root$ to $q$. 
Hence, $\x$ is not $0$-revisable. 
Next, assume the statement holds for $\nrevisions-1$ where $\nrevisions \geqslant 1$.
Suppose there is a $(\nrevisions + 1)$-height $\x$-inconsistent subtree $\calS$.
Let the two children of the root of $\calS$ be $p$ and $q$. 
Then $\x{p} \neq \x{q}$, and the subtrees of $\calS$ rooted at $p$ and $q$ are $\nrevisions$-height $\x$-inconsistent subtrees.
By the induction hypothesis, the sub-policy $\x|_{\descendant{p}\cup \{p\}}$ and $\x|_{\descendant{q}\cup \{q\}}$ are not $(\nrevisions-1)$-revisable, which means that for any plan adjustment policy given at $\tree(p)$ (or $\tree(q)$), there is a path in $\tree(p)$ (or $\tree(q)$) that requires at least $\nrevisions$ revisions.
Thus, for any plan given for $\tree$, there has to be a revision on the path from the root $\root$ to $p$ or $q$, showing that $\x$ is not $\nrevisions$-revisable.

We next prove the direct implication using induction on $\nrevisions$.
When $\nrevisions=0$ and $x$ is not $0$-revisable, clearly there must be same-stage nodes $p$ and $q$ such that $x(p) \neq x(q)$, which forms an inconsistent ELBE subtree with $p \join q$ with height $1$.
Assume the statement holds true for $\nrevisions-1$.
Given that $x$ is not $\nrevisions$-revisable, by \Cref{lm: inconsistent pair}, there are same-stage nodes $p$ and $q$ such that $x(p)\neq x(q)$ and $x|_{\descendant{p}\cup \{p\}}$ and $x|_{\descendant{q} \cup \{q\}}$ are not $(\nrevisions-1)$-revisable for $\tree(p)$ and $\tree(q)$.
Then, there is a height-$\nrevisions$ $x$-inconsistent ELBE subtree of $\tree(p)$ and a height-$\nrevisions$ $x$-inconsistent ELBE subtree of $\tree(q)$, respectively.
Combining them yields a height-$(\nrevisions+1)$ $x$-inconsistent ELBE subtree of $\tree$. 
\end{proof}

\section{MIP formulations}	\label{sec: formulations}

In this section, we introduce two basic MIP formulations \Cref{eq: cp} and \Cref{eq: st} to model $\revisableset{\nrevisions}$ under \Cref{assumption: 1-dimension}. 
We investigate their strength, discussing their integrality gap or establishing conditions under which they are sharp or ideal.
We present in the top five blocks of \Cref{fig: summary} the roadmap we follow to derive and analyze these MIP formulations.
The content of the other blocks is discussed in \Cref{sec: methods}.

Throughout this and the following section, we develop formulations under \Cref{assumption: 1-dimension}, which is without loss of generality as we explained in \Cref{fig: transform multi-d to 1-d}.
A complete discussion and description of the formulations for the multi-dimensional case can be found in \Cref{sec: multi-d formulations}, where we also argue that the ensuing polyhedral results carry over to this setting.

\begin{figure}[tbp]
\centering
\begin{tikzpicture}[
    node distance=0.8cm and 1.2cm, 
    flow_start/.style={
        rectangle, draw, fill=blue!10,
        text width=4.5cm, minimum height=0.8cm, align=center, rounded corners=2pt,
        font=\small
    },
    formulation/.style={
        rectangle, draw, fill=green!10,
        text width=3.2cm, minimum height=0.8cm, align=center, rounded corners=2pt,
        font=\small
    },
    variant_formulation/.style={
        rectangle, draw, fill=orange!10,
        text width=3.2cm, minimum height=0.7cm, align=center, rounded corners=2pt,
        font=\small
    },
    properties/.style={ 
        rectangle, draw, fill=yellow!10, rounded corners=2pt,
        text width=3.5cm, align=left, inner sep=1.5mm,
        dashed, font=\footnotesize\itshape
    },
    variant_properties/.style={ 
        rectangle, draw, fill=yellow!10, rounded corners=2pt,
        text width=2.5cm, align=left, inner sep=1.5mm,
        dashed, font=\footnotesize\itshape
    },
    edge_info/.style={ 
        rectangle, draw, fill=purple!10,
        text width=3cm, align=center, 
        rounded corners=2pt,
        inner sep=1.5mm,
        font=\small
    },
    arrow/.style={
        draw, thick, -{Stealth[length=2.5mm, width=1.5mm]}
    },
    info_connector/.style={
        draw, gray, dashed, shorten <=0.5mm, shorten >=0.5mm
    }
]

\node[flow_start] (root) {MIP formulations for $\revisableset{\nrevisions}$};

\node[formulation, below left=2.5cm and -2.6cm of root] (cp) {Complete plan formulation \Cref{eq: cp}};
\node[formulation, below right=2.5cm and -1.7cm of root] (subtree) {Subtree formulation \Cref{eq: st}};

\draw[arrow] (root.south) -- (cp.north);
\draw[arrow] (root.south) -- (subtree.north);

\node[properties, above left=0.3cm and -0.55cm of cp] (cp_info) {
    \begin{itemize}[nosep,leftmargin=*,itemsep=0pt,topsep=0pt,partopsep=0pt]
        \item Not necessarily sharp for short trees \\ (\Cref{prop: cpf is not sharp})
        \item Empirical results suggest it is strong for tall trees \\ (\Cref{conj: cpf bound})
    \end{itemize}
};
\draw[info_connector] (cp.west) -| (cp_info.south);

\node[properties, above right=0.3cm and -1.3cm of subtree] (subtree_info) {
    \begin{itemize}[nosep,leftmargin=*,itemsep=0pt,topsep=0pt,partopsep=0pt]
        \item Strong for short trees (ideal for $\nstages=\nrevisions+2$  \Cref{thm: subtree})
        \item Can be weak for tall trees (\Cref{ex: st for tall tree})
        \item Large size
    \end{itemize}
};
\draw[info_connector] (subtree.east) -| (subtree_info.south);

\node[variant_formulation, below=0.6cm of cp] (cp+) {Size reduction technique  \Cref{eq: cp+}};
\node[variant_formulation, below=0.6cm of cp+] (cp++) {Facet-defining inequalities \Cref{eq: cp++}};

\draw[arrow] (cp.south) -- (cp+.north); 
\draw[arrow] (cp+.south) -- (cp++.north); 

\node[variant_properties, left=0.3cm of cp+] (cp+_info) {
    Size reduced from $\order{|\nodes| \nstages}$ to $\order{|\nodes|}$ (\Cref{prop: cpf size reduction})
};
\draw[info_connector] (cp+.west) -- (cp+_info.east);
\node[properties, below left=0.3cm and -0.9cm of cp++] (cp++_info) {
    Strengthen significantly for short trees \\ (\Cref{thm: cp_facets})
};
\draw[info_connector] (cp++_info.east) -| (cp++.south);

\node[variant_formulation, below=2cm of subtree, yshift=-0.5cm] (stdp) {Subtree DP formulation \Cref{eq: stdp}};
\draw[arrow] (subtree.south) -- (stdp.north) coordinate[pos=0.5](edge_midpoint);

\node[edge_info, right=0.5cm of edge_midpoint] (edge_side_info) {Efficient constraint generation algorithm for subtree constraints \Cref{eq: DP relation}};
\draw[info_connector] (edge_midpoint) -- (edge_side_info.west);

\node[properties, below right=0.3cm and -1.3cm of stdp] (stdp_info) {
    \begin{itemize}[nosep,leftmargin=*,itemsep=0pt,topsep=0pt,partopsep=0pt]
        \item Size is reduced to $\order{|\nodes|^2\nrevisions}$ \\ (\Cref{cor: stdp size})
        \item Strong for short trees (sharp for $\nstages=\nrevisions+2$, \Cref{cor: stdp sharp})
    \end{itemize}
};
\draw[info_connector] (stdp_info.west) -| (stdp.south);

\end{tikzpicture}
\vspace{+0.2cm}
\caption{Summary of MIP formulations for $\revisableset{\nrevisions}$ and their properties.}
\label{fig: summary}
\end{figure}

\subsection{Complete plan formulation (CP)}\label{subs:cpf}

We derive the first formulation by introducing variables representing the plan adjustment policy $\plan$ and the associated revision policy $r_{\plan}$. 
Then, $\revisableset{\nrevisions}$ can be described as
\begin{equation} \label{eq: cp}
\begin{aligned}
    &\x{v} = \plan{v, \stage{v}}, \quad && \forall v \in \nodes,  \\
    &\rv{v} \geqslant \plan{v,t} - \plan{\parent{v},t}, \quad && \forall v \in \nodes \setdiff \set{\root}, \forall t \in [\stage{v}:\nstages],  \\
    &\rv{v} \geqslant \plan{\parent{v},t} - \plan{v,t}, \quad && \forall v \in \nodes \setdiff \set{\root}, \forall t \in [\stage{v}:\nstages], \\
    &\textstyle \sum_{v \in \scenario} \rv{v} \leqslant \nrevisions, \quad && \forall \scenario \in \scenarioset, \\
    &\x{v} \in \Bin,\;  \plan{v,t} \in \Bin,\; \rv{v} \in \Bin, \quad && \forall v \in \nodes, \forall t \in [\stage{v}:\nstages].  \\
\end{aligned}
\tag{CP}
\end{equation}
The first constraint states that strategic policy variables $\x$ are compatible with the plan adjustment policy variables $\plan$. 
The second and third constraints link the revision policy variables $r$ with $\plan$, so that if the plan adjustment policy indicates a change of plan at a particular node, then this forces the corresponding $r$ variable to be equal to $1$.
The fourth constraint imposes that the number of revisions for each scenario is bounded by $\nrevisions$.
We call this formulation the \emph{complete plan formulation} and refer to the polytope corresponding to its LP relaxation as the \emph{complete plan polytope} $P_{CP}$.

We first show that \Cref{eq: cp} is not sharp for the $\nrevisions$-revisable set, even for simple scenario trees.

\begin{figure}[tbp]
\centering
\begin{tikzpicture}[
	scale=0.8,
    grow=right,
    every node/.style={general node}
    ]
    \node [value-1 node, label={[yshift=-6.2em]$[1, \color{mycolor1} \sfrac{1}{2} \color{black}, \color{mycolor2} \sfrac{1}{2} \color{black} ]$}]{$\root$}
        child {node [level-1 node, value-1 node, half-revised node, label={[yshift=-5.5em] $[\color{mycolor1}1 \color{black}, \color{mycolor2} \sfrac{1}{2} \color{black} ]$}] {$v_2$}
            child {node [level-2 node, value-1 node, half-revised node] {$v_6$}}
            child {node [level-2 node, half-revised node] {$v_5$}}
        }
        child {node [level-1 node, half-revised node, label={[yshift=-5.5em]$[\color{mycolor1}0 \color{black}, \color{mycolor2} \sfrac{1}{2}\color{black} ]$}] {$v_1$}
            child {node [level-2 node, half-revised node, value-1 node] {$v_4$}}
            child {node [level-2 node, half-revised node] {$v_3$}}
        };
\end{tikzpicture}
\caption{A fractional point of $P_{\textnormal{CP}}$.} 
\label{fig:cpf is not sharp when h=2}
\begin{minipage}{\textwidth}
\footnotesize
Note. Nodes filled with gray have $x$-value $1$.  Otherwise, nodes have $x$-value $0$. 
The vectors beneath the nodes represent $\plan$-values. 
The fraction of the star shaded matches its $\rv$-value, \emph{i.e.},  $\rv{v_i}=\sfrac{1}{2}$ for each node $v_i$ where $i \in [6]$.
\end{minipage}
\end{figure}

\begin{proposition} \label[proposition]{prop: cpf is not sharp}
\Cref{eq: cp} is not sharp for $\revisableset{\nrevisions}$ even when $\nstages=3$ and $\nrevisions=1$.
\end{proposition}
\begin{proof}
Consider the scenario tree presented in \Cref{fig:cpf is not sharp when h=2} where $\nstages=3$.
Also consider the solution where $\x{\root}=\x{v_2}=\x{v_4}=\x{v_6}=1$, $\x{v_1}=\x{v_3}=\x{v_5}=0$, $\rv{v_1}=\rv{v_2}=\ldots=\rv{v_6}=\sfrac{1}{2}$,
and $\plan{\root}=[1, \sfrac{1}{2}, \sfrac{1}{2}]$, $\plan{v_1}=[0, \sfrac{1}{2}]$, $\plan{v_2}=[1, \sfrac{1}{2}]$, $\plan{u}=[\x{u}]$ for all $u \in \nodes_3$.
This solution belongs to $P_{\text{CP}}$.
Its projection onto the $\x$ variables does not belong to  $\revisableset{1}$
because the tree itself is an $\x$-inconsistent subtree with height 2 and thus by \Cref{thm: subtree char}, $\x \notin \revisableset{1}$.
\end{proof}
Indeed, the argument in the proof of Proposition~\ref{prop: cpf is not sharp} establishes that, for the scenario tree of \Cref{fig:cpf is not sharp when h=2}, any $\x \in \{0,1\}^6$ can be extended into a feasible solution to $P_{\text{CP}}$.
This extension is obtained by assigning the $\plan$ and $r$ variables as in the proof, after which it is easy to verify that any choice of binary values of $x$ yields a feasible solution to \Cref{eq: cp}. 
Consequently the optimal objective of the LP relaxation of $\operatorname{HC}_{1}$, where \Cref{eq: cp} is used to model the $1$-revision constraint, $z_1^{\LP}$, equals $z_2 = \max_{x \in \Bin^6} c\tr x = z_{\textrm{MS}}$.
Let $\bar{x}$ be an optimal solution to the latter problem.
Define
$\dot{x}=(\bar{x}(v_1), \bar{x}(v_2), \bar{x}(v_3), \bar{x}(v_4), 0, 0)$, $\breve{x}=(\bar{x}(v_1), \bar{x}(v_2), 0, 0, \bar{x}(v_5), \bar{x}(v_6))$, and $\ddot{x}=(0, 0, \bar{x}(v_3), \bar{x}(v_4), \bar{x}(v_5), \bar{x}(v_6))$.
Then, we write
\begin{equation*}
\begin{aligned}
    z_1 \geqslant \max \big\{ c\tr \dot{x}, c\tr \breve{x}, c\tr \ddot{x} \big\}
     \geqslant \frac{1}{3} \cdot \sum_{i=1}^6 2 c(v_i) \bar{x}(v_i)
     = \frac{2}{3} z_2,
\end{aligned}
\end{equation*}
where the first inequality holds because $\dot{x}$, $\breve{x}$, and $\ddot{x}$ are all $1$-revisable solutions, the second inequality holds because the maximum of several quantities is at least as large as their average, and the equality is because of the definition of $\bar{x}$.
It follows that $z_1^{\LP} \leqslant \sfrac{3}{2}\, z_1$, showing that the integrality gap is bounded in this case.

Although the above argument is specific to the case where $\nrevisions=1$ and $\nstages=3$, we have observed in our numerical experiments that the integrality gap appears to also be bounded for other values of $\nrevisions$ and $\nstages$. 
In fact, we make the following conjecture.
\begin{conjecture} \label{conj: cpf bound}
For the $\nrevisions$-revision hypercube problem \Cref{eq: hypercube} with $\revisableset{\nrevisions}$ formulated by \Cref{eq: cp}, $z_{\nrevisions}^{\LP} \leqslant \frac{2\nrevisions+1}{\nrevisions+1} z_\nrevisions$.
\end{conjecture}

\Cref{conj: cpf bound} suggests that the objective value of the LP relaxation is close to the true optimum, with \Cref{fig:cpf is not sharp when h=2} providing an example that achieves the worst-case bound.

\subsection{Subtree formulation}\label{subs:subtree}

Our second MIP formulation for $\revisableset{\nrevisions}$ is built solely on the strategic policy variables $\x$ and does not require the introduction of auxiliary variables $\plan$ and $r$. 
\Cref{thm: subtree char} establishes that a strategic policy $x$ is $\nrevisions$-revisable if and only if there does not exist an $x$-inconsistent ELBE subtree with height $\nrevisions+1$.
For any ELBE subtree with height $\nrevisions+1$, there are $2^{\nrevisions+1}-1$ sibling pairs. 
Hence, a strategic policy $\x$ is $\nrevisions$-revisable if and only if for any ELBE subtree $\calS$ with height $\nrevisions+1$, there are at most $2^{\nrevisions+1}-2$ sibling pairs $\{u, v\} \in \sib{\calS}$ such that $\x{u}\neq \x{v}$.
This condition can be expressed as
\begin{equation} \label{eq: st abs}
\begin{aligned}
    &\textstyle \sum_{\{u,v\} \in \sib{\calS}} \big\vert \x{u} - \x{v} \big\vert \leqslant 2^{\nrevisions+1} - 2, \quad && \forall \calS \in \subtreefamily{\nrevisions+1}{\tree}.
\end{aligned}
\end{equation}
This formulation is not linear due to the use of absolute values.
To linearize it, 
we introduce the concept of an \emph{orientation}.
For a given ELBE subtree $\calS$, an oriented ELBE subtree, denoted $\vec{\calS}$, is formed by assigning a specific order $(u,v)$ to each unordered sibling pair $\{u,v\} \in \sib{\calS}$. 
Let $\operatorname{ori}(\calS)$ denote the set of all possible oriented versions of $\calS$.
We obtain
\begin{equation} \label{eq: st}
\begin{aligned}
    &\textstyle \sum_{(u, v) \in \sib{\vec{\calS}}} \big( \x{u} - \x{v} \big) \leqslant 2^{\nrevisions+1} - 2, \quad && \forall \calS \in \subtreefamily{\nrevisions+1}{\tree},\: \forall\vec{\calS} \in \operatorname{ori}(\calS) \\
    &\x{v} \in \Bin, \quad && \forall v \in \nodes.
\end{aligned}
\tag{ST}
\end{equation}
We refer to formulation \Cref{eq: st} as the \emph{subtree formulation} and to each nontrivial constraint in \Cref{eq: st} as a \emph{subtree constraint}.
Further, we refer to the polytope obtained by relaxing integrality requirements in \Cref{eq: st} as the \emph{subtree polytope} $P_{\text{ST}}$.

We next argue that subtree constraints are strong.

\begin{theorem} \label{thm: subtree facet}
Subtree constraints define facets of $\conv{\revisableset{\nrevisions}}$ for any $1 \leqslant  \nrevisions \leqslant \nstages-2$.
\end{theorem} 

To establish this result, we consider an ELBE subtree of height $\nrevisions+1$, $\calS$, with orientation $\vec{\calS}$.
The corresponding subtree constraint is
\begin{equation} \label{eq: st constr facet}
    \sum_{(u,v) \in \sib{\vec{\calS}}} (\x{u}-\x{v}) \leqslant 2^{\nrevisions+1}-2.
\end{equation}
By our previous argument, \Cref{eq: st constr facet} is valid for $\revisableset{\nrevisions}$,
hence it is valid for $\conv{\revisableset{\nrevisions}}$.
Let $F_{\text{ST}}$ be the face of $\conv{\revisableset{\nrevisions}}$ defined by \Cref{eq: st constr facet}.
The following lemma, whose proof can be found in \Cref{proof: st facet - x}, provides ways of constructing points on $F_{\text{ST}}$.

\begin{restatable}{lemma}{LmSTFacetX}
\label[lemma]{lm: st facet - x}
Given an oriented ELBE subtree $\vec{\calS}$ of height $\nrevisions+1$ where $1 \le \nrevisions \le \nstages-2$,
\begin{enumerate}[label=(\Roman*)]
    \item For any $(u,v) \in \sib{\vec{\calS}}$ and $\theta \in \{0,1\}$, there is a solution $x^{u,v}_\theta \in F_{\textnormal{ST}}$ where $x^{u,v}_\theta(u')=1$ and $x^{u,v}_\theta(v')=0$ for all $(u',v') \in \sib{\vec{\calS}} \setdiff \{(u,v)\}$ and where all the other entries are equal to $\theta$. \label{LmSTFacetX-PartI}
    \item For any node $w \notin \nodes(\calS) \setdiff \{\subroot{\calS}\}$, there is $(u,v) \in \sib{\vec{\calS}}$ such that the point $x^{u,v}_0+e_w$ obtained by switching the value of $x^{u,v}_0(w)$ from $0$ to $1$ belongs to $F_{\textnormal{ST}}$. \label{LmSTFacetX-PartII}
\end{enumerate}
\end{restatable}

\begin{proof}[Proof of \Cref{thm: subtree facet}] 
Consider any inequality that defines $F_{\textnormal{ST}}$:
\begin{equation}   \label{eq: st facet - any ineq}
\sum_{w \in \nodes}\alpha_w\x{w} \leqslant \eta.
\end{equation}
We will show that this inequality is a scalar multiple of \Cref{eq: st constr facet}.

First, consider any node $w \notin \nodes(\calS) \setdiff \{\subroot{\calS}\} $.
We show that $\alpha_w=0$.
By \Cref{lm: st facet - x}, there is $(u,v) \in \vec{\calS}$ such that $x^{u,v}_0$ and $x^{u,v}_0 + e_w \in F_{\textnormal{ST}}$.
Since \Cref{eq: st facet - any ineq} defines $F_{\text{ST}}$, we plug these two solutions into the equality of \Cref{eq: st facet - any ineq} to obtain
$\sum_{(u',v')\in\sib{\vec{\calS}} \setdiff \{(u,v)\}}\alpha_{u'} + \alpha_w = \eta$ and $\sum_{(u',v')\in\sib{\vec{\calS}} \setdiff \{(u,v)\}}\alpha_{u'}=\eta$.
Subtracting one from the other implies $\alpha_w=0$. 

Second, we show that $\alpha_u=-\alpha_{v}=\alpha_{\bar{u}}$ for any $(u,v)$ and $(\bar{u},\bar{v})$ in $\sib{\vec{\calS}}$.
\Cref{lm: st facet - x} provides $x^{u,v}_0$ and $x^{u,v}_1$, both of which must satisfy \Cref{eq: st facet - any ineq} at equality.
It follows that $\sum_{(u',v')\in\sib{\vec{\calS}} \setdiff \{(u,v)\}}\alpha_{u'} = \eta$ and $\sum_{(u',v')\in\sib{\vec{\calS}} \setdiff \{(u,v)\}}\alpha_{u'} + \alpha_u + \alpha_v = \eta$ since $\alpha_w=0$ for any node $w \notin \nodes(\calS) \setdiff \{\subroot{\calS}\}$.
Subtracting one equation from the other yields $\alpha_u + \alpha_v = 0$, so that $\alpha_u = -\alpha_v$ for every $(u,v) \in \sib{\vec{\calS}}$.
For every $(\bar{u},\bar{v}) \neq (u,v) \in \sib{\vec{\calS}}$, both $x^{u,v}_0$ and $x^{\bar{u},\bar{v}}_0$ must satisfy \Cref{eq: st facet - any ineq} at equality. 
Subtracting the equation for $x^{\bar{u},\bar{v}}_0$ from that for $x^{u,v}_0$ yields $\alpha_u = \alpha_{\bar{u}}$.
It follows that there exists $\lambda$ such that $\alpha_u = \lambda$ and $\alpha_v=-\lambda$ for all $(u,v) \in \sib{\vec{\calS}}$.
Plugging $x^{u,v}_0$ for any $(u,v) \in \sib{\vec{\calS}}$ in \Cref{eq: st facet - any ineq} yields that $\lambda (2^{\nrevisions+1}-2) = \eta$. 
This implies that $\alpha_u=-\alpha_v=\sfrac{\eta}{ (2^{\nrevisions+1}-2)}$, \textit{i.e.}, \Cref{eq: st facet - any ineq} is a scalar multiple of \Cref{eq: st constr facet}.
\end{proof}

In addition to having fewer variables than \Cref{eq: cp}, \Cref{eq: st} also has the advantage that it is an ideal formulation for trees with at most $(\nrevisions+2)$ stages.
\begin{theorem} \label{thm: subtree}
Formulation \Cref{eq: st} is ideal for all revision budgets $\nrevisions \geqslant 1$ and scenario trees for which $\nstages \leqslant \nrevisions + 2$.
\end{theorem}
\begin{proof}
    The case $\nstages < \nrevisions + 2$ is trivial as the polytope $\conv{\revisableset{\nrevisions}}$ reduces to $[0,1]^{\nodes}$, and \Cref{eq: st} does not contain subtree constraints as there are no ELBE subtrees with height $\nrevisions+1$.
    For the case $\nstages = \nrevisions + 2$, any ELBE subtree with height $\nrevisions + 1 = \nstages - 1$ must contain nodes from all stages of $\tree$.
    In other words, for any $\calS \in \subtreefamily{\nrevisions+1}{\tree}$ and any sibling pair $\{p,q\}\in\sib{\calS}$, the parent of $p$ and $q$ is the same node in $\calS$ and $\tree$.
    Therefore, we further assume that every non-leaf node $v$ of $\tree$ has at least 2 children, otherwise $v$ does not show up in any subtree constraints.
    Define
    \begin{equation} \label{eq: Delta v}
        \Delta_x(v) := \max_{\calS \in \subtreefamily{\nstages-\stage{v}}{\tree(v)}} \bigg\{\sum_{\{p,q\} \in \sib{\calS}} \big|\x{p} - \x{q}\big|\bigg\},
    \end{equation}
    where $\subtreefamily{\nstages-\stage{v}}{\tree(v)}$ is the set of all ELBE subtrees of $\tree(v)$ that have nodes on every stage from $\stage{v}$ to $\nstages$.
    Therefore, $x \in P_{\text{ST}}$ is equivalent to $\Delta_x(\root) \leqslant 2^{\nrevisions+1} - 2$. 
    Further, we have the recurrence relation
    \begin{equation} \label{eq: Delta v recurrence}
        \Delta_x(v) = \max_{p, q \in \children{v}} \big\vert \x{p}-\x{q} \big\vert + \Delta_x(p) + \Delta_x(q).
    \end{equation}
    For convenience, we also introduce
    \begin{equation} \label{eq: delta v}
        \delta_x(v):=\Delta_x(v)-(2^{\nstages-\stage{v}}-2).
    \end{equation}
    It is clear that $\delta_x(v) \leqslant 1$ for any $v \in \nodes$ as the summation defining $\Delta_x(v)$ in \eqref{eq: Delta v} contains $(2^{\nstages-\stage{v}}-1)$ terms of value at most $1$.
    Furthermore, $x \in P_{\text{ST}}$ is equivalent to $\delta_x(\root)\leqslant 0$, and 
    \begin{align}\delta_x(v)=\max_{p,q\in\children{v}} \big\{ \big\vert \x{p}-\x{q} \big\vert + \delta_x(p) + \delta_x(q) \big\} - 2. \label{eqn: delta-recurs}
    \end{align}

    We next show that the polytope $P_{\text{ST}}$ is integral by contradiction.
    Suppose that there are scenario trees for which $P_{\textnormal{ST}}$ is not integral when $\nrevisions=\nstages-2$.
    Among all such trees, let $\tree$ be one that has the minimum $\nstages$ and, among all trees having minimum $\nstages$, one where $|\nodes|$ is minimum, \emph{i.e.}, we consider a smallest counterexample.
    Let $\bar{x}$ be a fractional vertex of $P_{\text{ST}}$.
First, we argue that $\delta_{\bar{x}}(v)\geqslant 0$ for all $v \in \nodes$.
Assume by contradiction that there exists a node $v$ such that $\delta_{\bar{x}}(v)<0$.
Then the subtree constraint of any oriented ELBE subtree containing $v$ has a positive slack in its corresponding subtree constraint.
Since any fractional $\bar{x}$-value within the subtree $\tree(v)$ could then be shifted upward or downward while preserving feasibility, no such fractional component can exist in $\tree(v)$.
Deleting all vertices $v$ with $\delta_{\bar{x}}(v)<0$ and their descendants, the resulting $\bar{x}$ restricted to the reduced tree remains a fractional vertex of the corresponding subtree polytope.
This is a contradiction to the assumption that $\tree$ is a smallest counterexample.
Next, since $\bar{x} \in P_{\text{ST}}$ implies $\delta_{\bar{x}}(\root)\leqslant 0$, we must have $\delta_{\bar{x}}(\root) = 0$.

Our goal is to show that we can construct points $\bar{x}^+ \neq \bar{x}^- \in P_{\text{ST}}$ such that $\bar{x}=\sfrac{1}{2} \, \bar{x}^+ + \sfrac{1}{2} \, \bar{x}^-$, 
leading to a contradiction.
We start with a claim that constructs points restricted to a subtree $\tree(v)$ for $v \in \nodes_2$, where $\nodes_2$ denotes the set of second-stage nodes.
The proof of this claim is provided in \Cref{proof: st integrality perturb delta}.
Intuitively, \Cref{claim: st integrality - perturb delta} allows us to treat the subtree rooted at any second-stage node as a single ``unit'' whose $\delta_{\bar{x}}$-value can be slightly increased or decreased.
    
\begin{restatable}{claim}{ClaimSTIntegralityPerturbDelta}
\label{claim: st integrality - perturb delta}
    For any $v\in \nodes_2$, and any $y \in [0,1]^{\descendant{v}}$ such that $y$ has at least one fractional entry and $0 < \delta_y(v) < 1$, there exists an $\bar{\varepsilon} > 0$, such that for any $0 < \varepsilon < \bar{\varepsilon}$, we can construct distinct vectors $y^\plus$ and $y^\minus \in [0,1]^{\descendant{v}}$ satisfying $y = \sfrac{1}{2}\, y^+ + \sfrac{1}{2}\, y^-$ and  $\delta_{y^\pm}(v) = \delta_y(v) \pm \varepsilon$.
\end{restatable}

    We now define a \emph{tight pair} to be a pair of second-stage nodes $(p,q)$ such that $0=\delta_{\bar{x}}(\root) = |\bar{x}(p)-\bar{x}(q)|+ \delta_{\bar{x}}(p)+\delta_{\bar{x}}(q)-2$.
    Let $\calE$ denote the set of all such tight pairs.
    If a second-stage node $v$ does not appear in any tight pair, then there cannot exist any fractional $\bar{x}$-value node in $\{v\} \cup \descendant{v}$.
    Indeed, if a fractional value exists within $\tree(v)$, we can perturb that fractional value and maintain feasibility, even without the knowledge of \Cref{claim: st integrality - perturb delta}.
    To make this precise, suppose $\bar{x}(u) \in (0,1)$ for some $u\in \{v\} \cup \descendant{v}$. Let $\bar{x}^+$ and $\bar{x}^-$ be obtained from $\bar{x}$ by shifting $\bar{x}(u)$ by $+\varepsilon$ and $-\varepsilon$, 
    \emph{resp.}, leaving all other coordinates unchanged.
    Since $v$ appears in no tight pair, every subtree inequality involving $u$ holds strictly at $\bar{x}$.
    Hence, for all sufficiently small $\varepsilon>0$, both $\bar{x}^{+}$ and $\bar{x}^{-}$ are feasible.
    As $\bar{x}=\sfrac{1}{2}(\bar{x}^+ + \bar{x}^-)$ with $\bar{x}^+\neq \bar{x}^-$, $\bar{x}$ is not a vertex, a contradiction.
    By the same reasoning used earlier, we can delete $v$ and all its descendants from $\tree$, and the restriction of $\bar{x}$ to the reduced tree remains a vertex of the corresponding subtree polytope.
    Thus, every second-stage node must participate in at least one tight pair, as we assume $\tree$ is a smallest counterexample.
    
    Next, observe that if two second-stage nodes $p$ and $q$ share the same $\bar{x}$-value (\textit{i.e.} $\bar{x}(p)=\bar{x}(q)$) but have different $\delta_{\bar{x}}$ values, say $\delta_{\bar{x}}(p) < \delta_{\bar{x}}(q)$, then node $p$ cannot be part of any tight pair, which is impossible.
    Therefore, two second-stage nodes with the same $\bar{x}$-value must also have the same $\delta_{\bar{x}}$-value. 

    The above discussion implies that the structure we consider is one in which every second-stage node participates in at least one tight pair and all nodes with the same $\bar{x}$-value have the same $\delta_{\bar{x}}$-value.
    Let $\bar{\varepsilon}$ denote a positive value for which \Cref{claim: st integrality - perturb delta} holds for all $v \in \nodes_2$ with $\delta_{\bar{x}}(v) \in (0,1)$.
    We then set 
    \begin{equation} \label{eq: define epsilon}
    \begin{aligned}
        \varepsilon := \min \Bigg\{ 
        \bar{\varepsilon},  &\min_{v\in \nodes_2, \bar{x}(v) \in (0,1)} \big\{ \min \{\bar{x}(v), 1-\bar{x}(v)\} \big\}, \\
        &\frac{1}{4} \min_{\substack{u,v \in \nodes_2,  (u,v) \notin \calE}} \Big\{ 2- \big(|\bar{x}(u)-\bar{x}(v)| + \delta_{\bar{x}}(u) + \delta_{\bar{x}}(v) \big) \Big\} 
        \Bigg\}, 
    \end{aligned}
    \end{equation}
    that is, we choose $\varepsilon$ sufficiently small so that \Cref{claim: st integrality - perturb delta} remains applicable and, at the same time, less than both the minimum distance from any fractional component of $\bar{x}$ to its bounds and one quarter of the minimum slack among all untight pairs $(u,v) \notin \calE$.
    There are three cases to consider.
    \begin{enumerate}[label=(\roman*)]
        \item All $\bar{x}(v) \in \nodes_2$ are integral, \emph{i.e.}, the fractional $\bar{x}$-values are only on the other stages.
        Then $\nodes_2$ contains at most two types of nodes: those with $\bar{x}$-value $0$ and those with $\bar{x}$-value $1$.
        Assume $p \in \nodes_2$ is such that $\bar{x}(p)=0$.
        By previous arguments, all $\bar{x}$-value-$0$ nodes in $\nodes_2$ share the same $\delta_{\bar{x}}$-value.
        We first argue that if $\delta_{\bar{x}}(p)\in \{0,1\}$ then there is no fractional $\bar{x}$-values in $\tree(p)$. 
        If $\delta_{\bar{x}}(p)=1$, then every two $\bar{x}$-value-$0$ nodes form a tight pair.
        However, no node in $\descendant{p}$ (or in the set of descendants of any $\bar{x}$-value-$0$ node) can be fractional—otherwise, perturbing its value would leave $\delta_{\bar{x}}(p)$ unchanged, contradicting the vertex property of $\bar{x}$.
        Hence, there must exist $q \in \nodes_2$ such that $\bar{x}(q)=1$ and $\delta_{\bar{x}}(q)=0$, which forms a tight pair with $p$, and there exists some fractional values within $\tree(q)$.
        This implies that $\bar{x}|_{\descendant{q}\cup\{q\}}$ is a fractional point of the subtree polytope corresponding to $\tree(q)$, and hence $\bar{x}|_{\descendant{q}\cup\{q\}}$ is not a vertex as $\tree$ is a smallest counterexample.
        Thus, there exist $\bar{x}^\pm$ such that $\delta_{\bar{x}^\pm}(q)=0$, which implies $\delta_{\bar{x}^\pm}(\root)=0$.
        The case with a second-stage node having $\bar{x}$-value $1$ and $\delta_{\bar{x}}$-value $1$ follows by symmetric arguments.
        
        It thus remains to consider the case in which there exist $p$ and $q \in \nodes_2$ such that $\bar{x}(p)=0$, $\bar{x}(q)=1$, with $\delta_{\bar{x}}(p), \delta_{\bar{x}}(q) \in (0,1)$.
        We use \Cref{claim: st integrality - perturb delta} by letting $y=\bar{x} |_{\descendant{v}}$ to create $\bar{x}^\pm$ for each $v \in \nodes_2$.
        If $\bar{x}(v) = 0$,  
        then we set $\bar{x}^\pm$ restricted on $\tree(v)$, $\bar{x}^\pm |_{\descendant{v}}$, such that it leads to $\delta_{\bar{x}^\pm|_{\descendant{v}}}(v) = \delta_{\bar{x}|_{\descendant{v}}}(v) \pm \varepsilon$, where $\varepsilon$ is defined as \Cref{eq: define epsilon}.
        If $\bar{x}(v) = 1$,  we set $\bar{x}^\pm |_{\descendant{v}}$ such that $\delta_{\bar{x}^\pm|_{\descendant{v}}}(v) = \delta_{\bar{x}|_{\descendant{v}}}(v) \mp \varepsilon$.
        Then, $|\bar{x}^\pm(p) - \bar{x}^\pm(q)| + \delta_{\bar{x}^\pm}(p) + \delta_{\bar{x}^\pm}(q) = 1 + \delta_{\bar{x}}(p) \pm \varepsilon + \delta_{\bar{x}}(q) \mp \varepsilon = 1 + \delta_{\bar{x}}(p) + \delta_{\bar{x}}(q) = 2$.
        Thus, each tight pair $(p,q) \in \calE$ remains feasible after the perturbation.
        For the other pairs, $(p,p')$ or $(q,q')$ where $p'$ and $q'$ have the same $\bar{x}$-values as $p$ and $q$, respectively, they are not tight pairs for $\bar{x}$.
        Thus, they will remain not tight under $\bar{x}^\pm$ because of the way we choose $\varepsilon$.
        Hence, $\delta_{\bar{x}^\pm}(\root) = \delta_{\bar{x}}(\root)$.
        
        \item There is a $v \in \nodes_2$ such that $\bar{x}(v) \in (0,1)$ and $\delta_{\bar{x}}(v) \in (0,1)$.
        It is clear that $\bar{x}(u) \neq \bar{x}(v)$ for all $u$ such that $(u,v) \in \calE$ as $\delta_{\bar{x}}(v) < 1$.
        We claim that either $\bar{x}(u) < \bar{x}(v)$ for all $u$ such that $(u,v) \in \calE$, or $\bar{x}(u) > \bar{x}(v)$ for all $u$ such that $(u,v) \in \calE$.
        Otherwise, if $\bar{x}(u) < \bar{x}(v) < \bar{x}(u')$ for some $u, u'$ such that $(u,v) \in \calE$ and $(u',v) \in \calE$, then we have
        \begin{equation} \label{eq: st integrality - tight pair}
        \begin{aligned}
        0 = \delta_{\bar{x}}(\root)
        & \geqslant
        \bar{x}(u') - \bar{x}(u) + \delta_{\bar{x}}(u') + \delta_{\bar{x}}(u) - 2 \\
        & = \bar{x}(u') - \bar{x}(v) + \bar{x}(v) - \bar{x}(u) + \delta_{\bar{x}}(u') + 
        \delta_{\bar{x}}(u) - 2 \\
        & = \big( \bar{x}(u') - \bar{x}(v) + \delta_{\bar{x}}(u') - 2 \big) + \big( \bar{x}(v) - \bar{x}(u) + \delta_{\bar{x}}(u)  - 2 \big) + 2 \\
        & = 2 \delta_{\bar{x}}(\root) - 2 \delta_{\bar{x}}(v) + 2 > 2 \delta_{\bar{x}}(\root) = 0,
        \end{aligned}
        \end{equation}
        which is impossible.
        The last inequality holds as $\delta_{\bar{x}}(v) < 1$.
        We now only focus on the case where $\bar{x}(u) < \bar{x}(v)$ for all $u$ such that $(u,v) \in \calE$ here, as the other case is similar.
        In this case, we create $\x^\pm$ in three steps. 
        First, we construct $\x^\pm|_{\descendant{v}}$ on $\tree(v)$ with $\delta_{x^\pm|_{\descendant{v}}}(v) = \delta_{\bar{x}}(v) \pm \varepsilon$ by \Cref{claim: st integrality - perturb delta}. 
        Second, we let $\bar{x}^\pm(v) = \bar{x}(v) \mp \varepsilon$.
        Third, for all the other nodes $w \notin \descendant{v} \cup \{v\}$, we set $\bar{x}^\pm(w) = \bar{x}(w)$.
        This leads to 
        $
        |\bar{x}(u) - \bar{x}^\pm(v)| + \delta_{\bar{x}}(u) + \delta_{\bar{x}^\pm}(v) - 2 = |\bar{x}(u) - \bar{x}(v)| + \delta_{\bar{x}}(u) + \delta_{\bar{x}}(v) - 2 
        = \delta_{\bar{x}}(\root),
        $
        for all $u$ such that $(u,v) \in \calE$.
        Hence, we have $\delta_{\bar{x}^\pm}(\root) = \delta_{\bar{x}}(\root)$.
        
        \item There are nodes $v \in \nodes_2$ such that $\bar{x}(v) \in (0,1)$ and for any such node $v$, $\delta_{\bar{x}}(v) \in \{0,1\}$.
        We assume $\delta_{\bar{x}}(v)=1$ for all such $v$ because otherwise $v$ does not appear in any tight pair, which is not possible.
        Further, every such node $v$ must share the same $\bar{x}$-value, otherwise, if $v'$ is such that $\delta_{\bar{x}}(v')=1$ and $\bar{x}(v') \neq \bar{x}(v)$, then $\delta_{\bar{x}}(\root) \geqslant |\bar{x}(v) - \bar{x}(v')|+\delta_{\bar{x}}(v) + \delta_{\bar{x}}(v') -2 > 0$, which is impossible.
        First, we adjust the $\bar{x}$-value on each $v$ such that $\delta_{\bar{x}}(v)=1$ by $\bar{x}^\pm(v) = \bar{x}(v) \mp \varepsilon$.
        Then, for any $u$ such that $(u,v) \in \calE$ such that $\bar{x}(u) < \bar{x}(v)$ (\emph{resp.} $\bar{x}(u) > \bar{x}(v)$), we have $\delta_{\bar{x}}(u) \in (0,1)$ and we use \Cref{claim: st integrality - perturb delta} to adjust the $\bar{x}$-value on $\tree(u)$ such that $\delta_{\bar{x}^\pm}(u) = \delta_{\bar{x}}(u) \pm \varepsilon$ (\emph{resp.} $\delta_{\bar{x}^\pm}(u) = \delta_{\bar{x}}(u) \mp \varepsilon$).
        
        To show the above construction maintains feasibility, \emph{i.e.}, $\delta_{\bar{x}^\pm}(\root) \leqslant 0$, we show that for any $u$ such that $(u,v)\in\calE$ and $\bar{x}(u) < \bar{x}(v)$, if $u$ forms a tight pair with another node $v' \neq v$ and $v' \in \nodes_2$, then $(v',v) \in \calE$ and $\bar{x}(v') \geqslant \bar{x}(v)$.
        To see this, we first argue that $\bar{x}(v') \geqslant \bar{x}(v)$.
        Because if otherwise, either $\bar{x}(v')<\bar{x}(u)$, thus $(v',u) \in \calE$, $(u,v) \in \calE$, and $\bar{x}(v') < \bar{x}(u) < \bar{x}(v)$, which is impossible by an argument similar to \Cref{eq: st integrality - tight pair}, or $\bar{x}(v') \geqslant \bar{x}(u)$, which implies $\bar{x}(v)-\bar{x}(u)+\delta_{\bar{x}}(v)+\delta_{\bar{x}}(u) > \bar{x}(v')-\bar{x}(u)+\delta_{\bar{x}}(v')+\delta_{\bar{x}}(u)$ as $1=\delta_{\bar{x}}(v)\geqslant \delta_{\bar{x}}(v')$, which is also impossible.
        Thus, we only have to show $(v',v) \in \calE$.
        This is because 
        \begin{equation}
        \begin{aligned}
        & \bar{x}(v')-\bar{x}(v)+\delta_{\bar{x}}(v')+\delta_{\bar{x}}(v) \\
        = &\bar{x}(v')-\bar{x}(u)+\delta_{\bar{x}}(v')+\delta_{\bar{x}}(u) - (\bar{x}(v)-\bar{x}(u) + \delta_{\bar{x}}(u) + \delta_{\bar{x}}(v)) + 2\delta_{\bar{x}}(v) \\
        = & 2 - 2 + 2\delta_{\bar{x}}(v) = 2,
        \end{aligned}
        \end{equation}
        where the second equality holds as $(v',u)$ and $ (u,v)\in\calE$.
        Hence, when we adjust $\delta_{\bar{x}^\pm}(u)$, any node $v'$ that forms a tight pair with $u$ is also adjusted to maintain feasibility.
        For $u$ such that $(u,v)\in \calE$ and $\bar{x}(u)>\bar{x}(v)$, similar arguments hold. 
        Therefore, we have $\delta_{x^\pm}(\root) = \delta_{\bar{x}}(\root)$.
    \end{enumerate}
    Thus, for all cases, there exists $\bar{x}^+ \neq \bar{x}^- \in P_{\text{ST}}$ with $\bar{x}=\sfrac{1}{2}\, \bar{x}^+ + \sfrac{1}{2}\, \bar{x}^-$, which violates the fact that $\bar{x}$ is a vertex and completes the proof.
\end{proof}

\begin{figure}[tbp]
    \centering
    \begin{minipage}[c]{0.42\textwidth}
        \centering
        \begin{tikzpicture}[
            grow=right,
            every node/.style={general node, minimum size=0.4cm}
            ]
        \tikzstyle{level 1}=[level distance=16mm, sibling distance=16mm]
        \tikzstyle{level 2}=[level distance=14mm, sibling distance=12mm]
        \tikzstyle{level 3}=[level distance=14mm, sibling distance=6mm]
        \node{}
            child {node {}
            child {node[value-1 node] {}
                child {node {}}
                child {node[value-half node]  {}}
            }
            }
            child {node[value-1 node] {}
            child {node[value-1 node] {}
                child {node {}}
                child {node[value-half node] {}}
            }
            }
            child {node[value-half node]  {}
            child {node[value-half node]  {}
                child {node {}}
                child {node[value-1 node] {}}
            }
            };
        \end{tikzpicture}
    \end{minipage}%
    \hfill
    \begin{minipage}[c]{0.56\textwidth}
        \centering
        \begin{tikzpicture}[
        grow=right,
        every node/.style={general node, minimum size=0.4cm}
        ]
        \tikzstyle{level 1}=[level distance=15mm, sibling distance=28mm]
        \tikzstyle{level 2}=[level distance=15mm, sibling distance=14mm]
        \tikzstyle{level 3}=[level distance=15mm, sibling distance=6mm]
        \node{}
            child {node {}
            child {node {}
                child {node {}}
                child {node[value-23rd node] {}}
            }
            child {node[value-23rd node]{} 
                child {node {}}
            child {node[value-23rd node, label={[xshift=+3em, yshift=-1em]\ldots \ldots}] {}}
            }
            }
            child {node[value-23rd node] {}
            child {node {}
                child {node {}}
                child {node[value-23rd node] {}}
            }
            child {node[value-23rd node]  {}
                child {node {}}
                child {node[value-23rd node] {}}
            }
            };
        \end{tikzpicture}
        \vspace{+0.2cm}
    \end{minipage}
    \caption{Fractional solutions showing \Cref{eq: st} is not ideal and can have large integrality gaps.} \label{fig:st fractional extreme points}
    \begin{minipage}{\textwidth} 
    \footnotesize
    Note. In the left panel, the striped nodes have $\x$-value $\sfrac{1}{2}$.
    In the right panel, the striped nodes have $\x$-value $\sfrac{2}{3}$.
    In both panels, the gray nodes have $\x$-value $1$ and the white nodes have $\x$-value $0$.
    \end{minipage}
\end{figure}

In contrast to \Cref{thm: subtree}, \Cref{eq: st} is typically not ideal
for scenario trees with more than $\nrevisions+2$ stages.
\begin{proposition}
\label[proposition]{prop: st not ideal in general}
    \Cref{eq: st} is not ideal for $\revisableset{\nrevisions}$, even when $\nstages=4$ and $\nrevisions=1$.
\end{proposition}

\begin{proof}
Consider the scenario tree presented in the left panel of \Cref{fig:st fractional extreme points} and consider the solution $\x$ where the decision at each node $v$ is represented by its shading: $x(v)=1$ for gray nodes, $x(v)=\sfrac{1}{2}$ for striped gray nodes, and $x(v)=0$ for white nodes. 
This fractional solution can be verified to be a vertex of $P_{\text{ST}}$.
\end{proof}

Further, we observe that \Cref{eq: st} may get weaker when the tree grows taller.
\Cref{ex: st for tall tree} shows an example where optimizing over the LP relaxation of \Cref{eq: st} can lead to an objective that is arbitrarily far from the optimal value over $\revisableset{\nrevisions}$.

\begin{example}  \label[example]{ex: st for tall tree}
Consider the $1$-revision hypercube problem defined on the perfect binary scenario tree that is depicted on the right panel of \Cref{fig:st fractional extreme points}, where the tree has height $\nstages$.
Construct an objective vector that takes value $+2^{\nstages-t}$ and $-2^{\nstages-t}$ respectively for each sibling pair in stage $t\in[\nstages]$.
The right panel of \Cref{fig:st fractional extreme points} presents a fractional solution of $P_{\text{ST}}$, in which the value at each node $v$ is represented by its shading: $x(v)=\sfrac{2}{3}$ if $v$ is striped gray, and $x(v)=0$ if $v$ is white. 
For each 2-height ELBE subtree $\calS$, we have $\sum_{\{u,v\} \in \sib{\calS}} \big\vert \x{u} - \x{v} \big\vert = 3 \cdot \sfrac{2}{3} = 2$.
Evaluating the objective at this fractional solution shows that the optimal value of the LP relaxation is at least $\sfrac{2}{3} \cdot \nstages \,2^{\nstages-1}$, whereas we can show that the optimal integer solution has objective value $\sum_{t=1}^\nstages 2^{\nstages-t} = 2^{\nstages}-1$ by induction on $\nstages$.
Hence, the optimal objective of the LP relaxation is at least a factor of order $\nstages$ times the optimal value over $\revisableset{\nrevisions}$.
\end{example}

\section{Tightening and reformulation techniques} \label{sec: methods}

In \Cref{sec: formulations} we introduced two formulations for the $\nrevisions$-revision constraint. 
We also established that \Cref{eq: cp} is not sharp and that \Cref{eq: st} has a large number of constraints. 
In this section, we derive results that alleviate these limitations.  
The roadmap for these results is presented in the bottom seven blocks of \Cref{fig: summary}.

\subsection{Enhancing the subtree formulation} \label{subsec: subtree enhancement}
The main limitation of \Cref{eq: st} in practice is its size, as it requires the introduction of a constraint for each oriented ELBE subtree of $\tree$.
In fact, for the case where $\tree$ is a perfect binary scenario tree with $\nstages$ stages, the number of $(\nrevisions+1)$-height oriented ELBE subtrees can be shown to be $\Theta(|\nodes|^{2^{\nrevisions+1}})$ when $\nrevisions \ll \nstages$.
To handle the rapid growth in the number of subtree constraints in \Cref{eq: st}, we derive a separation algorithm. 
Then, inspired by that procedure, we propose an extended formulation of the same strength as \Cref{eq: st} that requires much fewer constraints, without adding too many variables.  

\subsubsection{Subtree constraint generation algorithm} \label{subsubsec: st separation}

To find a violated subtree constraint given a fractional solution $\x$, we define the \emph{inconsistent value} of an ELBE subtree $\calS$ for a given solution $\x$ to be
\begin{equation} \label{eq: x-inconsistent-value}
    \Delta^\calS_x := \sum_{\{u,v\} \in \sib{\calS}} \big\vert \x{u} - \x{v} \big\vert.
\end{equation}
We then search for a height-$(\nrevisions+1)$ ELBE subtree with largest inconsistent value in the scenario tree $\tree$.
A violated subtree constraint exists if and only if that value is larger than $2^{\nrevisions+1}-2$.
Once such an ELBE subtree is identified, a violated subtree constraint is derived by the orientation where each pair $(u,v)$ is ordered such that $\x{u} \geqslant \x{v}$.

Next, we describe how to compute the largest inconsistent value using \DP.
The core idea of this \DP{} algorithm is that the most inconsistent ELBE subtree in $\tree(v)$ can be derived from the most inconsistent ELBE subtrees in $\{\tree(u)\}_{u \in \children{v}}$.
We have used this idea in the proof of \Cref{thm: subtree}, where in \Cref{eq: Delta v} we define the largest inconsistent value among ELBE subtrees rooted at a node $v$ in the tree.
There, when $\nstages = \nrevisions+2$, we only need to consider ELBE subtrees with height the same as $\tree(v)$ for the node $v$.
However, for the general case $\nstages > \nrevisions+2$, this is not sufficient.
We therefore extend the notation of \Cref{eq: Delta v} as follows. 
Let $\Delta_x(v,h)$ denote the inconsistent value of the most inconsistent ELBE subtree with height $h$ in $\tree(v)$.
We extend \Cref{eq: Delta v recurrence} to obtain the following recurrence relation for $\Delta_x(v,h)$:
\begin{equation} \label{eq: DP relation}
\begin{aligned}
     \Delta_x(v,h) &:= \max_{\calS \in \subtreefamily{h}{\tree(v)}} \Delta^\calS_x \\
    & = \max \left\{  \max_{\substack{p, q \in \descendant{v}: \\ \stage{p}=\stage{q},  p \join q = v}} \left\{ \begin{array}{c} \Delta_x(p,h-1) \\ + \Delta_x(q,h-1) \\ + (\x{p} - \x{q}) \end{array}\right\}, \max_{u \in \children{v}} \big\{ \Delta_x(u,h) \big\}  \right\}.
\end{aligned}
\end{equation}
The recurrence captures two cases for the most inconsistent height-$h$ ELBE subtree within $\tree(v)$: either its root is $v$ itself, or it is entirely contained within a smaller $\tree(u)$ where $u$ is a child of $v$.
With the initial values $\Delta_x(v,0)=0$ for all $v \in \nodes$,
a direct implementation of recurrence \eqref{eq: DP relation} yields an algorithm to compute $\Delta(\root, \nrevisions+1)$.

For each $v$, not all $\Delta_x(v,h)$, where $h$ in $[\nrevisions+1]$ are needed.
For example, it is not possible to find a subtree of height $h$ when $h > \nstages-\stage{v}$.
Further, trees with height $h \leqslant \nrevisions - \stage{v}$ are not important because they are too small to be a part of a $(\nrevisions+1)$-height ELBE subtree.
Using $\calH(v,\nrevisions)$ to denote a set of heights that it is sufficient to consider at a node $v$ for the recurrence, the above discussion shows that we can choose $\calH(v,\nrevisions) = [\max\{1, \nrevisions-\stage{v}+1\} : \min\{\nrevisions+1, \nstages-\stage{v}\}]$.

We refer to \Cref{alg: subtree separation alg} in Appendix~\ref{sec: constraint generation algorithm} for a pseudo-code description of the resulting separation algorithm. 
The following proposition establishes the correctness of \Cref{eq: DP relation}.
\begin{proposition} \label[proposition]{prop: stdp correctness}
Given a solution $x \in [0,1]^\nodes$, the largest inconsistent value $\Delta_x(\root, \nrevisions+1)$ can be found through the recurrence \Cref{eq: DP relation} with initial values $\Delta_x(v,0)=0$ for all $v \in \nodes$.
It requires $\calO(|\nodes|^2\nrevisions)$ steps to compute.
\end{proposition}
\begin{proof}
    We prove that the recurrence \Cref{eq: DP relation} is correct by showing that the outer $\max$ operation exactly describes  all possible ways a height-$h$ ELBE subtree can be formed within $\tree(v)$.
    For an ELBE subtree $\calS$ within $\tree(v)$, there are only two possibilities.
    On the one hand, it could be that $\calS$ has root $v$. 
    Then $v$ is the join of all the nodes of $\calS$ and thus the join of $p$ and $q$ where $(p,q)$ is a direct sibling pair with parent $v$ in $\calS$.
    This explains the first inner $\max$ in \eqref{eq: DP relation}. 
    On the other hand, it could be that the root of $\calS$ is a descendant of $v$.
    In this case, $\calS$ must be contained in $\tree(u)$ for some child $u$ of $v$.
    This explains the second inner $\max$ in \eqref{eq: DP relation}.  

    Equation \eqref{eq: DP relation} shows that the value of $\Delta_x(\root, \nrevisions+1)$ can be obtained from the values of $\Delta_x(v,h)$, of which there are no more than $|\nodes|\cdot (\nrevisions+1)$. 
    Each entry requires the computation of two inner $\max$ operations on entries that can be accessed or calculated in constant time. 
    Hence, it remains to bound the number of terms considered in these inner maximum operations. 
    For each pair $(p,q)$ at the same stage in $\tree$ and a given $h \leqslant \nrevisions$, $\Delta_x(p, h)$ and $\Delta_x(q, h)$ are enumerated exactly once, in the computation of $\Delta_x(p \join q, h+1)$.
    Similarly, for any non-root node $u \in \nodes \setdiff \{\root\}$ and $h \leqslant \nrevisions$, $\Delta_x(u,h)$ only appears once in the computation of $\Delta_x(\parent{u},h)$.
    Therefore, the total number of $\Delta$ terms arising in inner max computations is bounded above by $|\nodes|^2 \nrevisions + |\nodes| \nrevisions = \calO(|\nodes|^2 \nrevisions)$.
\end{proof}

Comparing the value of $\Delta_x(\root, \nrevisions+1)$ with $2^{\nrevisions+1}-2$ provides a polynomial procedure to determine if a strategic policy $\x$ is $\nrevisions$-revisable. 
Combining this algorithm with binary search on the possible values of $\nrevisions$ from $0$ to $\nstages-1$, we obtain 
\begin{corollary} \label[corollary]{cor: minimum revisability number}
    The minimum revisability number of a given binary strategic policy $\x$ can be determined in polynomial time.
\end{corollary}

\subsubsection{An extended formulation}

The concepts behind the above constraint generation algorithm serve as inspiration for a new formulation for $\revisableset{\nrevisions}$.
Specifically, we can turn $\Delta_x(v, h)$ into variables and the recurrence \Cref{eq: DP relation} together with the requirement that $\Delta_x(\root, \nrevisions+1) \leqslant 2^{\nrevisions+1}-2$ into constraints. This leads to the formulation:
\begin{equation} \label{eq: stdp}
\begin{aligned}
    & \mathrlap{\Delta(v, h)  \geqslant \Delta(p, h-1) + \Delta(q, h-1) + (\x{p} - \x{q}),}  \\
    & \hspace{+1cm} && \forall v \in \nodes \setminus \nodes_{\nstages}, \forall h \in \calH(v,\nrevisions),\\
    &&& \forall p, q \in \descendant{v}, p \join q = v, \\
    & \Delta(v, h) \geqslant \Delta(u, h),  
    \quad  && \forall v \in \nodes \setminus \nodes_{\nstages}, \forall u \in \children{v}, \\ 
    &&& \forall h \in \calH(v, \nrevisions) \cap \calH(u, \nrevisions), \\
    & \Delta(v, 0) = 0, \quad && \forall v \in \nodes, \\
    & \Delta(\root, \nrevisions+1)  \leqslant 2^{\nrevisions+1} - 2.
\end{aligned}
\tag{STDP}
\end{equation}
We refer to \Cref{eq: stdp} as the \emph{subtree DP formulation}. 
The following results are immediate consequences of \Cref{thm: subtree} and \Cref{prop: stdp correctness}.
\begin{corollary} \label[corollary]{cor: stdp size}
It holds that $P_{\text{ST}} = \proj{x}{P_{\text{STDP}}}$.
Further, the size of \Cref{eq: stdp} is $\order{|\nodes|^2\nrevisions}$.
\end{corollary}
\begin{corollary}  \label[corollary]{cor: stdp sharp}
    \Cref{eq: stdp} is sharp for $\revisableset{\nrevisions}$ when $\nstages\leqslant \nrevisions+2$.
\end{corollary}

\subsection{Enhancing the complete plan formulation}

\subsubsection{A size reduction technique} \label{subsubsec: cp size reduction}

The number of constraints and variables of \Cref{eq: cp} are of the same order. 
The number of variables of \Cref{eq: cp} is mainly determined by the number of $\plan$ variables.
Since for a node $v \in \nodes$ we define $\plan{v,t}$ for all $t \in [\stage{v}:\nstages]$, the total number of $\plan$ variables is
$\sum_{t \in [\nstages]} |\nodes_t| \cdot (\nstages - t + 1) = \calO(|\nodes| \nstages)$,
where the order bound is tight for ``tall and thin'' trees.
Specifically, it is tight for paths, where the number of $\plan$ variables is $\sfrac{|\nodes| \nstages}{2}$.

However, in the path example, there is no need to introduce that many $\plan$ variables, as no revisions are necessary. 
In fact, the formulation only requires $\plan$ variables at the root node for the entire path.
Using this idea, we can reduce the size of \Cref{eq: cp} by removing specific $\plan$ variables.
We say that a non-root node $v$ is an \emph{only child} node if and only if its parent has only one child, $v$.
Let
\begin{equation*}
    \nodes_{\textnormal{only}} := \{ v \in \nodes \setdiff \{\root\} : \vert \children{\parent{v}} \vert = 1 \}
\end{equation*}
denote the set of only child nodes of $\tree$.
For example, in the left panel of \Cref{fig:st fractional extreme points}, all third-stage nodes are only child nodes.
The plan adjustment variables $\plan{v,t}$ for an only child node $v$ and stage $t \in [\stage{v}:\nstages]$ can be fixed to the value of its parent node $\parent{v}$ at stage $t$, \emph{i.e.}, we can enforce $\plan{v,t}=\plan{\parent{v},t}$ and $r(v)=0$.
These restrictions change the feasible region in the space of $\plan$ and $r$ variables but do not change the projection onto the space of $\x$ variables, \emph{i.e.}, the resulting model is still a formulation of $\revisableset{\nrevisions}$.
Discarding variables $\plan{v, t}$ and $r(v)$ for all $v \in \nodes_{\textnormal{only}}$ from \Cref{eq: cp} yields the following formulation:

\begin{equation} \label{eq: cp+}
\begin{aligned}
    &\x{v} = \plan{v, \stage{v}}, \quad && \forall v \in \nodes \setdiff \nodes_{\textnormal{only}}, \\
    &\x{v} = \plan{\overbar{\operatorname{pa}}(v), \stage{v}}, \quad && \forall v \in \nodes_{\textnormal{only}}, \\
    &\rv{v} \geqslant \plan{v,t} - \plan{\overbar{\operatorname{pa}}(v),t}, \quad && \forall v \in \nodes \setdiff (\set{\root} \cup \nodes_{\textnormal{only}}), t \in [\stage{v}:\nstages],  \\
    &\rv{v} \geqslant \plan{\overbar{\operatorname{pa}}(v),t} - \plan{v,t}, \quad && \forall v \in \nodes \setdiff (\set{\root}  \cup \nodes_{\textnormal{only}}), t \in [\stage{v}:\nstages], \\
    &\textstyle \sum_{v \in \scenario \setdiff \nodes_{\textnormal{only}}} \rv{v} \leqslant \nrevisions, \quad && \forall \scenario \in \scenarioset, \\
    &\x{v} \in \Bin, \quad && \forall v \in \nodes, \\  
    & \plan{v,t} \in \Bin,\; \rv{v} \in \Bin, \quad && \forall v \in \nodes \setdiff \nodes_{\textnormal{only}}, t \in [\stage{v}:\nstages],  \\
\end{aligned}
\tag{CP+}
\end{equation}
where the operator $\overbar{\operatorname{pa}}(v)$ finds the nearest ancestor of $v$ that is not an only child.
The size of \Cref{eq: cp+} is smaller than that of \Cref{eq: cp} by a factor $\nstages$, as we record next. 
\begin{proposition} \label[proposition]{prop: cpf size reduction}
    \Cref{eq: cp+} has size $\order{|\nodes|}$.
\end{proposition}
\begin{proof}
It suffices to count the number of $\plan{\cdot,\cdot}$ variables, where the size of the vector $\plan$ at each node $v$ is either $\nstages-\stage{v}+1$ if $v$ is not an only child node, and  $0$ otherwise.

    Let $F$ denote the number of $\plan(\cdot, \cdot)$ variables used in the formulation. 
    We use induction on $\nstages$ to prove that $F \leqslant 2|\nodes|-\nstages$, which will directly establish that $F = \calO(|\nodes|)$.
    For $\nstages=1$, the statement clearly holds as the formulation requires the single plan variable $\plan{\root, 1}$.
    For $\nstages>1$, let $F_v$ denote the number of $\plan$ variables used in the subtree $\tree(v)$ for $v \in \children{\root}$. 
    Then $F = \nstages + \sum_{v \in \children{\root}} F_v$.
    If $\children{\root} = \{v\}$, then we discard $\nstages-1$ variables at $v$, and thus it is easy to see (through induction) that $F_v \leqslant 2|\nodes(v)|-(\nstages-1) - (\nstages-1) = 2(|\nodes|-1)-2(\nstages-1) = 2|\nodes|-2\nstages$.
    Thus, $F \leqslant 2|\nodes| - \nstages$.
    If $ |\children{\root}| \geqslant 2$, we have $F_v \leqslant 2|\nodes(v)| - (\nstages-1)$ for every $v \in \children{\root}$ by induction, which implies $F \leqslant \nstages + \sum_{v \in \children{\root}} 2|\nodes(v)| - |\children{\root}|(\nstages-1) = 2|\nodes| - \nstages - (|\children{\root}|-2)(\nstages-1)$, where the equality holds because $\sum_{v \in \children{\root}} |\nodes(v)| = |\nodes|-1$.
    Since $|\children{\root}| \geqslant 2$, we conclude that $F \leqslant 2|\nodes| - \nstages$.
\end{proof}

\subsubsection{A new class of facet-defining inequalities} \label{subsubsec: cp facet}
In \Cref{prop: cpf is not sharp}, we have shown that \Cref{eq: cp} is not sharp for the set $\revisableset{\nrevisions}$, even for the case $\nstages=3$ and $\nrevisions=1$.
A key strength of \Cref{eq: st} in this case is that it is ideal. 
We could therefore think of using the subtree constraints to strengthen \Cref{eq: cp}.
Unfortunately, even though these constraints are facet-defining for $\revisableset{\nrevisions}$ in the space of $\x$, they are not facet-defining in the space of $(\x, \plan, r)$.
The following theorem shows that we can modify them so as to make them facet-defining in the space of $(\x, \plan, r)$. 

\begin{theorem} \label{thm: cp_facets}
    Let $\calS$ be an ELBE subtree with height $\subheight{\calS} \in [\nrevisions]$. 
    Assume that the stage of root $\subroot{\calS}$ is strictly larger than $\nrevisions-\subheight{\calS}$. 
    Then for any oriented version $\vec{\calS} \in \operatorname{ori}(\calS)$, the inequality 
    \begin{equation} \label{eq: cp_facets}
        \sum_{w \leq \subroot{\calS}} r(w) + \sum_{(u,v) \in \sib{\vec{\calS}}} (x(u)-x(v)) \leqslant \nrevisions - \subheight{\calS} + 2^{\subheight{\calS}} - 1
    \end{equation}
    defines a facet of $\conv{P_{\textnormal{CP}}^I}$.
\end{theorem}

\def\ancestorset{\calS^{\leq}}
In the rest of this section, we suppose $\calS$ is an ELBE subtree of $\tree$ satisfying the condition in \Cref{thm: cp_facets}.
To streamline notation, we use $\stage{\calS}$ to denote $\stage{\subroot{\calS}}$.
We use $\calS$ instead of $\vec{\calS}$ to denote the oriented ELBE subtree.
We use $(p,q)$ to denote the pair of children of $\subroot{\calS}$ in $\calS$.
We define $\ancestorset :=\{w: w \leq \subroot{\calS} \}$.
The assumptions of \Cref{thm: cp_facets} guarantee 
that the root node $\subroot{\calS}$ has level $\stage{\calS} \geqslant \nrevisions-\subheight{\calS}+1$
and
$|\ancestorset| = \stage{\calS} \geqslant \nrevisions - \subheight{\calS} + 1$.
These assumptions ensure that the first summation in \Cref{eq: cp_facets} is not empty.

To prove the theorem, we make use of the following two ancillary lemmas, whose proofs can be found in \Cref{proof: cp facet - valid}.
\begin{restatable}{lemma}{LmCPFacetValid}
\label[lemma]{lm: cp facet - valid}
    Inequality \Cref{eq: cp_facets} is valid for $P_\textnormal{CP}^I$.
\end{restatable}

\Cref{lm: cp facet - valid} implies that \Cref{eq: cp_facets} defines a face of $P_\textnormal{CP}^I$, which we refer to as $F_{\textnormal{CP}}$ in the following derivations. 
Next, \Cref{lm: cp facet - construct solutions} identifies families of 
integer points on $F_{\textnormal{CP}}$. 

\begin{restatable}{lemma}{LmCPFacetConstructSolutions}
\label[lemma]{lm: cp facet - construct solutions}
Given an ELBE subtree $\calS$ with height in $[\nrevisions]$ 
    and $\stage{\subroot{\calS}} > \nrevisions-\subheight{\calS}$.
    \begin{enumerate}[label=(\Roman*)]
    \item \label{cp_facet_pointI}
    For each $\theta \in \{0, 1\}$ and each $J \subset \ancestorset$ with $|J|=\nrevisions-\subheight{\calS}$, there exists a solution $(\hat{\plan}, \hat{r}_J, \hat{x}_\theta) \in F_{\textnormal{CP}}$  such that (i) $\hat{x}_\theta(u)=1$ and $\hat{x}_\theta(v)=0$ for all $(u,v) \in \sib{\calS}$, $\hat{x}_\theta(w)=\theta$ for all $w \in \nodes \setdiff \nodes(\calS) \cup \{\subroot{\calS}\}$, and (ii) $\hat{r}_J(w)=1$ iff $x(w)=1-\theta$ or $w \in J$.
    \item \label{cp_facet_pointII}
    For each $\theta \in \{0, 1\}$, each $J \subset \ancestorset$ with $|J|=\nrevisions-\subheight{\calS}+1$, and any $(u,v) \in \sib{\calS}$, there exists a solution $(\breve{\plan}, \breve{r}_J, \breve{x}^{u,v}_\theta) \in F_{\textnormal{CP}}$ such that (i) $\breve{x}^{u,v}_\theta(u)=\breve{x}^{u,v}_\theta(v)=\theta$, and $\breve{x}^{u,v}_\theta(u')=1$, $\breve{x}^{u,v}_\theta(v')=0$ for other pairs $(u',v') \in \sib{\calS} \setdiff \{(u,v)\}$, and (ii)
     for any $w \in \nodes \setdiff \nodes(\calS) \cup \{\subroot{\calS}\}$,  $\breve{r}_J(w)=1$ iff $w \in J$.
    \item \label{cp_facet_pointIII}
    For any $i \in \nodes \setdiff \ancestorset$, there exists a solution $(\dot{\plan}, \dot{r}, \dot{x}) \in F_{\textnormal{CP}}$ such that $(\dot{\plan}, \dot{r}+e_i, \dot{x}) \in F_{\textnormal{CP}}$.
    \item \label{cp_facet_pointIV}
    For any node $i \in \nodes$, there exists a solution $(\ddot{\plan}, \ddot{r}, \ddot{x}) \in F_{\textnormal{CP}}$ such that $\ddot{r}(i)=1$.
    \end{enumerate}
\end{restatable}

\begin{proof}[Proof of \Cref{thm: cp_facets}]
Consider any inequality that defines $F_{\textnormal{CP}}$:
\begin{equation} \label{eq: any inequality}
    \sum_{i \in \nodes} \alpha_i r(i) + \sum_{i \in \nodes} \beta_i \x{i} + \sum_{\substack{i \in \nodes \\ t \in [\stage{i}+1:\nstages]}}  \gamma_{i,t} \plan{i, t} \leqslant \eta.
\end{equation}
We assume without loss of generality that there are no variables $\{\plan(v, \stage{v})\}_{v \in \nodes}$ since they can be projected out of the formulation through the equalities $x(v)=\plan(v, \stage{v})$.
We prove that \Cref{eq: any inequality} is a scalar multiple of \Cref{eq: cp_facets} by showing
\begin{enumerate}[label=(\alph*)]
    \item $\alpha_i = 0, \forall i \in \nodes\setdiff \ancestorset$. \label{cp_facets_cond1}
    \item $\beta_i=0, \forall i \in \nodes \setdiff \nodes(\calS) \cup \{\subroot{\calS}\}$. \label{cp_facets_cond2}
    \item $\gamma_{i,t} = 0, \forall i \in \nodes, t \in [\stage{i}+1 : \nstages]$. \label{cp_facets_cond3}
    \item $\alpha_w = \beta_u = -\beta_v, \forall w \in \ancestorset, (u,v) \in \sib{\calS}$. \label{cp_facets_cond4}
\end{enumerate}
This will imply that the inequality is facet-defining

For \ref{cp_facets_cond1}, \Cref{lm: cp facet - construct solutions}\ref{cp_facet_pointIII} directly implies that $\alpha_i=0$.
For \ref{cp_facets_cond2}, it suffices to show that there exists a solution $(\plan, r, x) \in F_{\textnormal{CP}}$ such that $r(i)=1$.
This is because the solution $(\plan', r', x')$ obtained from $(\plan, r, x)$ by setting $x'(i) = 1-x(i)$ and keeping all other components of $x$ unchanged is feasible and satisfies \Cref{eq: cp_facets} at equality, showing that $\beta_i=0$.
Therefore, \Cref{lm: cp facet - construct solutions}\ref{cp_facet_pointIV} proves the claim.
To prove \ref{cp_facets_cond4}, assume \ref{cp_facets_cond3} holds.
Consider $(u,v) \in \sib{\calS}$ and choose $J \subset \ancestorset \setdiff \{w\}$ with cardinality $K-\subheight{\calS}$. 
We then take the three solutions $(\hat{\plan}, \hat{r}_{J}, \hat{x}_0)$, $(\breve{\plan}, \breve{r}_{J\cup \{w\}}, \breve{x}_0^{u,v})$, and $(\breve{\plan}', \breve{r}_{J\cup \{w\}}, \breve{x}_1^{u,v})$ given by \Cref{lm: cp facet - construct solutions}\ref{cp_facet_pointI} and \ref{cp_facet_pointII}.
\Cref{lm: cp facet - construct solutions}\ref{cp_facet_pointII} stipulates the existence of such a tuple for each $\theta \in \{0,1\}$, but it does not guarantee the same plan adjustment policy to work for both values of $\theta$. 
Hence, we allow the policies to differ and denote them by $\breve{\plan}$ for $\theta=0$ and by $\breve{\plan}'$ for $\theta=1$.
Plugging these solutions into \Cref{eq: any inequality} yields the following equalities: $\sum_{i \in J} \alpha_i + \sum_{(u',v') \in \sib{\calS}} \beta_{u'} = \eta$, $\sum_{i \in J} \alpha_i + \alpha_w + \sum_{(u',v') \in \sib{\calS}} \beta_{u'} - \beta_u = \eta$, and $\sum_{i \in J} \alpha_i + \alpha_w + \sum_{(u',v') \in \sib{\calS}} \beta_{u'} + \beta_v = \eta$.
By subtracting the second (\textit{resp.} third) equality from the first (\textit{resp.} second), we obtain that $\alpha_w = \beta_u = -\beta_v$.

It thus remains to prove \ref{cp_facets_cond3}.
We prove $\gamma_{i,t}=0$ for any $t$ and any $i$ such that $\stage{i} < t$. 
To this end, we next introduce the concept of a t-diffusion. 
Consider a solution $(\tilde{\plan},\tilde{r},\tilde{x}) \in F_{\textnormal{CP}}$ and a node $s \in \nodes$ such that $s=\root$ or $\tilde{r}(s)=1$. 
A $t$-\emph{diffusion} starting at $s$ is defined on $\tilde{r}$ as $D:=\{v\in \descendant{s} \cup \{s\} \mid \stage{v} \leqslant t, \tilde{r}(u)=0, \forall u \in P(s,v) \setdiff \{s\} \}$, where $P(s,v)$ represents the unique path from $s$ to $v$.
Equivalently, $D$ is obtained by exploring every branch below $s$ and collecting nodes until either a revised node is encountered or stage $t$ is reached.

Given a t-diffusion $D$, we construct a new solution $(\tilde{\plan}',\tilde{r},\tilde{x}')$ by performing flipping operations on $\tilde{\plan}(\cdot, t)$ along $D$.
To be specific, let $\tilde{\plan}'(l,t)=1-\tilde{\plan}(l,t)$ for $l \in D$ and $\stage{l}<t$; let $\tilde{x}'(l)=1-\tilde{x}(l)$ for $l \in D$ and $\stage{l}=t$; keeping all other variables unchanged.
The solution $(\tilde{\plan}', \tilde{r}, \tilde{x}')$ remains feasible.
Further, for the solution $(\tilde{\plan}', \tilde{r}, \tilde{x}')$ to remain in $F_{\textnormal{CP}}$, we need to make sure that if $v \in (\nodes(\calS) \setdiff \{\subroot{\calS}\} ) \cap \nodes_t $, then $v \notin D$.
This prevents any $\tilde{x}(v)$ to be flipped for $v$ in $\calS$.
If $D$ is such that $(\tilde{\plan}',\tilde{r},\tilde{x}') \in F_{\textnormal{CP}}$, we say that $D$ is a \emph{feasible} $t$-diffusion.
Further, if $\stage{i}<t$ for some $i \in D$, and $\gamma_{j,t}=0$ for all $j \in D$, $j \neq i$ such that $\stage{j} < t$, then we derive $\gamma_{i,t}=0$ from plugging in $(\tilde{\plan},\tilde{r},\tilde{x})$ and $(\tilde{\plan}',\tilde{r},\tilde{x}')$ into \Cref{eq: any inequality}.
Hence, we next construct a solution $(\tilde{\plan},\tilde{r},\tilde{x})$ and a $t$-diffusion for each $i \in \nodes$ and $t \in [\stage{i}+1 : \nstages]$ to show that $\gamma_{i,t}=0$. 

Let $\descendant{i}_t$ denote the descendants of $i$ that are on stage $t$, \emph{i.e.}, $\nodes_t \cap \descendant{i}$.
First, we discuss the case where $\descendant{i}_t \cap \nodes(\calS) \setdiff \{\subroot{\calS}\} = \varnothing$.
Suppose that $i$ is any node at \emph{maximum} stage with $\gamma_{i,t} \neq 0$.
By \Cref{lm: cp facet - construct solutions}\ref{cp_facet_pointIII}, there exists a solution $(\dot{\plan}, \dot{r}, \dot{x})$ such that $\dot{r}(i)=1$.
Let $(\tilde{\plan},\tilde{r},\tilde{x}) = (\dot{\plan}, \dot{r}, \dot{x})$,
then the $t$-diffusion starting at $i$ is feasible since any non-root node in $\calS$ is not contained in $\descendant{i}_t$ and thus not in this $t$-diffusion.
Hence, we derive $\gamma_{i,t}=0$.

Second, if $\descendant{i}_t \cap \nodes(\calS) \setdiff \{\subroot{\calS}\} \neq \varnothing$, then $i$ is either an ancestor of $\subroot{\calS}$, a descendant of $\subroot{\calS}$, or $\subroot{\calS}$ itself.
We first prove the case where $i > \subroot{\calS}$.
Suppose that $i$ is the node at \emph{maximum} stage with $\gamma_{i,t} \neq 0$.
Either $i$ is an ancestor or descendant of $p$, or $i$ is an ancestor or descendant of $q$, where $p$ and $q$ are the left and right child of $\subroot{\calS}$ in $\calS$, respectively.

If $i \geq p$ (\emph{resp.} $i \geq q$), consider $(\tilde{\plan},\tilde{r},\tilde{x})=(\hat{\plan}, \hat{r}_J+e_i+e_{R_\theta}, \hat{x}_\theta)$ where  $(\hat{\plan},\hat{r},\hat{x})$ is defined in \Cref{lm: cp facet - construct solutions}\ref{cp_facet_pointI} with $\theta=1$ (\emph{resp.} $\theta=0$), $J$ is any eligible set, and we let $R_\theta$ be all the left nodes (\emph{resp.} right nodes) excepting $p$ (\emph{resp.} $q$).
This solution is feasible as at most $\subheight{\calS}$ revisions are taken over any downstream path in $\tree(\subroot{\calS})$.
Hence, $(\tilde{\plan},\tilde{r},\tilde{x}) \in F_{\textnormal{CP}}$ and there is a feasible $t$-diffusion starting at $i$, which proves $\gamma_{i,t}=0$.

If $i < p$ (\emph{resp.} $i < q$). 
When $t < \stage{p}$ (\emph{resp.} $\stage{q}$), it reduces to the case $\descendant{i}_t \cap \nodes(\calS) \setdiff\{\subroot{\calS}\} = \varnothing$.
When $t > \stage{p}$ (\emph{resp.} $\stage{q}$), we repeat the previous construction for the case $i \geq p$ (\emph{resp.} $i \geq q$) and we have $(\tilde{\plan},\tilde{r},\tilde{x}) \in F_{\textnormal{CP}}$, and $t$-diffusion starting at $i$ is feasible.
When $t = \stage{p}$ (\emph{resp.} $\stage{q}$), consider $(\hat{\plan}, \hat{r}_J, \hat{x}_\theta)$ where $\theta=0$ (\emph{resp.} $\theta=1$) and $J$ to be any eligible set.
Then, we define $(\hat{\plan}',\hat{r}_J',\hat{x}'_\theta)$ by setting $\hat{\plan}'(p, \cdot)=\hat{\plan}(\parent{p}, \cdot)$, $\hat{r}_J'(p)=0$ and $\hat{x}'_\theta(p)=0$ (\emph{resp.} $\hat{\plan}'(p, \cdot)=\hat{\plan}(\parent{p}, \cdot)$, $\hat{r}_J'(q)=0$ and $\hat{x}'_\theta(q)=1$) and by keeping the other entries unchanged.
This solution is still feasible but not on the face $F_{\textnormal{CP}}$ since we changed the $x$-value on $p$ (\emph{resp.} $q$).
Let $(\tilde{\plan},\tilde{r},\tilde{x})=(\hat{\plan}',\hat{r}_J'+e_i,\hat{x}'_\theta)$, which is feasible and consider the $t$-diffusion starting at $i$.
This diffusion, even starting from a non-feasible solution, creates a feasible solution since the $x$-value on $p$ (\emph{resp.} $q$) is flipped back.
Plugging in the solution defined by the $t$-diffusion $(\tilde{\plan}',\tilde{r}',\tilde{x}')$ and the solution $(\hat{\plan},\hat{r}_J,\hat{x}_\theta)$ to \Cref{eq: any inequality}, we get $\gamma_{i,t}=0$ as we have proved $\gamma_{j,t'}=0$ for all $j > \subroot{\calS}$ and $t'\neq t$.

For the case $i < \subroot{\calS}$,
suppose $i$ is the node with $\gamma_{i,t}=0$ that lies at \emph{minimum} stage.
We let $i^+$ denote the child of $i$ that is in $\ancestorset$.
We consider $(\tilde{\plan}, \tilde{r}, \tilde{x})=(\breve{\plan}, \breve{r}_J, \breve{x}^{p,q}_0)$ as defined in \Cref{lm: cp facet - construct solutions}\ref{cp_facet_pointII} such that $J$ is a subset of $\ancestorset$ with $i^+ \in J$.
Then, the $t$-diffusion defined on $\tilde{r}$ starting at $\root$ is feasible and all the nodes other than $i$ in the diffusion have $\gamma_{\cdot, t}$-value $0$.
Hence $\gamma_{i,t}=0$.

Last, for the case where $i =\subroot{\calS}$, let $(\tilde{\plan},\tilde{r},\tilde{x})=(\hat{\plan}, \hat{r}_J+e_q, \hat{x}_0)$.
The $t$-diffusion starting at $\root$ is feasible and since $i=\subroot{\calS}$ is the only node where $\gamma_{\cdot,t}$ is not fixed, $\gamma_{i,t}=0$.

Hence, we showed that \Cref{eq: any inequality} is a scalar multiple of \Cref{eq: cp_facets}, which proves that \Cref{eq: cp_facets} is facet-defining.
\end{proof}

A special case of inequalities \eqref{eq: cp_facets} is when the ELBE subtree $\calS$ reduces to a single leaf node $\ell$.
In this case, \Cref{eq: cp_facets} reduces to the revision budget constraint $\sum_{v \leq \ell} r_v \leqslant \nrevisions$ for scenario $\ell$.
Although the proof of \Cref{thm: cp_facets} does not explicitly cover this special case, a simple separate argument can show that these revision budget constraints are facet-defining for $\conv{P_{\textnormal{CP}}^I}$.

Secondly, as was the case for subtree constraints, there is a large number of inequalities \eqref{eq: cp_facets}.
Using ideas similar to those we developed in \Cref{subsec: subtree enhancement}, we can derive a \DP{} algorithm to separate these  inequalities.
Alternatively, we can introduce new variables $\Delta(v, h)$ as in \Cref{eq: stdp} and then reformulate \Cref{eq: cp_facets} as:
\begin{equation} \label{eq: cp facet (stdp)}
\sum_{i \leq v} r(i) + \Delta(v, h) \leqslant \nrevisions - h + 2^{h} - 1, \quad \forall v \in \nodes \setdiff \nodes_\nstages, \forall h \in \calH(v,\nrevisions).
\end{equation}
Combining \Cref{eq: cp+}, \Cref{eq: stdp}, and \Cref{eq: cp facet (stdp)}, we obtain the following formulation \Cref{eq: cp++}:
\begin{equation} \label{eq: cp++}
\begin{aligned}
    &\x{v} = \plan{v, \stage{v}}, \quad && \forall v \in \nodes \setdiff \nodes_{\textnormal{only}}, \\
    &\x{v} = \plan{\overbar{\operatorname{pa}}(v), \stage{v}}, \quad && \forall v \in \nodes_{\textnormal{only}}, \\
    &\rv{v} \geqslant \plan{v,t} - \plan{\overbar{\operatorname{pa}}(v),t}, \quad && \forall v \in \nodes \setdiff (\set{\root} \cup \nodes_{\textnormal{only}}), \forall t \in [\stage{v}:\nstages],  \\
    &\rv{v} \geqslant \plan{\overbar{\operatorname{pa}}(v),t} - \plan{v,t}, \quad && \forall v \in \nodes \setdiff (\set{\root} \cup \nodes_{\textnormal{only}}), \forall t \in [\stage{v}:\nstages], \\
    &\textstyle \sum_{v \in \scenario \setdiff \nodes_{\textnormal{only}}} \rv{v} \leqslant \nrevisions, \quad && \forall \scenario \in \scenarioset, \\
    &\textstyle \sum_{i \leq v, i \notin \nodes_{\textnormal{only}}} r(i) + \Delta(v, h) \leqslant \nrevisions - h + 2^{h} - 1, \quad && \forall v \in \nodes \setdiff \nodes_\nstages, \forall h \in \calH(v,\nrevisions), \\
    & \mathrlap{\Delta(v, h)  \geqslant \Delta(p, h-1) + \Delta(q, h-1)}  \\
    & \hspace{+3.3cm} + (\x{p} - \x{q}), \hspace{+1cm} && \forall v \in \nodes \setminus \nodes_{\nstages}, \forall h \in \calH(v,\nrevisions),\\
    &&& \forall p, q \in \descendant{v}, p \join q = v, \\
    & \Delta(v, h) \geqslant \Delta(u, h),  
    \quad  && \forall v \in \nodes \setminus \nodes_{\nstages}, \forall u \in \children{v}, \\
    &&& \forall h \in \calH(v, \nrevisions) \cap \calH(u, \nrevisions), \\
    & \Delta(v, 0) = 0, \quad && \forall v \in \nodes \\
    &\x{v} \in \Bin, \quad && \forall v \in \nodes, \\  
    & \plan{v,t} \in \Bin,\; \rv{v} \in \Bin, \quad && \forall v \in \nodes \setdiff \nodes_{\textnormal{only}}, \forall t \in [\stage{v}:\nstages].  \\
\end{aligned}
\tag{CP++}
\end{equation}

\subsection{Applicability of decomposition methods}
\label{subsec: decomposition methods}

As discussed in \Cref{sec: related work}, scenario reduction can be applied to reduce the size of the scenario tree and improve the computational tractability of solving extensive form \MSP{s}. Decomposition methods may also improve computational tractability, and can be used either in combination with scenario reduction or as an alternative.
There are many decomposition methods that have been developed for \MSP{}; see \cite{romeijnders2026survey} for a recent survey. 
Many such decomposition approaches are applied to the extensive form deterministic equivalent \cite[e.g.][]{rathi2026column,atakan2018ph,sandikci2017,barntt2017,gade2016,wolf2013dynamic,lulli2004bandp,caroe1999}. Some methods apply to general \MSP{} formulations and can be applied directly to \Cref{eq: cp}. Other methods make the assumption that linking constraints couple each node only with its parent.
\Cref{eq: cp} is not in this form, since the revision budget couples all nodes on a scenario path,  but can be brought into this form by a straightforward modification. In particular, one can add state variables that record the active plan at the start of the node and the number of revisions that have been used so far.
On the other hand, \Cref{eq: st} has constraints that involve multiple distinct nodes of the same stage of the scenario tree, which cannot be expressed through node-parent linking constraints, and does not adhere to the structure assumed in the existing decomposition methods.
Hence, existing decomposition methods for \MSP{} tend not to be applicable to this formulation without modification. 
We leave the development and application of decomposition methods for either formulation as an avenue of future research.

Many decomposition methods have been developed under the assumption of stagewise independence of uncertainty. 
In such settings, it is possible to develop dynamic programming routines that avoid enumeration of the scenario tree, typically variants of the Stochastic Dual Dynamic Programming approach proposed in \cite{pereira1991}; see \cite{fullner2025sddp} for a recent survey. 
Again, most such approaches could be directly applied to a modification of \Cref{eq: cp} wherein the state passed between stages includes the active plan and the remaining revision budget.
Since the resulting state would include binary variables, the recourse function is nonconvex. This would necessitate the use of cuts developed for recourse functions involving integer variables, which has been an area of recent research interest \cite[e.g.][]{bansal2026normalizationreludualcut,fullner2026lagrangian,yang2025enhanced,deng2024relulagrangiancutsstochastic,ahmed2022stochastic,chen2022lagrangian,zou2019}.
As before, since constraints of the subtree formulation couple distinct nodes at the same stage of the scenario tree, existing decomposition methods could not be applied directly, and this formulation would require the development of new decomposition procedures. 
We again leave the development of such decomposition approaches as an avenue of future research. 
In this work, we do not make the stagewise independence assumption. 
Indeed, this assumption is typically violated in the motivating application of air transportation management; see \Cref{subsec: saghp}.

\section{Computational experiments} \label{sec: experiments}
In this section, we conduct computational experiments to 
examine the impact of adding $\nrevisions$-revision constraints in \MSP{} models and to compare different MIP formulations for the $\nrevisions$-revisable set $\revisableset{\nrevisions}$.
To provide a comprehensive comparison, we consider different sizes and structures of scenario trees together with different base problems.
We will start with the $\nrevisions$-revision hypercube problem, which provides a direct assessment of the strength of  different formulations, independent of the presence of additional constraints.
Then, we will test the performance of the $\nrevisions$-revision approach on uncapacitated lot-sizing problems, capacity planning problems, and SAGHP, showcasing its practical applicability and providing further insights into  formulations.

We run the experiments for \Cref{fig:hypercube phase-2}, \Cref{tab: capacity planning}, \Cref{fig: saghp} and \Cref{tab: statistics saghp}, \Cref{fig:saghp 6obj compare}, and \Cref{fig:hypercube penalty} with Gurobi 13.0.0 \cite{gurobi}, and all the other experiments using Gurobi 12.0.0.
Models are formulated using Julia 1.11\cite{bezanson2017julia} with package JuMP 1.23.6 \cite{dunning2017jump} and run on a macOS system (Ventura 13.0) with a hardware configuration of 64 GB RAM and an Apple M1 Max chip.
All formulations except \Cref{eq: st} are solved in their extensive form, with all constraints added up front.
For \Cref{eq: st}, we instead generate the subtree constraints lazily, as detailed in \Cref{sec: experiments - hypercube}.

Our experiments require the construction of scenario trees. 
We mainly consider two types of scenario trees: perfect binary trees (\Btree) and sparse trees with a fixed height and a target number of nodes (\Stree).
Details about how we generate these trees can be found in \Cref{subsec: generate trees}.

\subsection{\texorpdfstring{$\nrevisions$}{K}-revision hypercube problems}
\label{sec: experiments - hypercube}
As the $\nrevisions$-revision hypercube problem has only bounds on the $\x$ variables in addition to the $\nrevisions$-revision constraint, it is free from problem-specific structures and hence offers a strong reference point for comparing the strength of formulations for $\revisableset{\nrevisions}$.

We perform our numerical experiment on both \Btree{s} and \Stree{s}.
We assume that there is a single decision variable for each node as in  \Cref{assumption: 1-dimension}.
Hence, the only data needed is the objective coefficient $c(v)$ for the variable associated with each node $v \in \nodes$.
We generate $c(v)$ uniformly at random  from the integers between $-10$ and $10$.

We conduct experiments to compare our five proposed formulations: \Cref{eq: cp}, \Cref{eq: cp+}, \Cref{eq: cp++}, \Cref{eq: st}, and \Cref{eq: stdp}.
For \Cref{eq: st}, which contains a large number of inequalities, we employ a separation algorithm adapted to $0$-$1$ solutions from that described in \Cref{subsubsec: st separation}. 
The algorithm is implemented via JuMP's lazy constraint callback, generating subtree constraints on the fly; see \Cref{paragraph: lazy constraint st} for details. 
In contrast, all the other formulations are static; we add all of their constraints to the model at the beginning.
We evaluate the performance of each formulation through the following metrics, with all results averaged over 10 distinct instances for each problem configuration.
First, we report the average time to solve the instances to optimality, subject to a 1000-second time limit. 
Second, we present the relative gap of the initial LP relaxation, computed as $\sfrac{|b-f|}{|f|}$, where $b$ is the optimal value of the LP relaxation and $f$ is the optimal value of the integer program. 
Third, we list the number of nodes explored by Gurobi's branch-and-bound algorithm.

Table \ref{tab:hypercube} presents the results of our experiments.
The top section is for \Btree{s} whereas the bottom section is for \Stree{s}.
In the \Btree{} setting, the performance difference between \Cref{eq: cp} and \Cref{eq: cp+} is small, as the reduction techniques used in \Cref{eq: cp+} are most effective for sparse trees.
The results also highlight a clear improvement of \Cref{eq: stdp} over \Cref{eq: st} as trees grow larger. 
Consequently, our subsequent experiments will not consider \Cref{eq: st} further.
In the \Stree{} setting, as the number of nodes is approximately fixed, the solution time of each formulation is relatively stable in response to changes in tree heights.
The results in both sections highlight a trade-off between different formulations. 
The subtree-based formulations, \Cref{eq: st} and \Cref{eq: stdp}, provide stronger relaxations for smaller trees. 
However, their strength diminishes relative to complete plan type formulations as the tree size increases, with \Cref{eq: cp+} showing more stable performance.
As expected, the combined formulation \Cref{eq: cp++} is consistently the strongest, as it integrates the strengths of both types of formulation.
However, it does not always result in the fastest solution times as it also has a large size.

\begin{sidewaystable}
\setlength{\belowcaptionskip}{0pt}
\caption{Computational results on $\nrevisions$-revision hypercube problems.
} \label{tab:hypercube}
\subcaption*{\footnotesize\textbf{(a)} Perfect binary trees (\Btree{s}).}
\setlength{\tabcolsep}{2.33mm}{
\begin{tabular}{cccccccccccccccccc}
\toprule
\multirow{2}{*}{$\nrevisions$} & \multirow{2}{*}{$\nstages$} & \multirow{2}{*}{$|\nodes|$} 
 & \multicolumn{5}{c}{time}  
 & \multicolumn{5}{c}{relative gap}
 & \multicolumn{5}{c}{node count}\\
\cmidrule(r){4-8}\cmidrule(r){9-13}\cmidrule(r){14-18}
 & & 
 & CP & CP+ & CP++ & ST & STDP  
 & CP & CP+ & CP++ & ST & STDP  
 & CP & CP+ & CP++ & ST & STDP \\
\midrule
\multirow{7}{*}{1}
& 3 & 7 
 & 0.0004 & 0.0004 & 0.0006 & \textbf{0.0003} & 0.0003
 & 3.5 & 3.5 & \textbf{0.0} & 0.0 & 0.0
 & \textbf{0.0} & 0.0 & 0.0 & 0.1 & 0.0 \\
& 4 & 15  
 & 0.001 & 0.001 & 0.003 & \textbf{0.0007} & 0.001
 & 7.2 & 7.2 & \textbf{0.0} & 0.0 & 0.0
 & \textbf{0.3} & 0.3 & 1.0 & 0.8 & 0.5 \\
& 5 & 31 
 & 0.005 & 0.005 & 0.006 & 0.009 & \textbf{0.004}
 & 10.0 & 10.0 & \textbf{0.0} & 0.4 & 0.4
 & \textbf{1.0} & 1.0 & 1.0 & 45.5 & 1.0 \\
& 6 & 63 
 & 0.01 & \textbf{0.01} & 0.02 & 0.2 & 0.03
 & 8.9 & 8.9 & \textbf{0.08} & 3.1 & 3.1
 & \textbf{1.0} & 1.0 & 1.0 & 2485.7 & 1.0 \\
& 7 & 127 
 & 0.04 & \textbf{0.04} & 0.09 & 4.8 & 0.1
 & 9.0 & 9.0 & \textbf{0.2} & 4.1 & 4.1
 & 3.8 & 3.8 & \textbf{1.0} & 19190.3 & 5.7 \\
& 8 & 255 
 & 0.1 & \textbf{0.1} & 0.6 & 474.9 & 2.2
 & 10.8 & 10.8 & \textbf{0.4} & 6.9 & 6.9
 & 11.2 & 11.2 & \textbf{1.0} & 406004.5 & 23.0 \\
& 9 & 511 
 & 0.4 & \textbf{0.4} & 6.8 & 1000 & 27.9
 & 12.8 & 12.8 & \textbf{1.7} & 11.5 & 11.5
 & 56.1 & 56.1 & \textbf{3.0} & 816934.0 & 149.6 \\
\midrule
\multirow{6}{*}{2}
& 4 & 15  
 & 0.0005 & 0.0005 & 0.001 & \textbf{0.0004} & 0.0005
 & \textbf{0.0} & 0.0 & 0.0 & 0.0 & 0.0
 & \textbf{0.0} & 0.0 & 0.0 & 0.0 & 0.0 \\
& 5 & 31
 & 0.004 & 0.003 & 0.004 & \textbf{0.001} & 0.002
 & 0.9 & 0.9 & \textbf{0.0} & 0.0 & 0.0
 & 0.8 & 0.8 & 1.0 & \textbf{0.4} & 1.0 \\
& 6 & 63
 & 0.01 & 0.01 & 0.01 & 0.03 & \textbf{0.007}
 & 3.0 & 3.0 & \textbf{0.0} & 0.03 & 0.03
 & \textbf{1.0} & 1.0 & 1.0 & 125.9 & 1.0 \\
& 7 & 127
 & 0.06 & 0.06 & \textbf{0.05} & 7.1 & 0.05
 & 3.7 & 3.7 & \textbf{0.08} & 0.6 & 0.6
 & 53.1 & 53.1 & \textbf{1.0} & 13333.2 & 1.0 \\
& 8 & 255 
 & 0.2 & \textbf{0.2} & 0.3 & 605.8 & 0.6
 & 5.5 & 5.5 & \textbf{0.2} & 1.2 & 1.2
 & 181.7 & 181.7 & \textbf{1.0} & 251072.6 & 17.6 \\
& 9 & 511
 & 0.5 & \textbf{0.5} & 1.8 & 1000 & 5.2
 & 6.3 & 6.3 & \textbf{0.4} & 2.4 & 2.4
 & 420.1 & 420.1 & \textbf{1.0} & 584415.8 & 26.9 \\
\bottomrule
\end{tabular}
}

\vspace{2ex}

\subcaption*{\footnotesize\textbf{(b)} Sparse trees with a target number of nodes (\Stree{s}).}
\vspace{0.5ex}
\setlength{\tabcolsep}{2.52mm}{
\begin{tabular}{cccccccccccccccccc}
\toprule
\multirow{2}{*}{$\nrevisions$} & \multirow{2}{*}{$\nstages$} & \multirow{2}{*}{$|\nodes|$}
  & \multicolumn{5}{c}{time}
  & \multicolumn{5}{c}{relative gap}
  & \multicolumn{5}{c}{node count}\\
\cmidrule(r){4-8}\cmidrule(r){9-13}\cmidrule(r){14-18}
 & & 
 & CP & CP+ & CP++ & ST & STDP 
 & CP & CP+ & CP++ & ST & STDP  
 & CP & CP+ & CP++ & ST & STDP  \\
\midrule
\multirow{6}{*}{1}
 & 3  & \multirow{6}{*}{200}
     & 0.02 & \textbf{0.01} & 0.1 & 9.2 & 0.1
     & 10.5 & 10.5 & \textbf{0.0} & 0.5 & 0.5
     & \textbf{1.0} & 1.0 & 1.0 & 23916.6 & 1.0 \\
 & 5  & 
     & 0.04 & \textbf{0.04} & 0.3 & 46.8 & 0.6
     & 10.8 & 10.8 & \textbf{0.6} & 3.0 & 3.0
     & 7.3 & 2.4 & \textbf{1.0} & 100039.5 & 1.0 \\
 & 7  & 
     & \textbf{0.07} & 0.07 & 0.3 & 41.1 & 0.5
     & 10.9 & 10.9 & \textbf{0.1} & 3.0 & 3.0
     & 9.3 & 5.9 & \textbf{1.0} & 75479.4 & 1.0 \\
 & 9  & 
     & 0.06 & \textbf{0.06} & 0.2 & 37.9 & 0.3
     & 6.6 & 6.6 & \textbf{0.06} & 1.8 & 1.8
     & 4.2 & 8.6 & \textbf{1.0} & 97124.8 & 1.0 \\
 & 11 &
     & 0.07 & \textbf{0.06} & 0.2 & 45.9 & 0.2
     & 10.6 & 10.6 & \textbf{0.9} & 4.4 & 4.4
     & 9.8 & 17.3 & \textbf{1.0} & 110043.1 & 1.0 \\
 & 13 &
     & 0.07 & \textbf{0.06} & 0.1 & 35.1 & 0.1
     & 5.9 & 5.9 & \textbf{0.4} & 1.0 & 1.0
     & 3.6 & 7.8 & \textbf{1.0} & 72520.4 & 1.0 \\
 \midrule
 \multirow{6}{*}{2}
 & 3  & \multirow{6}{*}{200}
     & 0.006 & \textbf{0.006} & 0.04 & 0.5 & 0.05
     & 1.1 & 1.1 & \textbf{0.0} & 0.0 & 0.0
     & 0.7 & 0.7 & 1.0 & 1249.8 & \textbf{0.4} \\
 & 5  & 
     & \textbf{0.06} & 0.07 & 0.09 & 276.6 & 0.09
     & 3.6 & 3.6 & \textbf{0.03} & 0.1 & 0.1
     & 15.9 & 26.9 & \textbf{1.0} & 121215.2 & 1.0 \\
 & 7  & 
     & 0.1 & \textbf{0.09} & 0.1 & 61.5 & 0.1
     & 3.2 & 3.2 & \textbf{0.06} & 0.1 & 0.1
     & 255.6 & 67.5 & \textbf{1.0} & 52203.5 & 1.0 \\
 & 9  & 
     & 0.09 & \textbf{0.08} & 0.1 & 153.2 & 0.1
     & 2.4 & 2.4 & \textbf{0.07} & 0.2 & 0.2
     & 50.3 & 35.0 & \textbf{1.0} & 76567.6 & 1.0 \\
 & 11  & 
     & 0.1 & \textbf{0.08} & 0.09 & 219.4 & 0.1
     & 3.1 & 3.1 & \textbf{0.1} & 0.2 & 0.2
     & 35.3 & 34.0 & \textbf{1.0} & 121281.7 & 1.0 \\
 & 13 &
     & 0.08 & 0.06 & 0.07 & 2.6 & \textbf{0.06}
     & 0.7 & 0.7 & \textbf{0.004} & 0.006 & 0.006
     & 3.4 & \textbf{1.0} & 1.0 & 5284.9 & 1.0 \\
\bottomrule
\end{tabular}
}
\begin{minipage}{\textwidth}
\footnotesize
Note. The column $|\nodes|$ represents the exact number of nodes in the scenario tree for part (a) and represents the target number of nodes from which actual instances can deviate by $\pm 5\%$ for part (b).
The column “time” reports the solution time in seconds.
The column “relative gap” gives the percentage optimality gap between the LP relaxation and the IP formulation.
The column “node count” indicates the number of nodes explored in the branch-and-bound algorithm.
All reported values are averages over ten instances.
\end{minipage}
\end{sidewaystable}

\begin{figure}[htbp]
\centering
\includegraphics[width=\linewidth]{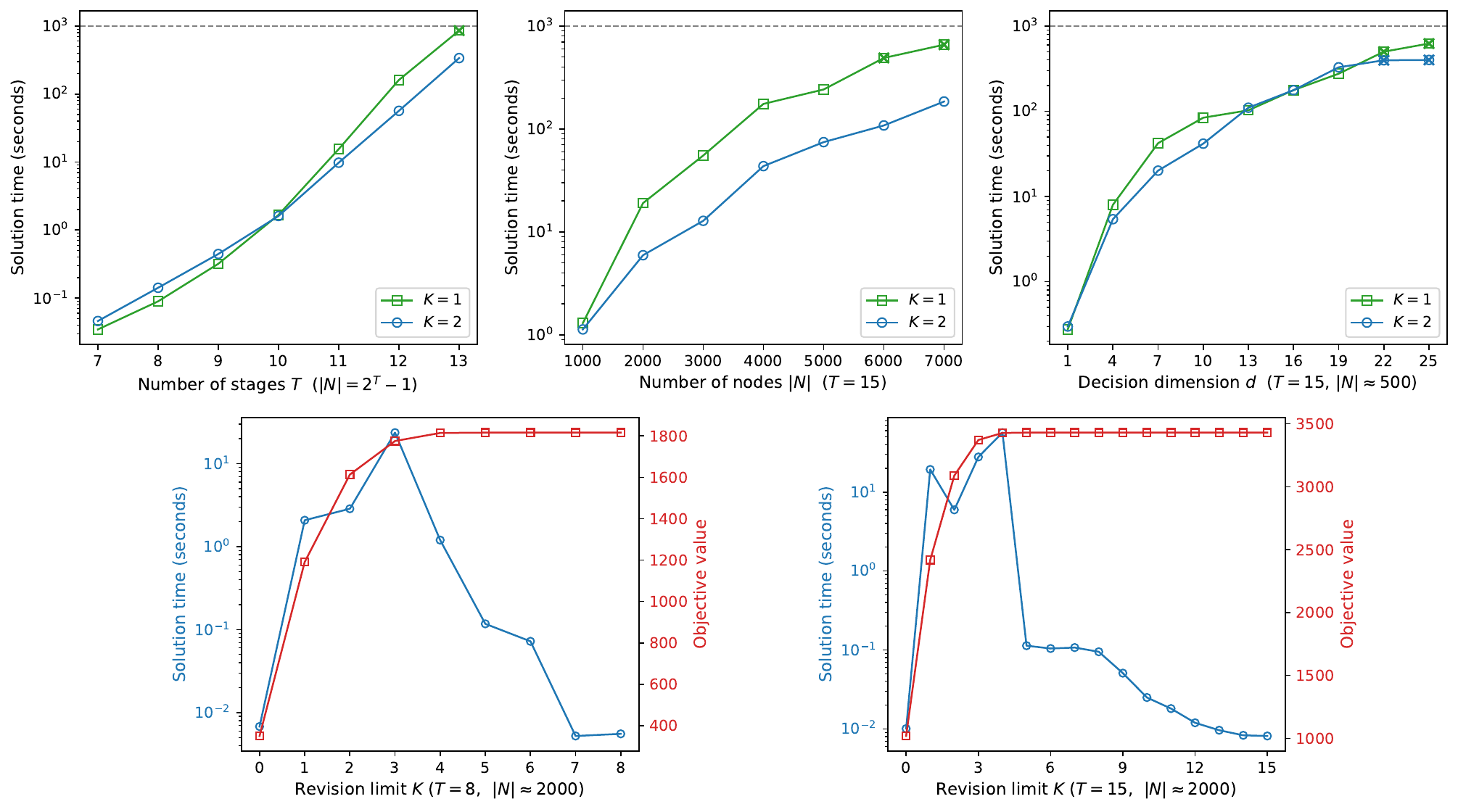}
\vspace{-1em}
\caption{Dependence of the best solution time on problem parameters.}
\label{fig:hypercube phase-2}
\begin{minipage}{\textwidth}
\footnotesize
Note. From left to right and top to bottom, we consider (1) \Btree{s} when $\nstages$ varies, (2) \Stree{s} when $\nstages=15$ and $|\nodes|$ varies, (3) a fixed tree ($\nstages=15$, $
|\nodes|\approx 500$) with the decision dimension $d$ at all stages varies, and (4)(\emph{resp.} (5)) a fixed instance with $\nstages=8$ (\emph{resp.} $\nstages=15$), $
|\nodes|\approx 2000$, and $d=1$, where $\nrevisions$ varies, reporting solution time on the left axis and the optimal objective on the right.
Each data point corresponds to an average over ten instances.
Time limit is set to 1000 seconds and a data point marked with an extra cross means it contains at least one time-out instance.
\end{minipage}
\end{figure}

\Cref{fig:hypercube phase-2} further explores the computational limits of the best performing formulation, which is \Cref{eq: cp+} in this case.
The top three plots show the change of the solution time when the problem scales up.
We observe that (i) the solution time grows exponentially with $\nstages$ (first panel), (ii) the solution time grows superlinearly with $|\nodes|$  (second panel), and (iii) increasing the decision dimension $d$ is less expensive than increasing $|\nodes|$ by the same factor (third panel). 
The bottom two panels show, for two distinct instances, that the solution time is non-monotone in $\nrevisions$. 
The hard instances are those with $1 \leqslant \nrevisions \leqslant 4$. 
In fact, the objective already reaches its fully adaptive value around $\nrevisions=4$, which makes instances with $\nrevisions \geqslant 5$ easier to solve.

\subsection{Uncapacitated lot-sizing problems}

We next consider instances where the stochastic uncapacitated lot-sizing problem is used as the base model; see \Cref{subsec: lot-sizing formulation} for the mathematical formulation and how we generate instances.
In the model, we impose an $\nrevisions$-revision constraint on the binary variables $\x$ indicating whether production occurs in a node or not.

We use \Cref{eq: cp}, \Cref{eq: cp+}, \Cref{eq: cp++}, and \Cref{eq: stdp} as the MIP formulations for the $\nrevisions$-revisable set.
\Cref{tab:ls-results} demonstrates that \Cref{eq: cp+} is the fastest formulation and that there is a substantial reduction in the running time for \Cref{eq: cp+} compared to \Cref{eq: cp}, showing the effectiveness of the size reduction technique.
We also solve the problem without the $\nrevisions$-revision constraint. 
Compared to the base problem (LS), \Cref{eq: cp+} is slower but remains within a reasonable range.
These results also corroborate our earlier observations that \Cref{eq: stdp} and \Cref{eq: cp++}, which both contain subtree constraints implicitly, tend to be faster when $\nrevisions$ grows. 
This aligns with the fact that subtree formulations are stronger when $\nrevisions$ is close to $\nstages$.
Finally, since the uncapacitated lot-sizing problem without the $\nrevisions$-revision constraint solves quickly, a natural question is to investigate whether the $\nrevisions$-revision constraint can be handled through penalties in the objective.
We study such penalty models and compare them computationally with the $\nrevisions$-revision model in \Cref{sec: penalty models}.

\begin{table}[htbp]
\centering
\caption{Experiments on uncapacitated lot-sizing problems.}
\label{tab:ls-results}
\setlength{\tabcolsep}{1.87mm}{
\begin{tabular}{ccccccccccccc}
\toprule
\multirow{2}{*}{$\nstages$}
 & \multirow{2}{*}{$|\nodes|$}
 & \multirow{2}{*}{$\nrevisions$}
 & \multicolumn{5}{c}{time}
 & \multicolumn{5}{c}{node count} \\
\cmidrule(lr){4-8}\cmidrule(lr){9-13}
 & &
 & MSP & CP & CP+ & CP++ & STDP
 & MSP & CP & CP+ & CP++ & STDP \\
\midrule
\multirow{2}{*}{6} & \multirow{2}{*}{108}
 & 1 & \multirow{2}{*}{0.02} & 0.10 & \textbf{0.08} & 0.21 & 0.14
     & \multirow{2}{*}{1.0} & 18.5 & 23.0 & \textbf{7.2} & 8.4 \\
 & & 2 & & 0.05 & \textbf{0.04} & 0.09 & 0.08
     & & \textbf{1.0} & \textbf{1.0} & \textbf{1.0} & \textbf{1.0} \\
\addlinespace[3pt]
\multirow{2}{*}{8} & \multirow{2}{*}{192}
 & 1 & \multirow{2}{*}{0.05} & 0.42 & \textbf{0.33} & 1.34 & 1.04
     & \multirow{2}{*}{48.6} & 364.5 & 304.5 & 124.2 & \textbf{95.1} \\
 & & 2 & & 0.30 & \textbf{0.23} & 0.68 & 0.66
     & & 489.9 & 340.2 & 51.8 & \textbf{23.2} \\
\addlinespace[3pt]
\multirow{2}{*}{10} & \multirow{2}{*}{300}
 & 1 & \multirow{2}{*}{0.09} & 1.47 & \textbf{0.87} & 4.07 & 3.86
     & \multirow{2}{*}{49.6} & 1237.1 & 650.1 & \textbf{252.0} & 379.1 \\
 & & 2 & & 1.29 & \textbf{0.87} & 1.78 & 1.75
     & & 2418.3 & 1760.1 & 384.9 & \textbf{182.3} \\
\addlinespace[3pt]
\multirow{2}{*}{12} & \multirow{2}{*}{432}
 & 1 & \multirow{2}{*}{0.13} & 4.09 & \textbf{2.01} & 15.54 & 13.62
     & \multirow{2}{*}{171.0} & 2454.5 & 1847.6 & \textbf{986.1} & 1065.5 \\
 & & 2 & & 5.40 & \textbf{3.03} & 7.04 & 6.62
     & & 4486.7 & 4332.6 & 986.4 & \textbf{773.7} \\
\addlinespace[3pt]
\multirow{2}{*}{14} & \multirow{2}{*}{588}
 & 1 & \multirow{2}{*}{0.24} & 12.11 & \textbf{5.50} & 58.04 & 38.43
     & \multirow{2}{*}{580.4} & 5096.6 & 3133.3 & \textbf{1545.1} & 1773.2 \\
 & & 2 & & 12.23 & \textbf{6.14} & 29.42 & 28.05
     & & 7084.1 & 8937.9 & 3538.3 & \textbf{2948.0} \\
\addlinespace[3pt]
\multirow{2}{*}{16} & \multirow{2}{*}{768}
 & 1 & \multirow{2}{*}{0.50} & 89.67 & \textbf{29.82} & 309.43 & 341.02
     & \multirow{2}{*}{2366.7} & 8622.1 & 6613.1 & 7155.8 & \textbf{5421.9} \\
 & & 2 & & 49.57 & \textbf{18.32} & 128.34 & 161.78
     & & 9441.5 & \textbf{6462.6} & 6572.3 & 7663.4 \\
\bottomrule
\end{tabular}
}
\begin{minipage}{\textwidth}
\footnotesize{Note. The ``MSP'' columns refer to the base MSP model.
The column $|\nodes|$ represents the target number of nodes from which actual instances can deviate by $\pm5 \%$.  
All reported values are averages over ten instances.}
\end{minipage}
\end{table}

\subsection{Capacity planning problems} \label{subsec: capacity planning}

Capacity planning problems involve making decisions regarding acquiring, expanding, or reallocating productive or service capacities over a multi-period planning horizon \cite{huang2005, huang2009}. 
These decisions are typically capital-intensive and have long-term implications, making robust planning under uncertainty essential.
\MSP{}  is thus a suitable framework for this problem, allowing for capacity adjustments in response to evolving conditions. 
Here, we consider a variation of the semiconductor tool planning model studied in \cite{huang2005}, which is challenging to solve even when the tree is small. 
To keep the setting realistic, we follow the demand process, problem dimensions, and the scenario tree size in \cite{huang2005}.
Details of the formulation and instance generation can be found in \Cref{subsec: capacity planning formulation}.
We generate \Stree{s}.
Because the trees we generate are small, we report the exact number of nodes they have. 

\Cref{tab: capacity planning} shows computational results. 
Firstly, we observe that the relative loss of the $\nrevisions$-revision over fully adaptive policies is much smaller than the relative value over a static policy, which indicates that the $\nrevisions$-revision constraints help create solutions that are more consistent across scenarios without significantly affecting the optimal value.
Secondly, we see that \Cref{eq: stdp} 
is faster than \Cref{eq: cp+} for this problem in almost all configurations except the largest size with $K=1$.
This is expected as \Cref{eq: stdp} is stronger for short trees.

\begin{table}[htbp]
 \centering
 
 \begin{minipage}{0.95\linewidth}
\centering
\captionsetup{justification=raggedright,singlelinecheck=false}
 \caption{Computational results for capacity planning problems.}
\setlength{\tabcolsep}{1.46mm}{
\begin{tabular}{ccccccccccccc}
\toprule
\multirow{2}{*}{$\nstages$}
 & \multirow{2}{*}{$|\nodes|$}
 & \multirow{2}{*}{$\nrevisions$}
 & \multirow{2}{*}[-2pt]{\shortstack{loss \% \\ (MSP)}}
 & \multirow{2}{*}[-2pt]{\shortstack{value \% \\ (static)}}
 & \multicolumn{5}{c}{time}
 & \multicolumn{3}{c}{node count} \\
\cmidrule(lr){6-10}\cmidrule(lr){11-13}
 & & & &
 & static & MSP & CP+ & CP++ & STDP
 & CP+ & CP++ & STDP \\
\midrule
3 & 7 & 1 & 0.00 & 1.43 & 0.5 & 0.5 & 0.6 & 0.5 & \textbf{0.5} & \textbf{7.4} & 27.6 & 19.4 \\
\addlinespace[3pt]
\multirow{2}{*}{4} & \multirow{2}{*}{15}
 & 1 & 0.14 & 2.16 & \multirow{2}{*}{3.8} & \multirow{2}{*}{3.7}
 & 8.7 & 8.7 & \textbf{8.0} & 2037.0 & 1590.0 & \textbf{1462.0} \\
 & & 2 & 0.00 & 2.30 & & &
   4.4 & 4.1 & \textbf{3.8} & 800.8 & 801.8 & \textbf{500.0} \\
\addlinespace[3pt]
\multirow{2}{*}{4} & \multirow{2}{*}{20}
 & 1 & 0.26 & 2.15 & \multirow{2}{*}{8.5} & \multirow{2}{*}{9.6}
 & 64.8 & 78.4 & \textbf{61.9} & \textbf{4537.0} & 5649.2 & 5890.2 \\
 & & 2 & 0.00 & 2.40 & & &
   39.7 & 43.1 & \textbf{33.2} & 5725.2 & 5751.4 & \textbf{4928.6} \\
\addlinespace[3pt]
\multirow{2}{*}{5} & \multirow{2}{*}{25}
 & 1 & 0.36 & 1.77 & \multirow{2}{*}{14.2} & \multirow{2}{*}{26.4}
 & 211.4 & \textbf{155.0} & 183.0 & 13988.8 & \textbf{9453.4} & 10769.0 \\
 & & 2 & 0.02 & 2.09 & & &
   78.0 & 121.7 & \textbf{73.2} & 9733.4 & 11777.0 & \textbf{7548.4} \\
\addlinespace[3pt]
\multirow{2}{*}{5} & \multirow{2}{*}{31}
 & 1 & 0.30 & 1.96 & \multirow{2}{*}{30.1} & \multirow{2}{*}{241.5}
 & 513.3 & 568.3 & \textbf{455.6} & 19573.6 & 24180.8 & \textbf{16481.2} \\
 & & 2 & 0.02 & 2.24 & & &
   485.8 & 477.7 & \textbf{457.3} & 44918.0 & 41100.6 & \textbf{36244.2} \\
\addlinespace[3pt]
\multirow{2}{*}{6} & \multirow{2}{*}{36}
 & 1 & 0.33 & 1.41 & \multirow{2}{*}{44.4} & \multirow{2}{*}{73.6}
 & \textbf{708.4} & 879.6 & 905.9 & \textbf{22341.2} & 26738.0 & 24065.2 \\
 & & 2 & 0.02 & 1.71 & & &
   452.7 & 435.5 & \textbf{385.4} & 21062.0 & 19962.8 & \textbf{19938.4} \\
\bottomrule
\end{tabular}
    }
    \label{tab: capacity planning}
\begin{minipage}{\textwidth}
\vspace{+1pt}
\footnotesize{
Note. The ``loss (\emph{resp.} value)'' column shows the relative loss (\emph{resp.} value) of the $\nrevisions$-revision in percentage, compared with the MSP (\emph{resp.} static policy) model.
 The ``time'' column shows the solution time (in seconds), with a time limit set to 1000s.
All reported values are averages over five instances.}
\end{minipage}
\end{minipage}
\end{table}

\subsection{Single airport ground-holding program (SAGHP) problems} \label{subsec: saghp}
We next study instances of SAGHP to investigate the effect of $\nrevisions$-revision constraints on problems with real-world data. 
Specifically, we compare the optimal values of the \MSP{} model without revision constraints, the $1$-revision model, 
and a $1$-PA model (see \Cref{sec: effect of revisions}).
In the second experiment, we also include the static, $2$-revision, and $2$-PA models.
The base model we use for SAGHP is that introduced by \cite{mukherjee2007, estes2020}, which can be found in \Cref{subsec: SAGHP formulation}.
We impose the $\nrevisions$-revision constraint on variables representing whether flights depart at each node or not.

We use real-world flight data from the San Francisco International Airport (SFO), the Newark Liberty International Airport (EWR), and the O'Hare International Airport (ORD).
We choose these three airports because they frequently experience a large number of ground delays.
Flight data are obtained from the Aviation System Performance Metrics (ASPM) database \cite{aspmdata}. 
All flights arriving at the studied airport on July 1st, 2024, within the designated planning horizon were included.
In addition, the weather-dependent airport arrival capacity data are taken from the FAA Airport Capacity Profiles \cite{capacitydata}, and the weather-evolution scenario trees are constructed in a similar way to those in \cite{mukherjee2007, estes2020}. 
Hence, our instances provide a realistic, established test setting for SAGHP.

The scenario tree captures the randomness of weather changes, and different weather conditions lead to different airport capacities.
We consider meteorological condition codes: ``visual''(V), ``marginal''(M), and ``instrument''(I). 
Each airport has a corresponding airport arrival capacity for each code, with capacity under visual conditions larger than that under marginal conditions, and capacity under marginal conditions larger than that under instrument conditions.

In each problem instance, there is a known  sequence of meteorological condition codes that will occur in the given order over the planning horizon, but with uncertain timing of the transition from one code in the sequence to the next. For example, one possible sequence may be VIV which would indicate the following: the meteorological condition code at the initial stage is V; the meteorological condition code will remain V for some unknown number of stages until severe weather impacts the airport, at which point the condition code will transition to I; the meteorological condition code will then remain I for some unknown number of stages until the weather improves, at which point the condition code will become V again; the condition code would then remain V for all remaining stages.

From one stage to the next, the probability of transition from one code to the next depends only on the sequence of codes that has been observed so far. In order to prevent excessively abrupt changes in weather, we permit transitions only every other stage. For example, if the current code in some even-numbered stage is I, then the code in the subsequent stage will also be I.
For each odd-numbered stage $t$, the transition probabilities are determined by parameters $q_+$ and $q_-$, representing the probability of weather improving and worsening respectively. For example, consider the VIV sequence of condition codes. If the code in some odd-numbered stage is V and the code I has not yet been observed, then the code in the next stage has a $q_-$ probability of transitioning to I and a $1-q_-$ probability of remaining V. Similarly, if the code in some odd-numbered stage is I, then the code in the next stage has a $q_+$ probability of transitioning to V and a $1-q_+$ probability of remaining as I.
Unless stated otherwise, we set $q_+=q_-=0.5$.
A scenario tree can be formed from this process in a standard way: in each stage $t$, create a node $v$ for each realization of the uncertainty up to the stage $t$; a node $v'$ in the subsequent stage $t+1$ is a child of $v$ if and only if the realizations corresponding to the two nodes agree up to the stage $t$.  In our experiments, we make use of three meteorological condition code sequences: VIV, IMV, and VIMV.
Further details of how the scenario tree is generated can be found in \Cref{subsec: SAGHP formulation}.

\begin{figure}
    \centering
    \includegraphics[width=0.9\linewidth]{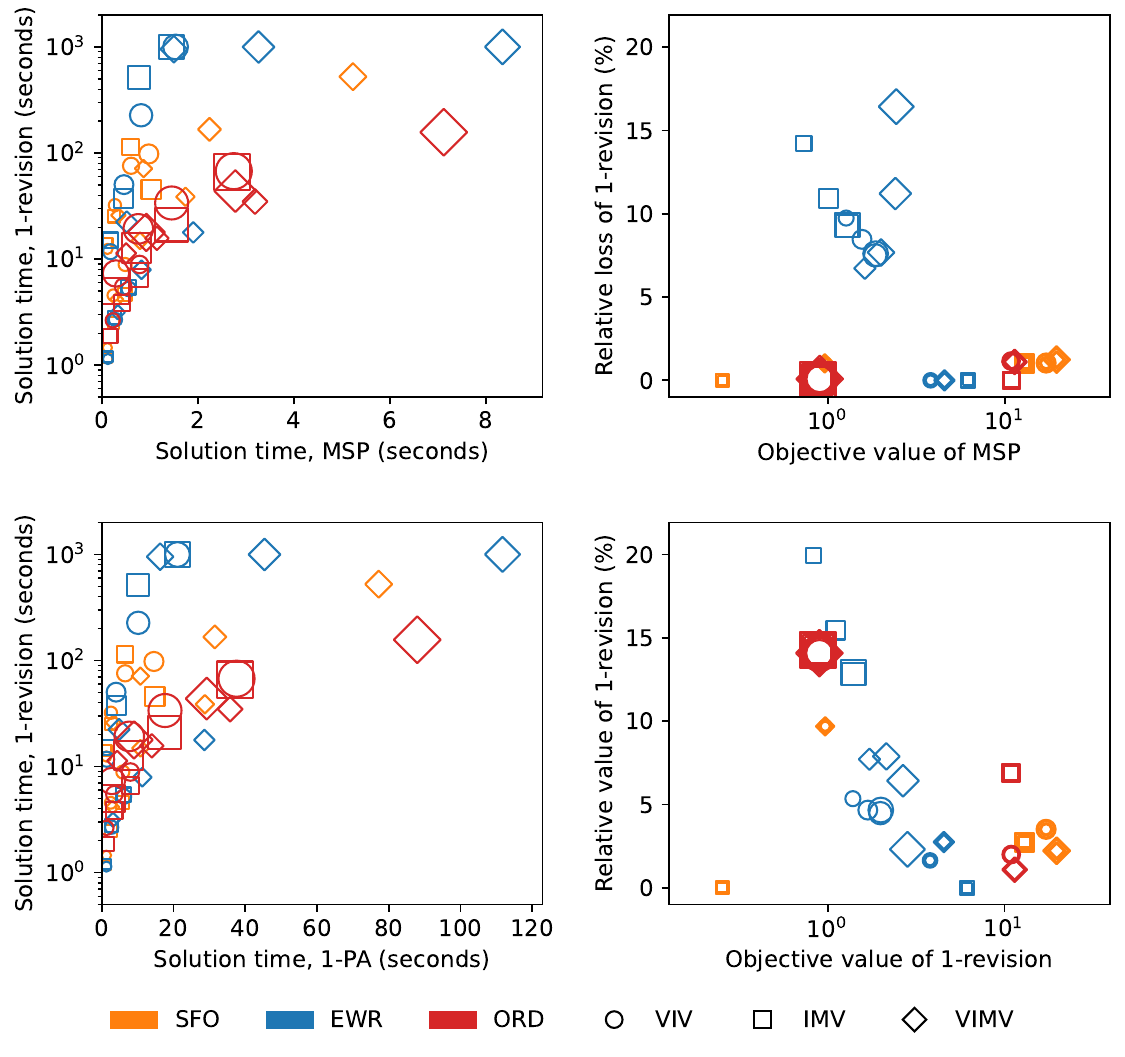}
    \caption{Computational results on SAGHP.}
      \begin{minipage}{\textwidth}
        \footnotesize
        Note. The top row compares MSP vs 1-revision: runtime (left panel) and relative loss of $1$-revision with respect to \MSP{} (right panel); the bottom row compares 1-PA vs 1-revision: runtime (left panel) and relative value of $1$-revision with respect to $1$-PA. Marker color encodes airport; marker shape encodes meteorological condition code sequence; marker size is proportional to the number of flights. A 1000-second time limit is imposed; if hit, the reported objective is that of the best solution found.
      \end{minipage}
      \label{fig: saghp}
    \label{fig:placeholder}
\end{figure}

\Cref{fig: saghp} graphically displays the results of the experiments for SAGHP on 62 instances that are obtained from different meteorological condition code sequences, different airports, and different time horizons, where the largest solved instance contains $4340$ scenario tree nodes and up to  $237$ strategic decisions in a single stage; see \Cref{subsec: SAGHP formulation}.
The top row shows the comparison between the original SAGHP problem formulated using \MSP{} models and that of SAGHP with 1-revision constraints.
The top left panel plots the solution time of the two models and the top right panel plots the relative loss of $1$-revision with respect to \MSP{}.
We observe that the 1-revision model is much harder to solve than the original  \MSP{} model.
However, the right panel establishes that the objective value achieved by the 1-revision model is close to that of the original problem for most instances, even for instances where the 1-revision model is not solved to optimality due to the imposed time limit. 
This implies that limiting revisions causes little degradation in objective value for most SAGHP instances.

In the bottom row, we compare the 1-revision SAGHP problem with the $1$-PA.
The left panel shows that solving the $1$-PA model is faster than solving the 1-revision model, but much slower than solving the original model.
The right panel shows that the relative value of 1-revision with respect to $1$-PA is generally larger than its relative loss with respect to \MSP{}.
Remarkably, for the ORD instances with the most flights, the bottom right panel of \Cref{fig: saghp} shows that the relative value of 1-revision is around $14$\%, whereas the top right panel shows that the relative loss is zero.

\Cref{tab: statistics saghp} further demonstrates the value of 1-revision generally outweighs its loss by reporting summary statistics computed over all 62 instances.

\begin{table}[htbp]
\centering
\begin{minipage}{0.8\linewidth}
\caption{Statistics of the 62 SAGHP instances in \Cref{fig: saghp}.}
\setlength{\tabcolsep}{5mm}{
\begin{tabular}{lcc}
\toprule
Statistic & Value of 1-revision & Loss of 1-revision \\
\midrule
Mean                        & 6.38\%  & 2.35\%  \\
Median                      & 4.00\%  & 1.02\% \\
Maximum                     & 19.9\%  & 16.4\%  \\
\# of instances = 0      & 6       & 19           \\
\# of instances in $(0,1\%)$  & 0      & 7               \\
\# of instances in $[1\%,10\%)$  & 41      & 32             \\
\# of instances $\geqslant 10\% $  & 15      & 4             \\
\bottomrule
\end{tabular}}
\label{tab: statistics saghp}
\end{minipage}
\end{table}

\begin{figure}[tbp]
\centering
\includegraphics[width=\linewidth]{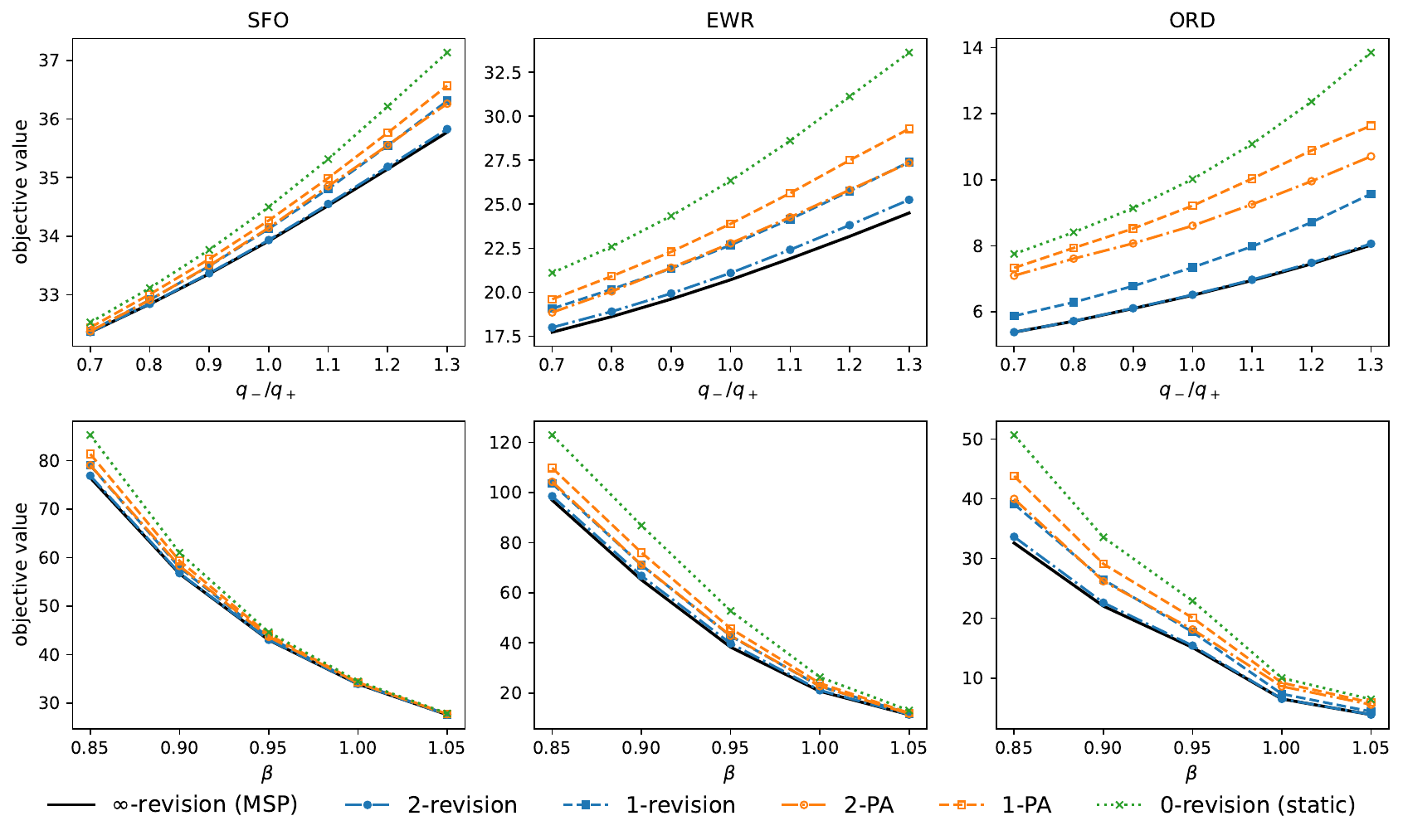}
\vspace{-1em}
\caption{Model objective comparisons on SAGHP with meteorological condition code sequence VIMV.} 
\label{fig:saghp 6obj compare}
\begin{minipage}{\textwidth}
\footnotesize
Note. The first and second rows show how the objective value changes as $q_-/q_+$ and $\beta$ vary, \emph{resp.} Each column shows the results from a single instance generated for SFO, EWR, and ORD, \emph{resp.}, with a 4-hour window that has the most arriving flights at each airport. We fix $\beta=1$ and $q_-=q_+=0.5$ for the $q$-sweep and $\beta$-sweep respectively. For a given $q_-/q_+$ value $\theta$, the actual $q_+$ and $q_-$ are set to $0.5$ and $0.5 \cdot \theta$, \emph{resp.} The y-axis objective value is the total cost of ground and air holding.
\end{minipage}
\end{figure}

Finally, we study how the loss and value of $\nrevisions$-revision change with different problem parameters.
Intuitively, revisions matter most when the weather is volatile and capacity is tight.
We thus vary the ratio $q_-/q_+$, which controls how likely the weather is to deteriorate rather than improve.
In addition to the probability parameters $q_+$ and $q_-$, we consider a factor $\beta$ that scales the original ground capacity for every node whose meteorological condition code is M or I, leaving the capacity under V unchanged.
\Cref{fig:saghp 6obj compare} presents the results of this analysis for meteorological code sequence VIMV.
The vertical gap between two curves is the absolute value/loss between two models.
The value of $K$-revision, compared to both $K$-PA and the static model, is significantly larger than the loss compared to \MSP{}; these gaps widen as $q_-/q_+$ grows but nearly vanish with $\beta=1.05$.
Interestingly, $1$-revision models almost always achieve a better objective than $2$-PA models, even though the latter allow more revisions per scenario. This demonstrates the value of allowing the timing of the revision to respond dynamically to the evolving uncertainty.

\section{Conclusion} \label{sec: conclusion}
We introduced the $\nrevisions$-revision approach for \MSP{}, and established its theoretical underpinnings with respect to complexity, formulations, and impact on objectives.
Our theoretical analysis and computational results suggest that our formulations work best in complementary settings: complete plan formulations are strong and efficient for large scenario trees, whereas subtree formulations are particularly strong for short and small trees.
Beyond the technical contributions, our results highlight the value of controlling the trade-off between flexibility and predictability in practical \MSP{} models, offering new insights for future research on sequential decision-making under uncertainty.

\section*{Acknowledgement}
The material in this paper is based upon work supported by the Air Force Office of Scientific Research under award number FA9550-23-1-0451.
We thank the review team for their constructive feedback, which led us to substantially improve the paper.

\printbibliography

\begin{appendices}



\makeatletter
\newenvironment{appendixproof}{%
  \pushQED{\qed}%
  \normalfont
  \topsep6\p@\@plus6\p@ \trivlist
  \item\relax
}{%
  \popQED\endtrivlist\@endpefalse
}
\makeatother

\crefalias{section}{appendix}
\crefalias{subsection}{appendix}

\section{MSP problem formulations and instance generations} \label{sec: problem formulations}

This section details the problem formulations used to evaluate the $\nrevisions$-revision approach, alongside the procedures for generating our test instances. 
Since problem difficulty often depends on scenario tree structure, we begin by describing two procedures for generating them.

\subsection{Scenario tree generation}
\label{subsec: generate trees}
We generate perfect binary trees (\Btree{s}) parameterized by their number of stages $\nstages$. 
These trees have $2^{\nstages}-1$ nodes and we set the probability of each scenario to be $\sfrac{1}{2^{\nstages-1}}$.
We generate sparse trees (\Stree{s}) with a specified target number of nodes $|\nodes|$ and height $\nstages$ recursively from the root.
At each step of the recursive procedure, we randomly select a node without children. 
If the node is at level $\nstages$, we move on to the next step without taking action.
Otherwise, we generate a number of children randomly from the distribution $\max\{\Binomial(m,\varrho),1\}$, \emph{i.e.}, we first sample from a binomial distribution and then take the maximum with $1$.
Here, $m$ is a parameter that describes the maximum number of children (which we set to be $3$) and $\varrho$ is the probability of keeping a child.
The construction process terminates when all the leaf nodes are at stage $\nstages$.
In our computational experiments, we will adjust the value of $\varrho$ depending on the target number of nodes and the height of the trees we want to create. 
The process described above does not guarantee that, in any given realization, the number of nodes is exactly what we desire.
Hence, we repeat this construction process until the number of nodes is in a specific interval.
Once the tree is created, we compute the probability of each scenario by assuming that, whenever a branching takes place in the tree, each of the children nodes is equally likely to occur.

\subsection{Uncapacitated lot-sizing problem} \label{subsec: lot-sizing formulation}

The data of this problem consists of a scenario tree $\tree$ with node set $\nodes$. 
For each $v \in \nodes$, we are given  
the demand $d(v)$ experienced at this node, the fixed 
setup cost $f(v)$ incurred for producing at this node, as well as the per-unit production cost $g(v)$ 
and the per-unit inventory-holding cost $h(v)$.

At node $v \in \nodes$, we use binary variable $\x{v}$ to indicate whether production occurs at this node, continuous variable $q(v)$ to denote the quantity produced at $v$, and continuous variable $s(v)$ to describe the inventory level for the product at the end of the period corresponding to node $v$.

Using this notation, the problem can be formulated as

\begin{equation} \label{eq: lot-sizing}
\begin{aligned}
\min \quad  
& \mathrlap{\sum_{\scenario \in \scenarioset}  p_{\scenario} \cdot \sum_{v \in \scenario}\big(f(v) \x{v} + g(v) q(v) + h(v) s(v)\big)}
&& \\
\text{s.t.} \quad
& s(v) = s(\parent{v}) + q(v) - d(v), \quad
&& \forall v \in \nodes \backslash \{\root\}, \\
& q(v) \leqslant M_v \cdot \x{v}, \quad
&& \forall v \in \nodes, \\
& s(\root) = 0, \quad && \\
& \x{v} \in \Bin, \; q(v), s(v) \geqslant 0, \quad
&& \forall v \in \nodes,
\end{aligned}
\tag{LS}
\end{equation}
where $M_v$ are sufficiently large constants. 
In particular, a valid choice of $M_v$ is the maximum remaining cumulative demand from node $v$ to the end of the decision horizon.
We consider the variables $\x{v}$ to be strategic.

\paragraph{Instance generation.} 
The data corresponding to each node $v$ is generated as follows:
the demand $d(v)$ is uniformly sampled from integers between $1$ and $10$ and then multiplied by $100$.
The fixed cost $f(v)$ is uniformly sampled from integers between $1$ and $20$ and then multiplied by $1000$.
The production cost $g(v)$ is uniformly sampled from integers between $1$ and $5$ and then multiplied by $40$.
Lastly, the holding cost $h(v)$ is uniformly sampled from integers between $1$ and $20$.

\subsection{Capacity planning problem} \label{subsec: capacity planning formulation}
The problem we consider is a variant of the tool planning problem studied in \cite{huang2005}.
This problem involves $N$ tools, $O$ operations, and $P$ products.
The manufacturer employs related tools and operations to make products.
The data of this problem consists of a scenario tree $\tree$ with nodes $\nodes$.
In addition, we are given an upper bound $U$ on the total quantity of tools produced or purchased and the processing time $V$ that a single tool type can provide in each period. 
Further, for each node $v$, we are given 
the fixed cost $f(v)_i$ of producing or purchasing tool type $i$, the per-unit cost $g(v)_i$ of tool type $i$, 
the per-unit holding cost $h(v)_i$ of tool type $i$,
the penalty cost $l(v)_p$ for being short of product type $p$, 
and the demand $d(v)_p$ of product $p$. We are also given the time $t_{ijp}$ required for using tool $i$ and operation $j$ on product $p$, which is the same across all nodes. 

The model has binary variables $\x{v}{i}$  to indicate if tool type $i$ is produced or purchased at node $v$  or not. 
In addition, for each node $v$, the model has continuous variables to describe the quantity $q(v)_i$ of tool type $i$ produced or purchased,
the inventory $s(v)_i$ of tool type $i$ at the end of the period, 
the quantity $w(v)_p$ of product $p$, 
the shortage $u(v)_p$ of product $p$,
and the quantity of product $p$ that pass through operation $j$ on tool $i$
at node $v$.

Using this notation, the problem can be formulated as the following 0-1 mixed integer program where the objective 
is to minimize the total cost of production and holding plus the shortage penalty:

\begin{equation} \label{eq:semiconductor tool planning}
\begin{aligned}
\min \quad 
& \mathrlap{ 
  \sum_{\scenario \in \scenarioset} p_{\scenario}
  \sum_{v \in \scenario}
  \Biggl(
    \sum_{i=1}^{N} 
      \bigl(
          f(v)_i \x{v}{i}
          + g(v)_i q(v)_i
          + h(v)_i s(v)_i
      \bigr) 
    \;+\;
    \sum_{p=1}^{P}
      l(v)_p \, u(v)_p
  \Biggr) } \\
\text{s.t.} \quad
& s(v)_i = s(\parent{v})_i + q(v)_i,  \quad
 && \forall\,v \in \nodes,\;  \forall\,i \in [N], \\
& q(v)_i \le U \, \x{v}{i}, \quad
 && \forall\,v \in \nodes,\; \forall\,i \in [N], \\
& \sum_{i=1}^{N} q(v)_i \le U, \quad
 && \forall\,v \in \nodes, \\
& w(v)_p + u(v)_p \ge d(v)_p, \quad
 && \forall\,v \in \nodes,\; \forall\,p \in [P], \\
& w(v)_p \le \sum_{i=1}^{N} o(v)_{ijp}, \quad
 && \forall\,v \in \nodes,\; \forall\,j \in [O],\; \forall\,p \in [P], \\
& \sum_{j=1}^{O} \sum_{p=1}^{P}
  t_{ijp}\,o(v)_{ijp}
  \;\;\le\;\; s(v)_i\, V, \quad
 && \forall\,v \in \nodes,\; \forall\,i \in [N], \\
& s(\root)_i = 0, && \forall\, i \in [N], \\
& \x{v}{i} \in \Bin, \; q(v)_i, s(v)_i \geqslant 0, \quad
 && \forall\,v \in \nodes,\; \forall\,i \in [N], \\
& w(v)_p, u(v)_p \geqslant 0, \quad
 && \forall\,v \in \nodes,\; \forall\,p \in [P],
\end{aligned}
\tag{TP}
\end{equation}
We consider the variables $\x{v}{i}$ to be strategic.
As we discussed before, there are several ways to impose a $\nrevisions$-revision constraint on these variables.
In our numerical experiment, we assume that a revision is triggered whenever any of the variables $\x{v}{i}$ at node $v$ deviates from the plan at the previous node. 
Hence, the revision budget can be seen as shared across all $i \in [N]$ and not imposed individually on each $i \in [N]$. 

\paragraph{Instance generation.}
We set $N=10$, $O=50$, $P=10$, $U=10^6$, and $V=3000$.
We generate the demand data as in \cite{huang2005}, where $d(v)_p$ is sampled from $\operatorname{LogNormal}(1+0.5t,0.5+0.1t) \cdot d_0$, where $t$ is the stage of the node and $d_0=1000$ is the initial demand.
For the other parameters, we generate $f(v)_i$, $g(v)_i$, and $h(v)_i$ by drawing integers uniformly between $1$ to $5$, and then scaling them by factors $60$, $6$, and $1$ respectively for each $i$.
We set $l(v)_p=10$ for each $p$.
For $t_{ijp}$, we first consider a product $p$ and generate its required operations.
We choose a random number $\eta$ in the interval $[\lfloor O/2 \rfloor, O]$ and uniformly select a subset of $[O]$ of size $\eta$ as required operations for $p$.
Next, for each required operation j, we draw a number uniformly from [N] and then draw that many tool indices uniformly from [N].
Finally, for each required operation $j$ and required tool $i$, we set $t_{ijp}$ uniformly random between 1 and 10.
If $j$ is required but $i$ is unavailable, we set $t_{ijp}=10^8$, and if $j$ is not required for $p$, we set $t_{ijp}=0$.

\subsection{Single airport ground-holding program} \label{subsec: SAGHP formulation}

We consider the single airport ground-holding program (SAGHP) model described in \cite{mukherjee2007, estes2020} where, 
given a set of flights $\Theta$ destined to a specific airport, the goal is to determine the controlled departure time for each incoming flight in order to minimize a combination of ground delays and expected air delays.
Aside from the scenario tree $\tree$ with nodes $\nodes$, we are given four categories of parameters. 
First, we are given the cost $c_g$ of holding a flight on the ground per time period and the cost $c_a$ of holding a flight in the air per time period
Second, for every flight $f$, we are given its earliest departure time $\mu_f$, its flight duration $\delta_f$, and the set of its possible departure times $\Gamma_f$.
Third, for every node $v$ in the scenario tree, we are given the maximum number $C_g(v)$ of flight landings that the airport can accommodate during the time period associated with $v$. 
Finally, for each stage $t$, we are given the number $E_{t}$ of exempted flights whose schedules cannot be altered and are required to land at the airport at stage $t$, and the set of flights $\Phi_t^{\text{arr}}$ whose earliest possible arrival is at stage $t$.
 
At each stage $t$, the decision maker chooses a subset of undeparted flights for departure while retaining all others on the ground until the next stage.
If flight $f$ is permitted to depart at stage $t$, the flight reaches the terminal airspace in stage $t+\delta_f$. 

For a given scenario node $v$ in some stage $\tau(v)$, the capacity value $C_g(v)$ determines the number of flights present in the terminal airspace that can land in stage $\tau(v)$. 
If the number of flights present exceeds the capacity, the excess flights must be held in the air until the next stage. 

The model has binary decisions $\x{v}{f}$ that indicate whether flight $f$ departs at node $v$ or not.
In addition, the model has integer variables to describe the number $w(v)$ of flights held in the air at node $v$, the number $l(v)$ of flights that will land at the airport at node $v$.
We assume that the variables $\x{v}{f}$ are strategic decisions on which we impose a $\nrevisions$-revision constraint. 
We consider that a revision occurs whenever the planned takeoff time of any flight changes. 
Hence, the revision budget is shared across flights and not imposed on each flight individually. 

Using this notation, (SAGHP) can be formulated as the following 0-1 mixed integer program where the objective is to minimize the total cost of holding flights in the air and on the ground.
Here, $x(\scenario(\stage{v}-\delta_f))_f$ indicates whether flight $f$ departs at time $\stage{v}-\delta_f$, or equivalently, whether $f$ arrives in the stage of node $v$.

\begin{equation} \label{eq:saghp}
\begin{aligned}
\min \quad 
& \mathrlap{\sum_{\scenario \in \scenarioset} p_{\scenario}
  \Biggl(
  \sum_{f\in \Theta}
    \sum_{\substack{v \in \scenario: \\ \stage{v} \in [\mu_f..\nstages]}}
      c_g\cdot(\stage{v}-\mu_f)\,\x{v}{f}
    +
    \sum_{v \in \scenario}
      c_a\, w(v)
  \Biggr) }
  \\
\text{s.t.} \quad
& \sum_{\substack{v \in \scenario: \\ \stage{v} \in \Gamma_f}} \x{v}{f} = 1, \quad 
 && \forall\,f\in \Theta, \scenario \in \scenarioset, \\
& E_{\stage{v}} + w(\parent{v}) + \sum_{f \in \Phi_{\stage{v}}^{\text{arr}}} \x{\scenario(\stage{v}-\delta_f)}{f} \\
&\hspace{+75pt} = l(v) + w(v), 
\quad 
 && \forall\, \scenario\in \scenarioset, v\in \scenario,\\
& l(v) \leqslant C_{g}(v), \quad 
 && \forall\,v \in \nodes, \\
& w(\root) = l(\root) = 0, && \\
& \x{v}{f} \in \Bin, \quad
 && \forall\, v\in \nodes, f\in \Phi_{\stage{v}}^{\text{arr}}, \\
& w(v), l(v) \geqslant 0, \quad
 && \forall\, v\in \nodes.
\end{aligned}
\tag{SAGHP}
\end{equation}

\paragraph{Instance generation.}
We use data from three airports: SFO, EWR, and ORD on July 1st, 2024 from the database \cite{aspmdata}.
To generate problem instances in \Cref{fig: saghp} and \Cref{tab: statistics saghp}, we select seven time horizons: 7:00-12:00, 7:00-13:00, 7:00-14:00, 7:00-15:00, 17:00-22:00, 17:00-23:00, and 17:00-24:00, and we consider meteorological condition code sequence VIV, IMV, VIMV.
This yields a total of $63$ combinations, the largest of which (ORD, 7:00-15:00, VIMV) we exclude since it does not fit in memory.
For instances in \Cref{fig:saghp 6obj compare}, we select the busiest 4-hour time horizon for each airport on that date and consider meteorological condition code sequence VIMV.

We next discuss how we generate a scenario tree given an airport, a meteorological condition code sequence, and a time horizon.
We first divide the time horizon into 15-minute intervals, where each interval corresponds to a stage in the scenario tree.
Thus, the number of intervals determines the number of stages $\nstages$.
The scheduled departure of each flight is converted into the corresponding stage index, and the flight time is rounded to the nearest number of stages.
From this processed dataset, we extract the parameters $\Theta$, $\mu_f$, $\delta_f$, $E_t$, $\Phi_t^{\text{arr}}$, and $\Gamma_f$ for each problem, \emph{i.e.}, a combination of an airport and a time horizon.

For each meteorological condition code sequence and the probability parameters $q_+$ and $q_-$ defined in \Cref{subsec: saghp},  we construct the scenario tree as follows.
The root node is assigned the first meteorological condition code in the sequence.
For each node $v$ at an odd-numbered stage, we create two child nodes: one retaining the same code with probability $1-q(v)$, and one transitioning to the next code in the sequence with probability $q(v)$. 
Here, we set $q(v) = q_+$ if the next code in the sequence corresponds to a higher capacity than the current code, and we set $q(v) = q_-$ otherwise.
For each node at an even-numbered stage, we attach a single child, retaining the same meteorological condition code as the parent.
If a node's weather type is the final code in the sequence, we create a single child with the same code.
The construction does not proceed beyond nodes at stage $\nstages$.
Hence, each condition code is assumed to persist for at least 30 minutes (two stages) and transitions are allowed only at nodes located at odd-numbered stages.

Hence, for each node $v$ in the scenario tree, there is a corresponding meteorological condition code. 
We then determine the arrival capacity $C_g(v)$ based on the assigned meteorological condition code, using 
the arrival capacity described in the FAA Airport Capacity Profile \cite{capacitydata}. 
In \Cref{fig:saghp 6obj compare}, we also varies $C_g(v)$ by a factor $\beta$ for $v$ whose meteorological condition code is M or I to test the influence of ground capacity.
There are multiple runway configurations for SFO and ORD. 
We use the data of 28/01 Side-Byes for SFO under visual and marginal conditions and 28/01 Intrail under instrument condition.
For ORD, we use the data of West Flow for all three conditions.
Finally, we set the penalty parameters $c_g=1$ and $c_a=3$.

\section{Revision penalty models} 
\label{sec: penalty models}

A natural variation of the $\nrevisions$-revision model is to remove the hard revision budget but add revision-related expressions as penalty terms in the objective.
We introduce three different penalty models by specifying their MIP formulations directly, where the variables can be interpreted similarly to \Cref{eq: cp}.
For simplicity, we only consider binary strategic decisions and no operational decisions.

\paragraph{Plan edits penalty.}
Consider a non-root node $v$, we define the number of plan edits at $v$ as $|\{t\in [\stage{v}:\nstages]: \pi(v,t) \neq \pi(\parent{v},t)\}|$.
We describe the following model that penalizes the total number of plan edits across all nodes by introducing the variable $e(v,t)$ to linearize the term $|\pi(v, t) - \pi(\parent{v}, t)|$:

\begin{equation} \label{eq: penalty - plan edits}
\begin{aligned}
	\max \quad  &  \mathrlap{\textstyle \sum_{v \in \nodes} \cc{v} \x{v} - \lambda \cdot \Big( \sum_{v\in \nodes, t \in [\stage{v}:\nstages]} e(v,t) \Big)} \\
	\st  \quad & \plan{v,\stage{v}} = \x{v}, \quad && \forall v \in \nodes,\\
    & e(v,t) \geqslant \plan{v,t}-\plan{\parent{v},t} \quad && \forall v \in \nodes, t \in [\stage{v}:\nstages]\\
    & e(v,t) \geqslant \plan{\parent{v},t}-\plan{v,t} \quad && \forall v \in \nodes, t \in [\stage{v}:\nstages]\\
     & \x{v}, \plan{v,t} \in \Bin, \quad &&\forall v \in \nodes, t \in [\stage{v}:\nstages].
\end{aligned}
\end{equation}

\paragraph{Expected revisions penalty.} We use the exact same set of variables as in \Cref{eq: cp} and let $p_v:=\sum_{\omega: v \in \omega} p_\omega$ be the probability of a node $v$.
Then, penalizing the expected number of revisions, we obtain
\begin{equation} \label{eq: penalty - expected num revisions}
\begin{aligned}
	\max \quad  &  \mathrlap{\textstyle \sum_{v \in \nodes} \cc{v} \x{v} - \lambda \cdot \Big( \sum_{v\in \nodes} p_v r(v)  \Big)} \\
	\st  \quad & \plan{v,\stage{v}} = \x{v}, \quad && \forall v \in \nodes,\\
    & r(v) \geqslant \plan{v,t}-\plan{\parent{v},t} \quad && \forall v \in \nodes, t \in [\stage{v}:\nstages]\\
    & r(v) \geqslant \plan{\parent{v},t}-\plan{v,t} \quad && \forall v \in \nodes, t \in [\stage{v}:\nstages] \\
     & \x{v}, \plan{v,t} \in \Bin, \quad &&\forall v \in \nodes, t \in [\stage{v}:\nstages].
\end{aligned}
\end{equation}

\paragraph{Revision number penalty.}
Lastly, we penalize the actual minimum revisability number, which is equivalent to the Lagrangian relaxation of the $\nrevisions$-revision constraints.
We use $\kappa$ to denote the actual revision number. We obtain
\begin{equation} \label{eq: penalty - revision number}
\begin{aligned}
	\max \quad  &  \textstyle \sum_{v \in \nodes} \cc{v} \x{v} - \lambda \cdot \kappa \\
	\st  \quad & \plan{v,\stage{v}} = \x{v}, \quad && \forall v \in \nodes,\\
    & r(v) \geqslant \plan{v,t}-\plan{\parent{v},t} \quad && \forall v \in \nodes, t \in [\stage{v}:\nstages]\\
    & r(v) \geqslant \plan{\parent{v},t}-\plan{v,t} \quad && \forall v \in \nodes, t \in [\stage{v}:\nstages] \\
    & \textstyle \kappa \geqslant \sum_{v \in \omega} r(v), \quad && \forall \omega \in \Omega, \\
     & \x{v}, \plan{v,t} \in \Bin, \quad &&\forall v \in \nodes, t \in [\stage{v}:\nstages].
\end{aligned}
\end{equation}

\begin{figure}[tbp]
\centering
\includegraphics[width=\linewidth]{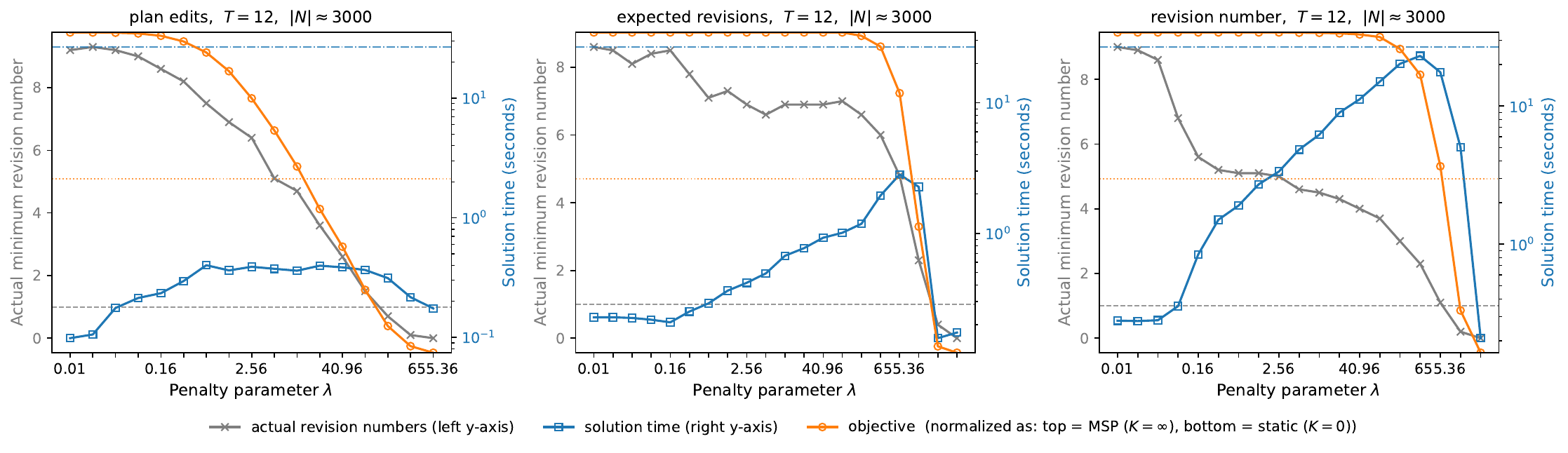}
\vspace{-1em}
\caption{Comparison of penalty models with the $1$-revision model on the hypercube problem.}
\label{fig:hypercube penalty}
\begin{minipage}{\textwidth}
\footnotesize
Note. From left to right, we compare the $1$-revision model with: (1) the plan edits penalty, (2) the expected revision penalty, and (3) the revision number penalty.
The penalty parameter $\lambda$ varies along the $x$-axis and doubles at each subsequent point.
The left and right y-axis correspond to the actual revision number (\emph{i.e.} minimum revisability number) and the solution time of the optimal penalty model solution, respectively.
Their data is represented using gray-cross and blue-square line, respectively.
The orange-circle line represents the original objective value without the penalty term, where the top and the bottom of each plot is scaled as the MSP ($\infty$-revision) value and the static ($0$-revision) value, respectively.
The lower gray dash line, the upper blue dash-dot line, and the middle orange dot line are the revision number $1$, the solution time, and the objective of the  $1$-revision model respectively.
Each data point corresponds to an average over ten instances.
\end{minipage}
\end{figure}

We use these models to create solutions for hypercube problems with $\nstages=12$ and $|\nodes|=3000$ for several values  of the penalty parameter $\lambda$.
For each model, we report the solution time and the optimal objective value, and we compute the actual revision number of the optimal solution, as the penalty parameter $\lambda$ varies.
The actual revision number (gray-cross line) stays above $1$ (the bottom gray dashed line) until $\lambda$ becomes sufficiently large. 
However, that comes at the cost of the original objective: the objective (orange-circle line) falls below that of the $1$-revision model (the middle dotted line).
The only exception occurs at $\lambda=655.36$ for the revision number penalty model (the rightmost plot), where both the objective and the revision number are close to those of the $1$-revision.
In that case, however, the solution time is also close to that of the $1$-revision; see the blue-square line that nearly touches the top blue dash-dot line.
Hence, although these models usually solve faster, they guarantee neither the desired revision number nor a good objective value. 
This guarantee matters in practice: a single plan revision can trigger substantial operational costs, such as rescheduling workers or inventory shipments in lot-sizing, or rearranging airline operations in ground delay programs. 
Such costs are difficult to price in an objective penalty, which may require tuning $\lambda$ and solving multiple models, whereas decision makers can specify how many revisions they can tolerate.

\section{MIP formulations for multi-dimensional strategic decisions}
\label{sec: multi-d formulations}
In \Cref{sec: formulations} and \Cref{sec: methods}, we introduced formulations under \Cref{assumption: 1-dimension}.
In this section, we extend them to multi-dimensional strategic decisions and show that the polyhedral results carry over.
Recall that $\dimstrategic{t}$ denotes the dimension of the strategic decisions at stage $t$. 
For CP-type formulations, we define variables $x(v)$ and $\pi(v,t)$ with $t \in [\stage{v}, \nstages]$ as $\textstyle\dimstrategic{\stage{v}}$- and $\dimstrategic{t}$-dimensional vectors, respectively.
\begin{equation} \label{eq: multi-d cp}
\begin{aligned}
    &\x{v}{i} = \plan{v, \stage{v}}_i, \quad && \forall v \in \nodes, \forall i \in [\dimstrategic{\stage{v}}]  \\
    &\rv{v} \geqslant \plan{v,t}_i - \plan{\parent{v},t}_i, \quad && \forall v \in \nodes \setdiff \set{\root}, \forall t \in [\stage{v}:\nstages], \forall i \in [\dimstrategic{t}]  \\
    &\rv{v} \geqslant \plan{\parent{v},t}_i - \plan{v,t}_i, \quad && \forall v \in \nodes \setdiff \set{\root}, \forall t \in [\stage{v}:\nstages], \forall i \in [\dimstrategic{t}]\\
    &\textstyle \sum_{v \in \scenario} \rv{v} \leqslant \nrevisions, \quad && \forall \scenario \in \scenarioset, \\
    &\x{v} \in \Bin^{\dimstrategic{\stage{v}}},\\ 
    &\plan{v,t} \in \Bin^{\dimstrategic{t}}, \rv{v} \in \Bin, \quad && \forall v \in \nodes, \forall t \in [\stage{v}:\nstages].  \\
\end{aligned}
\tag{CP-d}
\end{equation}
For ST-type formulations, we introduce variables $\delta(u,v)$ to model $\indicator{x(u) \neq x(v)}$. 
\begin{equation} \label{eq: multi-d st}
\begin{aligned}
    &\delta(u,v) \geqslant x(u)_i-x(v)_i, \quad && \forall u, v \in \nodes_t, \forall t \in [\nstages], \forall i \in [\dimstrategic{t}] \\
    &\delta(u,v) \geqslant x(v)_i-x(u)_i, \quad && \forall u, v \in \nodes_t, \forall t \in [\nstages], \forall i \in [\dimstrategic{t}] \\
    &\textstyle \sum_{(u, v) \in \sib{\calS}} \delta(u,v) \leqslant 2^{\nrevisions+1} - 2, \quad && \forall \calS \in \subtreefamily{\nrevisions+1}{\tree} \\
    &\x{v} \in \Bin^{\dimstrategic{\stage{v}}}, \delta(u,v) \in \Bin \quad && \forall u, v \in \nodes: \stage{u}=\stage{v}
\end{aligned}
\tag{ST-d}
\end{equation}
The other formulations can be adapted similarly. 
We next establish that the two multi-dimensional formulations are equivalent to applying \Cref{eq: cp}/\Cref{eq: st} to the expanded tree of \Cref{fig: transform multi-d to 1-d} in the sense of \emph{formulation isomorphism}.

\begin{definition}[Formulation isomorphism]
	Let $F_1$, $F_2$ be formulations in binary variables with index sets $V_1$, $V_2$, and constraints index sets $R_1$, $R_2$ with constraints $a_r^\top z \leqslant \beta_r$, $r \in R_j$, $j=1,2$.
	We say $F_1$ and $F_2$ are isomorphic if there is a bijection $\phi: V_1 \rightarrow V_2$ and a bijection $\psi: R_1 \rightarrow R_2$ such that the two constraints $a^\top_r (z_i)_i \leqslant \beta_r$ and $a^\top_{\psi(r)} (z_{\phi(i)})_i \leqslant \beta_\psi(r)$ are the same up to a scalar factor $\lambda_r > 0$, where $r \in R_1$. 
	In other words, the two formulations are exactly the same, up to scaling, under column and row relabeling.
\end{definition}

Given an original tree $\tree$, we obtain 
the \emph{expanded tree} $\tilde{\tree}$ by replacing each node $v \in \mathcal{N}$ at stage $\tau(v) = t$ by a chain of $d^1_t$ nodes $v^{1}, \dots, v^{d^1_t}$, one per strategic component. 
Formally, define $\tilde{\nodes}$, the node set of $\tilde{\tree}$, as $\tilde{\nodes}=\{v^{i} \mid v \in \nodes,\ i \in [d^1_{\stage{v}}]\}$.
The arc set of $\tilde{\tree}$ is composed of arcs $(v^{i-1}, v^i)$ for all $2 \leqslant i \leqslant d_{\stage{v}}$ and $(\parent{v}^{d^1_{\stage{\parent{v}}}}, v^{1})$ for all $v \in \nodes$.
We also use $\tilde{\nstages}$ and $\tilde{\tau}(\cdot)$ to represent the number of stages and the stage of a node for $\tilde{\tree}$.
\begin{proposition}
\label[proposition]{prop: formulation isomorphism}
Consider a tree $\tree$ and its expanded tree $\tilde{\tree}$.
\begin{enumerate}[label=(\Roman*)]
	\item If $\tree$ does not have only child nodes, formulation \Cref{eq: multi-d cp} for $\tree$ and \Cref{eq: cp+} on $\tilde{\tree}$ are isomorphic.
	\item The projection of formulation \Cref{eq: multi-d st} for $\tree$ on the $x$-space and formulation \Cref{eq: st} on $\tilde{\tree}$ are isomorphic.
\end{enumerate}	
\end{proposition}
\begin{proof}
    We first prove (I).
    For $\tree$ with no only child nodes, \Cref{eq: multi-d cp} is precisely shown above.
    The corresponding \Cref{eq: cp+} formulation for the expanded tree $\tilde{\tree}$ is 
\begin{equation} \label{eq: expanded cp+}
\begin{aligned}
    &\tilde{x}\big(v^i\big) = \tilde{\pi}\big(v^1, \tilde{\tau}({v^i})\big), \quad && \forall v \in \nodes, \forall i \in [1:\dimstrategic{\stage{v}}], \\
    &\tilde{r}\big(v^1\big) \geqslant \tilde{\pi}\big(v^1,\tilde{t}\big) - \tilde{\pi}\big(\parent{v}^1, \tilde{t}\big), \quad && \forall v \in \nodes \setdiff \set{\root}, \forall \tilde{t} \in [\tilde{\tau}(v^1):\tilde{\nstages}],  \\
    &\tilde{r}\big(v^1\big) \geqslant \tilde{\pi}\big(\parent{v}^1, \tilde{t}\big) - \tilde{\pi}\big(v^1,\tilde{t}\big), \quad && \forall v \in \nodes \setdiff \set{\root}, \forall \tilde{t} \in [\tilde{\tau}(v^1):\tilde{\nstages}],  \\
    &\textstyle \sum_{v \in \scenario} \tilde{r}\big(v^1\big) \leqslant \nrevisions, \quad && \forall \scenario \in \scenarioset, \\
    &\tilde{x}\big(v^i\big) \in \Bin, \quad && \forall v^i \in \tilde{\nodes} \\  
    &\tilde{\pi}\big(v^1, \tilde{t}\big), \; \tilde{r}\big(v^1\big) \in \Bin, \quad && \forall v \in \nodes, \forall \tilde{t} \in [\tilde{\tau}(v^1):\tilde{\nstages}].  \\
\end{aligned}
\end{equation}
Specifically, the first two rows in \Cref{eq: cp+} reduces to the first row in \Cref{eq: expanded cp+}, as the only child nodes in $\tilde{\tree}$ are $\{v^i\}_{v \in \nodes \setdiff \{\root\}, i \geqslant 2}$ and the nearest non-only-child ancestor of $v^i$ ($i \geqslant 2$) is $v^1$.
Further, since the plan and revision variables are only needed when a node is not an only child, we thus only have $\tilde{\pi}$ and $\tilde{r}$ variables for $v^1$.
It is then easy to verify that other inequalities in \Cref{eq: expanded cp+} correspond to \Cref{eq: cp+} precisely.
Then, we define the isomorphism as follows. 
We define $\phi$ such that it maps $x(v)_i$ to $\tilde{x}\big(v^i\big)$, $r(v)$ to $\tilde{r}\big(v^1\big)$, and $\pi(v,t)_i$ to $\tilde{\pi}\big(v^1,\tilde{\tau}(u^1)+i-1\big)$, where $u$ is selected as a descendant of $v$ at stage $t$.
The row correspondence is thus direct.
We only note that the second (or third) row in \Cref{eq: multi-d cp} maps to the second (or third) row in \Cref{eq: expanded cp+} by mapping $(t, i)$ to $\tilde{\tau}(u^1) + i - 1$ for some $u$ being a descendant of $v$ at stage $t$.

For (II), we can see that the projection of \Cref{eq: multi-d st} for $\tree$ on the $x$-space is precisely the subtree inequalities:
\begin{equation}
\label{eq: expanded st}
\begin{aligned}
    \textstyle \sum_{(u,v) \in \sib{\vec{\calS}}} & \; \left( x(u)_{i(u,v)} - x(v)_{i(u,v)} \right) \leqslant 2^{\nrevisions+1}-2, \\
    &\forall \calS \in \subtreefamily{\nrevisions+1}{\tree}, \forall \vec{\calS} \in  \operatorname{ori}(\calS), \forall i_{(u,v)} \in [1:\dimstrategic{\stage{v}}], (u,v) \in \sib{\calS}.
\end{aligned}
\end{equation}
We similarly define $\phi$ that maps $x(v)_i$ to $\tilde{x}\big(v^i\big)$.
Then, define $\psi$ that maps subtree constraints \Cref{eq: expanded st} to the one with $\tilde{\calS}$ defined by sibling pairs $\{(u^{i(u,v)}, v^{i(u,v)})\}$ in $\tilde{\tree}$.
This shows the isomorphism between the projected \Cref{eq: multi-d st} for $\tree$ and \Cref{eq: st} for $\tilde{\tree}$.
\end{proof}

In the first statement of \Cref{prop: formulation isomorphism}, we assume $\tree$ is only-child-node-free. 
If $\tree$ has only child nodes, then we can use a size reduced version of \Cref{eq: multi-d cp} using the techniques in \Cref{subsubsec: cp size reduction}. 
This assumption is therefore without loss of generality.

By \Cref{prop: formulation isomorphism}, we can study our formulations on the expanded tree under \Cref{assumption: 1-dimension}, and all the polyhedral results carry to the multi-dimensional case.
To derive a result for a multi-dimensional formulation, we simply check the corresponding result for the one-dimensional formulation on the expanded tree.
For instance, we can conclude that \Cref{eq: multi-d st} is ideal when $\nstages \leqslant \nrevisions + 2$ and $\dimstrategic{t}=1$ for all $t \in [T]$ by \Cref{thm: subtree}.
However, it is not ideal in general when $\nstages=3$, $\nrevisions=1$, and $\dimstrategic{t}>1$ for some $t$ by \Cref{prop: st not ideal in general}.

\section{Omitted proofs}  \label{sec: omitted proofs}

\ThmNPhardGeneral*

\begin{proof}
\label{proof: nphard general}
First, we argue that for any positive integer $\nstages > 3$, the problem $\hypercube{1}{\tree, \cc}$ where $\tree$ has three stages can be reduced to a problem $\hypercube{1}{\tree', \cc'}$ where $\tree'$ has $\nstages$ stages.
This reduction can be done by simply adding a path of length $\nstages-3$ of ``blank'' nodes with objective coefficient $0$ to each leaf of the original tree $\tree$.

Second, we show that for any positive integer $\nrevisions \leqslant \nstages-2$, the problem $\hypercube{\nrevisions}{\tree', \cc'}$ with $\tree'$ having $\nstages$ stages can be reduced to a problem $\hypercube{\nrevisions+1}{\tree'', \cc''}$, where $\tree''$ has $\nstages+1$ stages. 
We can construct $\tree''$ by expanding $\tree'$ with one more stage.
Specifically, for each leaf $\ell$ in $\tree'$, we create two children for $\ell$ in $\tree''$.
For each of the two children, we add one decision variable with objective coefficients $M$ and $-M$ respectively, where $M = 1 + \sum_{v \in \nodes(\tree')} \sum_{i \in [\dimstrategic{\stage{v}}]} | c'(v)_i |$, which is polynomial in the size of the original data.
Since $\hypercube{\nrevisions+1}{\tree', \cc'} < M$, then in an optimal plan adjustment policy for $\hypercube{\nrevisions+1}{\tree'', \cc''}$, a revision must occur for at least one of the children of $\ell$. 
Otherwise, we incur a cost of $M$ in our objective, but gain no more than $\hypercube{\nrevisions+1}{\tree', \cc'} - \hypercube{\nrevisions}{\tree', \cc'}$, which is upper bounded by $\hypercube{\nrevisions+1}{\tree', \cc'} < M$.
Therefore, solving $\hypercube{\nrevisions}{\tree', \cc'}$ reduces to solving $\hypercube{\nrevisions+1}{\tree'', \cc''}$.

The above results show that there is a polynomial reduction from $\hypercube{1}{\tree, c}$, where $\tree$ has $3$ stages to $\hypercube{\nrevisions}{\tree', \cc'}$ where $\tree'$ has $\nstages$ for any value of $\nrevisions$ and $\nstages$ such that $1 \leqslant \nrevisions \leqslant \nstages-2$.
\end{proof}

\PropLargeLossValue*
\begin{proof}
\label{proof: large loss and value}
For the first statement, consider a stochastic uncapacitated lot-sizing problem (\Cref{subsec: lot-sizing formulation}) where all scenarios have equal probabilities; see the left panel of \Cref{fig:loss and value of 1-revision}.
We assume $g(v)=0$ for all $v \in \nodes$.
For the root node $\root$, we set $d(\root)=f(\root)=0$.
We let $D$, $F$, and $S$ be positive integers.
For every node $v$ marked with a solid square, we set $h(v)=S$ and $d(v)=f(v)=0$.
For every node $v$ marked with a hollow square, we set $h(v)=d(v)=f(v)=0$.
For every node $v$ marked with a solid diamond, we set  $d(v)=D$ and $f(v)=h(v)=0$.
Lastly, for every node $v$ marked with a hollow diamond, we set  $d(v)=1$, $f(v)=F$, and $h(v)=0$.
On the one hand, one can verify that $z_{\operatorname{MS}}=1.5S$, where the optimal policy is to produce a quantity of $2$ at the root node, then to produce the amount needed to fulfill the demand at every solid diamond node, and not to produce at any of the other nodes.
On the other hand, $z_1 \geqslant \min\set{S\cdot D, F}$, because for each scenario, a $ 1$-revisable policy either chooses to produce at a hollow diamond node, or not produce at a solid diamond node, meaning that there is a holding cost at a solid square node.
Therefore, the loss of $1$-revision relative to $z_{\text{MS}}$, $\sfrac{(z_1-z_{\text{MS}})}{|z_{\text{MS}}|}$, can be made arbitrarily large by selecting $D$ large and $F$ large relative to $S$.

For the second statement, consider the scenario tree shown in the right panel of \Cref{fig:loss and value of 1-revision}, where all scenarios have equal probabilities.
We assume that $d(\root)=f(\root)=g(\root)=0$, $h(\root)=S$ and also that $g(v)=0$, $h(v)=S$  for all $v \in \nodes\setdiff\set{\root}$.
We set $d(v)=D_1$, $f(v)=F_1$ for nodes $v$ marked with a hollow circle and $d(v)=D_2$, $f(v)=F_2$ for nodes $v$ marked with a solid circle. 
We choose $S$, $D_1$, $D_2$, $F_1$, and $F_2$ to be positive integers such that $D_1 < D_2$ and $D_1 S < F_2 < (D_2-D_1) S < F_1$.
Specifically, we let $D_1 = 1$, $D_2 \geqslant 3$, $F_1 = D_2 S$, and $F_2 = 2S$.
Then, the optimal strategy for the \MSP{} model would be to produce at the solid nodes whereas not to produce and to use stored products at the hollow nodes.
Then, to compute $z_1$, the plan adjustment policy corresponding to an optimal $1$-revisable policy is  to produce at each stage and to revise at each hollow node.
However, every partially adaptive model will force two nodes marked with different shades at the same stage to take the same strategic decision.
One can verify that $z_1 = 2.75 D_1 S + 0.75 F_2 = 4.25 S$ and $z^{\text{PA}}_1 \geqslant \min \{ 0.25 D_2 S, 0.25 F_1 \} = 0.25 D_2 S$.
Therefore, when $D_2$ gets larger, the value of $1$-revision relative to $z_1$, $\sfrac{(z_1^{PA}-z_1)}{|z_1|}$, becomes arbitrarily large.
\end{proof}

\begin{figure}[tbp]
\centering
    \begin{minipage}[t]{0.5\textwidth}
    \begin{tikzpicture}[
        grow=right,
        every node/.style={},
        level 1/.style={level distance=14mm, sibling distance=18mm},
        level 2/.style={level distance=11mm, sibling distance=14mm},
        level 3/.style={level distance=14mm, sibling distance=9mm},
        level 4/.style={level distance=11mm}
    ]
        \node [small node] {}
            child {
                node[club node] {} 
                child {
                    node[diamond node] {} 
                    child {
                        node[club node] {} 
                        child {
                            node[diamond node] {} 
                        }
                    }
                    child {
                        node[spade node] {}
                        child {
                            node[heart node] {}
                        }
                    }
                }
            }
            child {
                node[spade node] {}
                child {
                    node[heart node] {}
                     child {
                        node[spade node] {}
                        child {
                            node[heart node] {}
                        }
                    }
                    child {
                        node[club node] {}
                        child {
                            node[diamond node] {}
                        }
                    }
                }
            };
    \end{tikzpicture}
    \end{minipage}
    \begin{minipage}[t]{0.4\textwidth}
    \begin{tikzpicture}[
        grow=right,
    ]
    \tikzstyle{level 1}=[level distance=14mm, sibling distance=14mm]
    \tikzstyle{level 2}=[level distance=14mm, sibling distance=12mm]
    \tikzstyle{level 3}=[level distance=14mm, sibling distance=8mm]

    \node[small node]{}
        child {node[snowy node] {}
            child {node[snowy node] {}
                child {node[snowy node] {} }
                child {node[sunny node] {} }
            }
            child {node[sunny node] {}
                child {node[sunny node] {} }
            }
        }
        child {node[sunny node] {}
            child {node[sunny node] {}
                child {node[sunny node] {} }
            }
        };
    \end{tikzpicture}
    \end{minipage}
\caption{Examples of large loss and large value of $1$-revision for stochastic uncapacitated lot-sizing.}
\begin{minipage}{\textwidth}
    \footnotesize
    Note. For both panels, the nodes with the same shape and shading have the same input data.
\end{minipage}
\label{fig:loss and value of 1-revision}
\end{figure}

\LmInconsistentPair*

\begin{proof}
\label{proof: inconsistent pair}
We prove the result by contradiction.
Assume $(\tree,x)$ is an example where the statement does not hold.
Then, in each stage $t$, there is a value $\theta_t\in\{0,1\}$ such that all the nodes at stage $t$ with $x$-value $\theta_t$ corresponds to a $(\nrevisions-1)$-revisable sub-policy (or there is no nodes with $x$-value $\theta_t$).
Otherwise, we will find a same-stage pair that takes different value and both corresponds to a sub-policy that is not $(\nrevisions-1)$-revisable.

We now show that, in such case, $x$ is $\nrevisions$-revisable.
Define the plan adjustment policy at $\root$ by setting $\plan(\root,1)=x(\root)$ and $\plan(\root,t)=1-\theta_t$ for all $t \in [2:\nstages]$.
Let $U=\{u\in\nodes : x(v)=1-\theta_{\stage{v}}, \forall v \leq u \}$.
The induced subgraph of $U$ is a tree rooted at $\root$ because if $u \in U$ then $\parent{u} \in U$.
Recursively, define the plan adjustment policy at each $u \in U \setdiff \{\root\}$ by setting $\plan(u,t)=\plan(\parent{u},t)$ for all $t \in [\stage{u}:\nstages]$.
Let $B:=\{w \in \nodes \setdiff U : \parent{w} \in U\}$ be the boundary children of $U$.
For any $w \in B$, by definition of $U$ we have $x(w) = \theta_{\stage{w}}$ whereas $\plan(\parent{w},\stage{w})=1-\theta_{\stage{w}}$, thus a revision occurs at $w$.
Also, by our assumption, $x|_{\descendant{w} \cup \{w\}}$ is $(\nrevisions-1)$-revisable in $\tree(w)$, so there exists a plan on $\tree(w)$ requiring at most $(\nrevisions-1)$ revisions per scenario below $w$.
Collect all these plans defined at $\tree(w)$ for all $w\in B$.
Since $(\tree(w))_{w\in B}$ are disjoint, plans created in this manner are uniquely defined.
Assembling the plans defined for nodes in $U$, we derive a plan adjustment policy $\plan$ that takes no more than $\nrevisions$ revisions in each scenario.
This is because for any scenario $\omega$, either $\omega$ contains no node of $B$, or it meets $B$ for the first time at some unique node $w\in B$.
In the former case, $\omega$ stays in $U$ and uses no revision.
In the latter case, $\omega$ incurs exactly one revision at $w$ and then at most $\nrevisions-1$ further revisions.
Therefore, $\plan$ makes $x$ $\nrevisions$-revisable, the desired contradiction.
\end{proof}

\LmSTFacetX*
\begin{proof}
\label{proof: st facet - x}
For any $(u,v) \in \sib{\vec{\calS}}$ and $\theta \in \{0,1\}$, we construct the solution $x$ where $x(u)=x(v)=\theta$, $x(u')=1$ and $x(v')=0$ for all oriented sibling pairs $(u',v')$ in $\vec{\calS}$ that are not $(u,v)$, and where all other entries of $x$ are also $\theta$. 
We refer to this solution as $x^{u,v}_{\theta}$. 
Since exactly $(2^{\nrevisions+1}-2)$ pairs of siblings are set to be inconsistent, \Cref{eq: st constr facet} holds at equality.
Also, there is no $x^{u,v}_\theta$-inconsistent ELBE subtree by our construction.
Hence, $x^{u,v}_\theta \in \revisableset{\nrevisions}$ and thus $x^{u,v}_\theta \in F_{\text{ST}}$, which shows \ref{LmSTFacetX-PartI}.

For \ref{LmSTFacetX-PartII}, it is clear that $x^{u,v}_0+e_w$ also satisfies \Cref{eq: st constr facet} at equality for any $w \notin  \nodes(\calS) \setdiff \{\subroot{\calS}\}$.
We now describe how, given $w \notin \nodes(\calS) \setdiff \{\subroot{\calS}\}$, we can select $(u,v)$ so that $x^{u,v}_0+e_w \in \revisableset{\nrevisions}$.
To this end, we introduce a notation and a claim.
We define $\sigma^\calS_x(v):=\sum_{i \in \nodes(\calS)\cap (\descendant{v} \cup \{v\})} x(i)$, for any ELBE subtree $\calS$, $v \in \nodes(\calS)$, and any integer solution $x$.
That is, $\sigma^\calS_x(v)$ counts the number of nodes in $\calS$ at or beneath $v$ that are assigned $x$-value $1$.
\begin{claim} \label{claim: st facet - left part}
    Let $\bar{x} \in \{0,1\}^\nodes$ and let $\calS'$ be an $\bar{x}$-inconsistent ELBE subtree of height $\nrevisions+1$, then there exists a node $v \in \nodes(\calS')$ such that $\smash{\sigma^{\calS'}_{\bar{x}}}(v)=2^{\nrevisions}$,
    and this happens iff $v$ is the child of the root $\subroot{\calS'}$ in $\calS'$ with $\bar{x}$-value $1$.
\end{claim}
\begin{proof}[Proof of \Cref{claim: st facet - left part}]
    Since $\calS'$ is $\bar{x}$-inconsistent, there is an orientation $\vec{\calS'}$ such that a node has an $\bar{x}$-value $1$ if and only if it is a left node. 
    Let $p$ and $q$ be the left and right child of $\subroot{\calS'}$ in $\vec{\calS'}$, respectively.
    There are exactly $2^\nrevisions$ left nodes in the subtree rooted at $p$ and $2^{\nrevisions}-1$ left nodes in the subtree rooted at $q$ in $\vec{\calS'}$.
    Therefore, $\sigma^{\calS'}_{\bar{x}}(p)=2^\nrevisions$.
    Also, $\sigma^{\calS'}_{\bar{x}}(\subroot{\calS'}) = 2^{\nrevisions+1}-1$ and $\sigma^{\calS'}_{\bar{x}}(u)<2^\nrevisions$ for $u \in \nodes(\calS')\setdiff\{\subroot{\calS'}, p\}$.
    Hence, the unique node in $\calS'$ with $\sigma^{\calS'}_{\bar{x}}(v)=2^{\nrevisions}$ is precisely $v=p$.
\end{proof}
Our approach focuses on the set of left nodes of $\calS$, which we denote by $\calL(\vec{\calS}) := \set{u \mid (u,v) \in \sib{\vec{\calS}}}$.
Then, to choose the pair $(u,v)$, we only need to choose the left node $u$.
We use $(p,q)$ to denote the children of $\subroot{\calS}$ in $\vec{\calS}$, \emph{i.e.}, $p$ is the left node of $\vec{\calS}$ that is at the earliest stage.
We discuss how to select $u$ according to the position of $w$.
Then, we show $\bar{x} := x^{u,v}_0 + e_w \in \revisableset{\nrevisions}$ by showing there is no $(x^{u,v}_0 + e_w)$-inconsistent ELBE subtree with height $\nrevisions+1$.
\begin{enumerate}[label=(\roman*)]
    \item If $\stage{w} \leqslant \stage{p}$, we choose $u$ such that $u \in \calL(\vec{\calS}) \setdiff (\descendant{p} \cup \{p\})$, \emph{i.e.}, $u$ is a left node of $\vec{\calS}$ but not $p$ or a descendant of $p$.
    Since $p$ and its descendants consist of the ``left part'' of $\vec{\calS}$, $u$ is thus in the ``right part'' of $\vec{\calS}$.
    Suppose by contradiction that there is an oriented $\bar{x}$-inconsistent ELBE subtree $\vec{\calS'}$ of height $\nrevisions+1$ and call its first left node $p'$.
    Specifically, $\calL(\vec{\calS'})=\calL(\vec{\calS}) \cup \set{w} \setdiff \set{u}$ because, to be violated, the subtree constraint must include sufficiently many sibling pairs with $(1,0)$ values, and thus sufficiently many nodes with value $1$, which include $w$ and not $u$.
    Then, by our choice of $u$, $\smash{\sigma^{\calS'}_{\bar{x}}}(p) = \sigma^{\calS}_{\bar{x}}(p)=2^\nrevisions$.
    Thus by \Cref{claim: st facet - left part}, $p'=p$.
    This yields a contradiction to $\stage{w} \leqslant \stage{p}$ because $p$ must be the only left node at the earliest stage in $\vec{\calS'}$.
    
    \item If $\stage{w} > \stage{p}$ and $w \notin \descendant{p}$, we let $u \in \leftnodes{\vec{\calS}} \cap \descendant{p}$.
    Assume by contradiction that there is an oriented $(\nrevisions+1)$-height $\bar{x}$-inconsistent ELBE subtree $\vec{\calS'}$ with $p'$ being the left child of $\subroot{\calS'}$. 
    We must have $\calL(\vec{\calS'}) = \calL(\calS) \cup \set{w} \setdiff \set{u}$.
    Since $\stage{w} > \stage{p}$, $p$ is at the earliest stage in $\leftnodes{\vec{\calS'}}$, and thus $p'=p$.
    However, $\sigma^{\calS'}_{\bar{x}}(p) = \sigma^{\calS}_{\bar{x}}(p) - \bar{x}(u) = 2^\nrevisions-1$ as $u \in \leftnodes{\calS} \cap \descendant{p}$ and $w \notin \descendant{p}$, a contradiction.
    
    \item If $\stage{w} > \stage{p}$ and $w \in \descendant{p}$, 
    we choose $u$ such that $u \in \calL(\vec{\calS}) \setdiff (\descendant{p} \cup \{p\})$, \emph{i.e.}, $u$ is in the ``right part'' of $\vec{\calS}$.
    Define $\vec{\calS'}$ and $p'$ as previously. 
    Similarly, we must have $p'=p$.
    Hence, $\sigma_{\bar{x}}^{\calS'}(p') = \sigma_{\bar{x}}^{\calS}(p) + \bar{x}(w) =  2^{\nrevisions}+1$, which is also impossible according to \Cref{claim: st facet - left part}, yielding the desired contradiction.
\end{enumerate}
Thus, for any $w \notin \nodes(\calS) \setdiff \{\subroot{\calS}\}$ there is a pair $(u, v)$ such that $x^{u,v}_0 + e_w \in \revisableset{\nrevisions}$.
\end{proof}

\ClaimSTIntegralityPerturbDelta*
\begin{proof}
\label{proof: st integrality perturb delta}
We rephrase the claim as follows:
given a tree $\tree$ and a solution $y \in [0,1]^{\nodes \setdiff \{\root\}}$ with some fractional entry and $\delta_y(\root) \in (0,1)$, we can find $y^+ \neq y^- \in [0,1]^{\nodes \setdiff \{\root\}}$ such that $\delta_{y^\pm}(\root) = \delta_{y}(\root) \pm \varepsilon$ for sufficiently small $\varepsilon > 0$.
We use induction.

First, consider the base case $\nstages = 2$.
In this case, $\tree$ is composed of a root node and its children. 
Thus, the ELBE subtrees are just the root and a pair of its children.
Hence, $\delta_y(\root)$ captures the largest difference between the children of the root.
If $\delta_y(\root)$ is fractional, then we can pick a pair $(u,v)$ such that $|y(u) - y(v)| = \delta_y(\root)$, and at least one of $y(u)$ or $y(v)$ is fractional.
We only consider the case where the larger of the two is fractional, (\textit{i.e.}, we assume $y(u)$ is fractional and $y(u)>y(v)$), as the proof of the other case is similar. 
Then, for every node $v' \in \nodes_2$, $y(v') \leqslant y(u)$, otherwise if $y(v') > y(u)$ then $y(v') - y(v) > y(u) - y(v) = \delta_y(\root)$, which is a contradiction.
We set $\bar{\varepsilon}:= \min \{ 1-y(u), \min_{w \in \nodes_2: y(w) < y(u)} \{y(u) - y(w)\} \}$.
Next, we define $y^{\pm}$ such that $y^{\pm}(u') = y(u') \pm \varepsilon$ for $u' \in \nodes_2$ with $y(u')=y(u)$, and define $y^{\pm}(w)=y(w)$ for $w \in \nodes_2$ where $y(w) < y(u)$.
For any $\varepsilon$ positive such that $\varepsilon < \bar{\varepsilon}$, 
it follows directly that $\delta_{y^\pm}(\root) = \delta_y(\root) \pm \varepsilon$.

Next, we assume that this statement holds for any tree that has less than $\nstages$ stages, and prove that the result still holds for trees with $\nstages$ stages.
For such a tree $\tree$, let $y$ be a vector with at least one fractional entry.
Let $\calE$ denote the set of all tight pairs of second-stage nodes that determine the value of $\delta_y(\root)$, specifically,
\begin{equation} \label{eq: active edge set}
    \calE := \big\{ (u,v): u,v \in \nodes_2,\, \delta_y(\root) = |y(u) - y(v)| + \delta_y(u) + \delta_y(v) - 2 \big\}.
\end{equation}
Then, $H = (\nodes_2', \calE)$ forms a graph, where $\nodes_2' \subseteq \nodes_2$ is the set of nodes appearing in at least one tight pair.
As $\delta_y(\root) > 0$, we must have that $\delta_y(u) > 0$ for all $u \in \nodes_2'$ and $y(u) \neq y(v)$ for any $(u,v) \in \calE$.
We consider one specific pair $(u,v)$ and assume wlog that $y(u) > y(v)$.
Suppose $(u,v')\in \calE$ is another tight pair such that $v' \neq v$, we claim that $y(u) > y(v')$.
Otherwise, if $y(v) < y(u) \leqslant y(v')$, then we have
\begin{equation} \label{eq: proof step in subtree thm}
\begin{aligned}
\delta_y(\root)
& \geqslant
y(v') - y(v) + \delta_x(v') + \delta_y(v) - 2 \\
& = y(v') - y(u) + y(u) - y(v) + \delta_y(v') + 
\delta_y(v) - 2 \\
& = \big( y(v') - y(u) + \delta_x(v') - 2 \big) + \big( y(u) - y(v) + \delta_y(v)  - 2 \big) + 2 \\
& = 2 \delta_y(\root) - 2 \delta_y(u) + 2 \geqslant 2 \delta_y(\root),
\end{aligned}
\end{equation}
which implies $\delta_y(\root) \leqslant 0$, contradicting our assumption that $0 < \delta_y(\root) < 1$.
Similarly, if $y(u) < y(v)$, then we can show that $y(u) < y(v')$ for all $v' \neq v$ such that $\{u,v'\} \in \calE$.
Therefore, there are only two types of nodes in $H$: for each $u \in \nodes_2'$, either $y(u) < y(v)$ for all $v$ such that $(u,v) \in \calE$, or $y(u) > y(v)$ for all $v$ such that $(u,v) \in \calE$.
We color the first type of nodes red and the second type blue.
Then, $H$ is a bipartite graph as there is no edge between two red (\textit{resp.} blue) nodes since they all have smaller (\textit{resp.} larger) $y$-values than their neighbors.
For $u \in \nodes_2'$, if $y(u)=0$, then $u$ is a red node, and if $y(u)=1$, then $u$ is a blue node.

Since, $\delta_y(\root) < 1$, the existence of a node with $y$-value $0$ and $\delta_y$-value $1$ implies that no node with $y$-value $1$ and $\delta_y$-value $1$ can exist (and vice versa), so the two cases are mutually exclusive.
We therefore assume wlog that no node satisfies $y=0$ and $\delta_y=1$.
In this case, for every red node $u \in \nodes_2$, either $y(u) \in (0,1)$ or $y(u)=0 $ and $\delta_y(u) \in (0,1)$.
We build vectors $y^{\pm}$ as follows.
For $u$ such that $y(u)=0$ and $\delta_y(u) \in (0,1)$, we adjust the $y$-value on $\tree(u)$ such that $\delta_{y^\pm}(u) = \delta_y(u) \pm \varepsilon$ by induction.
For every red node $u$ for which $y(u) \in (0,1)$, we define $y^\pm(u) = y(u) \mp \varepsilon$.
We set $y^{\pm}(v)=y(v)$ for all blue nodes $v$ and all $v \in \nodes_2 \setdiff \nodes_2'$.
In the above construction, we choose $\varepsilon$ in $(0, \bar{\varepsilon})$ where we define 
\begin{equation}
\begin{aligned}
    \bar{\varepsilon}:= 
    \min\Bigg\{ \hat{\varepsilon}, & \min_{v \in \nodes_2: y(v) \in (0,1)} \big\{ \min \{y(v), 1-y(v)\} \big\}, \\ 
    & \min_{u,v \in \nodes_2: (u,v) \notin \calE} \Big\{ \big| \delta_y(\root)- \big( |y(u) - y(v)| + \delta_y(u) + \delta_y(v) - 2 \big) \big| \Big\} \Bigg\}.
\end{aligned}
\end{equation}
and 
$\hat{\varepsilon}$ is the minimum of the constants that are provided by the induction hypothesis for all $u$ such that $y(u)=0$ and $\delta_y(u) \in (0,1)$.

For each active pair $(u,v) \in \calE$, after the adjustment, the value $|y(u) - y(v)| + \delta_y(u) + \delta_y(v) - 2$ is perturbed by $\pm \varepsilon$ exactly.
This is because each pair $(u,v)$ contains a red node and a blue node, and the values of $y(\cdot)$ and $\delta_{y}(\cdot)$ are not changed for a blue node, whereas the contributions of $y(\cdot)$ or $\delta_{y}(\cdot)$ change by $\varepsilon$ for a red node. 
Therefore, we have $\delta_{y^\pm}(\root) = \delta_y(\root) \pm \varepsilon$.
\end{proof}

\LmCPFacetValid*
\begin{proof}
\label{proof: cp facet - valid}
Consider a binary vector $(\plan, r, \x)$ that satisfies the first three sets of constraints of \Cref{eq: cp}. 
Assume that it does not satisfy \Cref{eq: cp_facets}, then our goal is to show that $(\plan, r, \x) \notin P_{\text{CP}}^I$, \emph{i.e.}, the revision budget constraint for some scenario is violated.
Since $\calS$ is an ELBE subtree with height $\subheight{\calS}$, $\Delta_x^\calS \leqslant 2^{\subheight{\calS}} - 1$, where $\Delta_x^\calS$ is defined in \Cref{eq: x-inconsistent-value}.
We let $\zeta=\nrevisions - \sum_{w \leq \subroot{\calS}} r(w)$, representing the remaining balance of our revision budget after node $\subroot{\calS}$.
It cannot be the case that $\zeta \geqslant \subheight{\calS}$ as otherwise inequality \Cref{eq: cp_facets} would not be violated. 
Further, $\zeta < 0$ implies that $(\plan, r, \x) \notin P_{\text{CP}}^I$ immediately.
It is therefore sufficient to consider the case $0 \leqslant \zeta \leqslant \subheight{\calS}-1$.
The assumption that \Cref{eq: cp_facets} is violated implies that $\Delta_x^\calS > 2^{\subheight{\calS}} - \subheight{\calS} - 1 + \zeta$.
We claim that $\tree(\subroot{\calS})$ is not $\zeta$-revisable, and therefore the revision budget will exceed the balance on some scenario.
To prove the claim, we only need to find an inconsistent ELBE subtree with height $\zeta+1$.
When $\zeta = \subheight{\calS} - 1$, the claim holds because $\calS$ itself is an $\subheight{\calS}$-height inconsistent ELBE subtree.
For $0 \leq \zeta < \subheight{\calS} - 1$, there are $2^{\subheight{\calS} - (\zeta+1)} \geqslant \subheight{\calS} - \zeta$  many disjoint ELBE subtrees with height $\zeta+1$ inside of $\calS$.
We refer to them as $\{\calS_i\}_{i \in [\subheight{\calS}-\zeta]}$.
Hence, $\Delta_x^\calS \leqslant 2^{\subheight{\calS}} - 1 - \sum_{i } (2^{\zeta+2}-1 - \Delta_x^{\calS_i})$.
Since $\Delta_x^\calS > 2^{\subheight{\calS}}-1 - (\subheight{\calS}-\zeta)$, we have $\sum_{i} (2^{\zeta+2}-1-\Delta_x^{\calS_i}) < \subheight{\calS} - \zeta$. 
Thus, there is at least one $\calS_i$ such that $\Delta_x^{\calS_i}=2^{\zeta+2}-1$, which implies that it is $x$-inconsistent.
\end{proof}

\begin{lemma}
\label{lm: cp facet - project cp to path}
    Define formulation
    \begin{equation} \label{eq: path}
    \begin{aligned}
        &\textstyle \sum_{v \in P^*(\mu,\nu)} \rv{v} \geqslant \x{\mu} - \x{\nu}, \quad && \forall \mu, \nu \in \nodes: \stage{\mu}=\stage{\nu}, \\
        &\textstyle \sum_{v \in \scenario} \rv{v} \leqslant \nrevisions, \quad && \forall \scenario \in \scenarioset, \\
        &\x{v} \in \Bin,\;  \rv{v} \in \Bin, \quad && \forall v \in \nodes,  \\
    \end{aligned}
    \end{equation}
where $P^*(\mu,\nu):= \{v: v > \mu \join \nu, v \leq \mu, \nu\}$, \emph{i.e.}, $P^*(\mu,\nu)$ represent the nodes on the unique path connecting the two same-stage nodes $\mu$ and $\nu$ in the tree from which node $\mu \join \nu$ is removed.

Formulation \Cref{eq: path} is an MIP formulation for $\revisableset{\nrevisions}$, and for any $(r,x)$ satisfying \Cref{eq: path} there is a plan adjustment policy $\plan$ such that $(\plan, r, x) \in P_{\textnormal{CP}}^I$.
\end{lemma}

\begin{proof}
We first argue that solutions $\x \in \revisableset{\nrevisions}$ satisfy the constraints of \eqref{eq: path}.
Let $\x \in \revisableset{\nrevisions}$. 
Then,  by definition, there is a corresponding plan adjustment policy $\plan$ and a revision policy $r$.
Clearly, $(r,\x)$ satisfies the last two sets of inequalities of \Cref{eq: path}.
Consider $\mu$ and $\nu$ to be same-stage nodes. 
Then there is a unique path connecting $\mu$ and $\nu$, say $(v_1, v_2,\ldots, v_{2l+1})$, where $v_l = \mu \join \nu$.
Thus, $P^*(\mu,\nu)=(v_1,\ldots,v_{l-1}, v_{l+1}, \ldots, v_{2l+1})$.
It is clear that the first inequality in \eqref{eq: path} is satisfied for these $\mu$ and $\nu$ when $\x{\mu}=\x{\nu}$. 
When $\x{\mu} \neq \x{\nu}$, at least one revision must occur on $P^*(\mu,\nu)$, otherwise $\x{\mu}=\x{\nu}=\plan{\mu\join \nu, \stage{\mu}}$, showing that this inequality is also satisfied in that case.

We now show the converse.
Let $(r,\x)$ be a feasible solution of \Cref{eq: path}.
We next describe a procedure to construct $\plan$.
For a node $v$ such that $r(v)=1$, assume that each node $u\in\descendant{v}$ such that $r(u)=1$ has a valid $\plan(u,\cdot)$ defined. 
Then we find the longest path $(u_0=v, u_1, \ldots, u_k)$ for which $r(u_i)=0$ and $u_{i-1}=\parent{u_i}$ for all $i \in [k]$.
We define $\plan(v, \stage{v}+i)=x(u_i)$ for all $i \in \{0\}\cup[k]$ and $\plan(v, t) = 0$ for all $t > \stage{v} + k$.
We also use the above procedure to define the plan at the root node $\root$ (even when $r(\root)$ is not equal to $1$.)
Hence the variables $\plan$ are completely specified for each node such that its $r$-value is $1$ and for the root node.
Next, for each node $v$ such that $r(v)=0$, we find the closest ancestor that is revised (or is the root), say $w$, and define $\plan(v,t)=\plan(w,t)$ for all $t \in [\stage{v}:\nstages]$.

We next show that $\plan$ is consistent with $\x$ and $r$.
The consistency between $\plan$ and $r$ is clear because our construction step ensures that if $r(v)=0$ for some non-root node $v$, then $\plan(v,t)=\plan(\parent{v},t)$ for all $t\in [\stage{v}:\nstages]$.
To show consistency between $\plan$ and $x$, we need to show $\plan{v, \stage{v}} = \x{v}$ for all $v\in \nodes$.
If $r(v)=1$, or $v$ is the root, then $\plan{v, \stage{v}}=\x{v}$ holds true by construction.
If $r(v)=0$, assume that $\plan{v, \stage{v}} \neq \x{v}$ by contradiction.
Let $w$ be the closest ancestor of $v$ that is revised (or is the root). 
Thus $\plan{w, \stage{v}} \neq \x{v}$.
Hence, there is a longest non-revision path $(w, u_1, \ldots, u_k)$ such that $\plan{w,\stage{v}}=\x{u_j}$ where $\stage{u_j}=\stage{v}$ and $j \leqslant k$, which implies $\x{u_j} \neq \x{v}$.
However, it is clear that there is no revision in $P^*(v, u_j)$, which violates the first constraint of \Cref{eq: path} for $\mu=v$ and $\nu=u_j$. 
This is a contradiction to the fact that $(r,x)$ satisfies \Cref{eq: path}.
\end{proof}

\LmCPFacetConstructSolutions*
\begin{proof}
\label{proof: cp facet - construct solutions}
\begin{enumerate}[label=(\Roman*)]
\item 
The solution $(\hat{r}_J, \hat{x}_\theta)$ satisfies \Cref{eq: cp_facets} at equality by our definition.
We now show that it also satisfies \Cref{eq: path}.
Every same-stage pair with different $x$-values has a revision on the value-1 node, thus the first set of inequalities holds true.
Then, we check that the revision budget inequalities are satisfied.
Since a revision takes place either at nodes in $J$ or at a node in $\nodes(\calS) \setdiff \{\subroot{\calS}\}$, $\sum_{v \in \scenario} \hat{r}_J(v) \leqslant |J| + \subheight{\calS} = \nrevisions$ for any scenario $\scenario$.
Using \Cref{lm: cp facet - project cp to path}, we conclude that there exists some $\hat{\plan}$ such that $(\hat{\plan},\hat{r}_J,\hat{x}_\theta) \in F_{\textnormal{CP}}$.

\item  
By the above construction, \Cref{eq: cp_facets} holds at equality.
Next, to define the values of the other variables, we consider the set of common ancestors of $u$ and $v$ in $\nodes(\calS) \setdiff \{\subroot{\calS}\}$, which forms a path in $\calS$ that we  denote as $Q^{u,v}$.
For each node $\mu \in Q^{u,v}$ and its sibling pair $\nu$, set $\breve{r}_J(\mu)=1$, $\breve{r}_J(\nu)=0$.
Further, if $\mu$ is a left node (\emph{resp.} right node), set $\breve{x}^{u,v}_\theta(j)=0$ (\emph{resp.} $\breve{x}^{u,v}_\theta(j)=1$) and $\breve{r}_J(j)=0$ for all the other nodes on this stage that are not in $\calS$, \emph{i.e.}, $j \in \nodes_{\stage{\mu}} \setdiff \nodes(\calS)$.
Since nodes in $Q^{u,v}$ must be on different stages, the above definition will not lead to variables being assigned different values.
In addition, we set the other undefined pairs $\breve{r}_J(u')=\breve{r}_J(v')=1$ for all $(u',v') \in \sib{\calS} \setdiff \{(u,v)\}$ such that neither $u'$ nor $v'$ belongs to $Q^{u,v}$.
Then, set $\breve{r}_J(u)=\breve{r}_J(v)=0$.
For the rest of nodes $i$ that has no value defined, set $\breve{x}^{u,v}_\theta(i)=\theta$.

We first prove that the first set of inequalities in \Cref{eq: path} holds true. 
Consider any same-stage node pair $(o,o')$ with $\breve{x}^{u,v}(o) \neq \breve{x}^{u,v}(o')$, and note that such pair must contain one node in $\calS$, say $o$.
If $\breve{r}_J(o)=1$, then the first set of inequality holds true clearly.
If $\breve{r}_J(o) = 0$, then $o$ must be a sibling of a node in $Q^{u,v}$ and by definition $o$ has the same $\breve{x}^{u,v}$-value with the rest of nodes at stage $\stage{o}$ that have $\breve{r}_J$-value $0$, thus $o'$ is also in $\calS$ and $\breve{r}_J(o')=1$.
To show the revision budget constraints hold, we claim that any path excluding  $\subroot{\calS}$ in $\calS$ has at most $\subheight{\calS}-1$ revisions.
This is because, if this path contains $u$ or $v$, clearly it takes no more than $\subheight{\calS}-1$ revisions as $\breve{r}_J(u)=\breve{r}_J(v)=0$. 
If this path does not contain $u$ or $v$, it must contain some node $\nu$ such that the sibling of $\nu$ is an ancestor of $u$ and $v$. 
By our definition, $\breve{r}_J(\nu)=0$ and hence it also takes at most $\subheight{\calS}-1$ revisions.
Thus, $\sum_{v\in \scenario} \breve{r}_J(\scenario) \leqslant |J| + \subheight{\calS} - 1 = \nrevisions$ for any $\scenario\in \scenarioset$.
Therefore, we can extend it to $(\breve{\plan}, \breve{r}_J, \breve{x}^{u,v}_\theta) \in F_{\textnormal{CP}}$.

\item We show how to select $(\dot{\plan}, \dot{r}, \dot{x})$ depending on $i$.
Denote by $p$ and $q$ the left and right children of $\subroot{\calS}$ in $\calS$, respectively. Since $i \notin \ancestorset$, $i$ cannot be an ancestor of both $p$ and $q$.
Also, $i$ cannot be a descendant of both $p$ and $q$.
We apply \ref{cp_facet_pointI} to construct the solution as follows.
If $i$ is an ancestor or descendant of $p$ (\emph{resp.} $q$) but not in $\nodes(\calS)$, then we will choose $(\dot{\plan}, \dot{r}, \dot{x}) = (\hat{\plan},  \hat{r}_{J}, \hat{x}_\theta)$ with $\theta=1$ (\emph{resp.} $\theta=0$).
If $i$ is a non-root node of $\calS$, then if $i$ is a left node (\emph{resp.} right node), we choose $(\dot{\plan}, \dot{r}, \dot{x}) = (\hat{\plan}, \hat{r}_{J}, \hat{x}_\theta)$ with $\theta=1$ (\emph{resp.} $\theta=0$).
If $i$ is neither an ancestor nor descendant for both $p$ and $q$, we choose $(\dot{\plan}, \dot{r}, \dot{x}) = (\hat{\plan}, \hat{r}_{J}, \hat{x}_0)$ (here, setting $(\dot{\plan}, \dot{r}, \dot{x}) = (\hat{\plan}, \hat{r}_{J}, \hat{x}_1)$ would also work). 
For all the above cases, the solution satisfies $\dot{r}(i)=0$ and guarantees that each scenario containing $i$ has fewer than $\nrevisions$ revisions.

\item If $i\notin \ancestorset$, then we have shown in \ref{cp_facet_pointIII} that there is some $\theta \in \{0,1\}$ such that $(\hat{\plan}, \hat{r}_J + e_i, \hat{x}_\theta) \in F_{\textnormal{CP}}$.
We will use it as our solution $(\ddot{\plan}, \ddot{r}, \ddot{x})$. 
For the case where $i \leq \subroot{\calS}$.
We pick $J \subset \ancestorset$ such that $i \in J$ and $|J|=\nrevisions-\subheight{\calS}+1$.
Since $\nrevisions - \subheight{\calS} +1 \geqslant 1$, such $J$ exists.
Setting $(\ddot{\plan}, \ddot{r}, \ddot{x}) = (\breve{\plan}, \breve{r}_{J}, \breve{x}^{p,q}_0)$ by applying \ref{cp_facet_pointII} proves the claim.

\end{enumerate}
\end{proof}

\section{Subtree constraint separation algorithm} \label{sec: constraint generation algorithm}

To separate subtree constraints, we implement the recurrence \Cref{eq: DP relation} in \Cref{alg: subtree separation alg}, where we also implement a mechanism to record the corresponding ELBE subtree.
\begin{algorithm}[htbp]
\caption{A separation algorithm for subtree constraints}
\label{alg: subtree separation alg}
\begin{algorithmic}[1]
\Require A solution $\bar{x} \in [0,1]^{\nodes}$, a scenario tree $\tree$, and an integer $\nrevisions$.
\Ensure An oriented $\bar{x}$-inconsistent ELBE subtree with height $\nrevisions+1$ if exists.
\State $\Delta(v,0) \gets 0,  \quad \forall v \in \nodes$
\State $\vec{\calS}(v,0) \gets \operatorname{OriELBESubtree}(), \quad \forall v \in \nodes_\nstages$
\State $ l \gets \nstages - 1$
\While{$l \geqslant 1$}
\For{$v \in \nodes_l$}
\For{$h \in \calH(v, \nrevisions)$}
    \For{ $(p, q) \in \{(p, q):\, \stage{p}=\stage{q},\, p \join q = v \}$}
    \If{  $\Delta(p, h-1) + \Delta(q, h-1) + (\bar{x}(p) - \bar{x}(q)) > \Delta(v, h)$}
        \State $\Delta(v, h) \gets \Delta(p, h-1) + \Delta(q, h-1) + (\bar{x}(p) - \bar{x}(q))$
        \State $\vec{\calS}(v, h) \gets \textsc{Join}((p,q), \vec{\calS}(p, h-1), \vec{\calS}(q, h-1))\}$
    \EndIf
    \EndFor
    \If{$h < \nstages - l$}
    \For{$u \in \children{v}$}
    \If{$\Delta(u, h) > \Delta(v, h)$}
	    \State $\Delta(v,h) \gets \Delta(u, h)$
        \State $\vec{\calS}(v, h) \gets \vec{\calS}(u, h)$
    \EndIf
    \EndFor
    \EndIf
\EndFor
\EndFor
\State $l \gets l - 1$
\EndWhile
\If{$\Delta(\root, \nrevisions+1) > 2^{\nrevisions+1} - 2$}
    \State \Return $\vec{\calS}(\root, \nrevisions+1)$    
\Else
    \State \Return
\EndIf
\end{algorithmic}
\end{algorithm}

In the \DP, $\vec{\calS}(v,h)$ denotes one ELBE subtree achieving the largest inconsistency value among all subtrees rooted at or below node $v$ and with height $h$. 
During the recursion, when a pair $(p,q)$ of descendant nodes is selected to join, the function $\textsc{Join}( (p,q),\vec{\calS_1},\vec{\calS_2} )$, where we require $\vec{\calS}_1$ and $\vec{\calS}_2$ to be under $p$ and $q$ respectively, creates a new ELBE subtree with the following procedure:
first, select $p \join q$ as the root;
second, the original root nodes of $\vec{\calS}_1$ and $\vec{\calS}_2$ are replaced by the nodes $p$ and $q$, respectively;
and third, the nodes $p$ and $q$ are set to be the children of $p \join q$, and thus $\vec{\calS}_1$ and $\vec{\calS}_2$ are combined into the new tree.

\paragraph{Separating binary solutions.}
\label{paragraph: lazy constraint st}
For an integer solution $\bar{x}$, \Cref{alg: subtree separation alg} can be simplified and be made more efficient.
We use $\vec{\calS}(v)$ to denote one oriented $\bar{x}$-inconsistent ELBE subtree with the largest height that is below $v$.
We similarly traverse the scenario tree bottom-up and construct $\vec{\calS}(v)$ for each node $v$.
The algorithm identifies the maximum height of $\vec{\calS}(u)$ for all $u \in \children{v}$, denoted as $h_0$.
Then, among all children $u \in \children{v}$ such that $\vec{\calS}(u)$ has height $h_0$, we seek a pair $u_1, u_2$ such that $\bar{x}(u_1) \neq \bar{x}(u_2)$.
If such a pair exists, then we define $\vec{\calS}(v):=\textsc{Join}((u_1, u_2), \vec{\calS}(u_1), \vec{\calS}(u_2))$, which has height $h_0 + 1$.
If no such pair is found among the children of $v$, the algorithm recursively searches one level deeper, examining grandchildren of $v$ whose associated subtrees have height $h_0$. 
This step is necessary as the $\bar{x}$-inconsistent subtree might contain these grandchildren without containing their parents.
This procedure continues until either a pair with different $\bar{x}$-values is identified, allowing $\vec{\calS}(v)$ to be expanded by one level, or no such pair exists among descendants at levels $\stage{v}+1, \ldots, \nstages-h_0$. 
In the latter case, the algorithm propagates a tallest subtree unchanged by setting $\vec{\calS}(v):=\vec{\calS}(u)$ for one $u \in \children{v}$ such that the height of $\vec{\calS}(u)$ is $h_0$.
Finally, if $\vec{\calS}(\root)$ has height exceeding $\nrevisions$, a violated subtree constraint is identified.
This procedure has running time $\order{|\nodes| \nstages}$, which improves upon the $\order{|\nodes|^2 \nrevisions}$ complexity of \Cref{alg: subtree separation alg}.
\end{appendices}

\end{document}